\documentclass{amsart}
\usepackage{amsthm, amssymb, amsmath, dsfont, mathrsfs, color, soul, float, upgreek}
\usepackage{stmaryrd} % \llbracket and \rrbracket
\usepackage[main=english,brazilian]{babel}
\usepackage[utf8]{inputenc}
\usepackage[T1]{fontenc}
\usepackage{boxedminipage}
\usepackage{enumerate}
\usepackage{mathtools}
\usepackage{graphicx}
\usepackage{bbm}
\usepackage{cite}      %package that allows multiple citations inside \cite
\usepackage{comment}

\usepackage{tikz}
\usetikzlibrary{patterns}
\usetikzlibrary{calc}
\usetikzlibrary{intersections}
\usetikzlibrary{through}
\usetikzlibrary{backgrounds}
\usepackage[hidelinks,bookmarksnumbered]{hyperref}
\hypersetup{pdfstartview=FitH}

\newcommand\encircle[1]{%
  \tikz[baseline=(X.base)] 
    \node (X) [draw, shape=circle, inner sep=-1.9] {\strut \scriptsize #1};}

\newcommand{\stateh}{\encircle{\textnormal{h}}}
\newcommand{\statei}{\encircle{\textnormal{i}}}

\newcommand\tinyencircle[1]{%
  \tikz[baseline=(X.base)] 
    \node (X) [draw, shape=circle, inner sep=-2.6] {\strut \tiny #1};}

\newcommand{\tinystateh}{\tinyencircle{\textnormal{h}}}
\newcommand{\tinystatei}{\tinyencircle{\textnormal{i}}}

\newcommand{\dist}{\ensuremath{\mathop{\mathrm{dist}}\nolimits}}
\newcommand{\per}{\mathop{\mathrm{per}}\nolimits}
\newcommand{\sfv}{\ensuremath{\mathsf{v}}}
\newcommand{\plow}{\underline{\smash{p}}}
\newcommand{\Tlow}{\underline{\smash{\Theta}}}
\newcommand{\alphav}{\alpha_{\mathsf v}}

\newcommand{\cB}{\ensuremath{\mathcal{B}}}

\newcommand{\cJ}{\ensuremath{\mathcal{J}}}

\newcommand{\cQ}{\ensuremath{\mathcal{Q}}}

\newcommand{\EE}{\ensuremath{\mathbb{E}}}

\newcommand{\NN}{\ensuremath{\mathbb{N}}}
\newcommand{\PP}{\ensuremath{\mathbb{P}}}

\newcommand{\RR}{\ensuremath{\mathbb{R}}}
\newcommand{\ZZ}{\ensuremath{\mathbb{Z}}}

\newtheorem{definition}{Definition}[section]
\newtheorem{remark}{Remark}[section]
\newtheorem{theorem}{Theorem}[section]
\newtheorem*{theorem*}{Theorem}
\newtheorem{lemma}{Lemma}[section]
\newtheorem{proposition}[lemma]{Proposition}
\newtheorem{corollary}[lemma]{Corollary}
\theoremstyle{definition}
\newtheorem*{definition*}{Definition}
\newtheorem{claim}{Claim}

\let\oldtocsection=\tocsection
\let\oldtocsubsection=\tocsubsection
\let\oldtocsubsubsection=\tocsubsubsection

\renewcommand{\tocsection}[2]{\hspace{0em}\oldtocsection{#1}{#2}}
\renewcommand{\tocsubsection}[2]{\hspace{1em}\oldtocsubsection{#1}{#2}}
\renewcommand{\tocsubsubsection}[2]{\hspace{2em}\oldtocsubsubsection{#1}{#2}}

\begin{document}

%%%%%%%%%% Opening
\title{Contact process on interchange process}

\author{Marcelo Hil\'ario}
\address[M.~Hilário]{ICEx, Universidade Federal de Minas Gerais, Brazil}
\email{mhilario@mat.ufmg.br}

\author{Daniel Ungaretti}
\address[D.~Ungaretti]{Instituto de Matem\'atica, Universidade Federal
do Rio de Janeiro, Brazil}
\email{daniel@im.ufrj.br}

\author{Daniel Valesin}
\address[D.~Valesin]{Department of Statistics, University of Warwick, United Kingdom}
\email{daniel.valesin@warwick.ac.uk}

\author{Maria Eul\'alia Vares}
\address[M. E.~Vares]{Instituto de Matem\'atica, Universidade Federal do Rio de Janeiro, Brazil}
\email{eulalia@im.ufrj.br}

\begin{abstract}
We introduce a model of epidemics among moving particles on any locally finite graph. 
At any time, each vertex is empty, occupied by a healthy particle, or occupied by an infected particle.
Infected particles recover at rate $1$ and transmit the infection to healthy particles at neighboring vertices at rate $\lambda$. 
In addition, particles perform an interchange process with rate $\mathsf{v}$, that is, the states of adjacent vertices are swapped independently at rate $\mathsf{v}$, allowing the infection to spread also through the movement of infected particles.
On $\mathbb{Z}^d$, we start with a single infected particle at the origin and with all the other vertices independently occupied by a healthy particle with probability $p$ or empty with probability $1-p$. 
We define $\lambda_c(\mathsf{v}, p)$ as the threshold value for $\lambda$ above which the infection persists with positive probability and analyze its asymptotic behavior as $\mathsf{v} \to \infty$ for fixed $p$.
\end{abstract}

\maketitle

\tableofcontents

\section{Introduction}
\subsection{Model}
We introduce the \emph{interchange-and-contact process} as a model for the spread of an infection among a moving population.
This continuous-time interacting particle system is informally described as follows. 
At any point in time, each site of~$\mathbb Z^d$ (with~$d \ge 1$) can be in one of three states: $0$ (vacant),~$\stateh$ (occupied by a healthy particle) and~$\statei$ (occupied by an infected particle). The dynamics has three rules:
\begin{itemize}
\item infected particles recover ($\statei \to \stateh$) at rate $1$;
\item healthy particles become infected ($\stateh \to \statei$) at rate $\lambda$ times the number of infected neighbors;
\item for each edge~$e$ of~$\mathbb Z^d$, the states of the sites to which~$e$ is incident are swapped at rate $\mathsf{v}$.
\end{itemize}

\iffalse
It is a continuous-time interacting particle system $(\zeta_t)_{t \ge 0}$ taking values in the state space $\{0,\stateh, \statei\}^{\mathbb{Z}^d}$, with $d\geq 2$.
Here $0$, $\statei$ and $\stateh$ are interpreted respectively as vacant, occupied by an infected individual and occupied by a healthy individual.
The process starts from a random configuration~$\zeta_0$ with~$\zeta_0(o)=\statei$ and independently, for each~$x \in \mathbb Z^d \backslash \{o\}$, $\zeta_0(x)=\stateh$ with probability~$p$ and~$\zeta_0(x)=0$ with probability~$1-p$.
For fixed~$\lambda > 0$,~$\sfv > 0$ and~$p \in (0,1]$, the process evolve as follows:
\begin{itemize}
\item infected individuals recover ($\statei \to \stateh$) at rate $1$;
\item healthy individuals become infected ($\stateh \to \statei$) at rate $\lambda$ times the number of infect neighbors;
\item independently at rate $\mathsf{v}$, the states of each edge are swapped.
\end{itemize}
\fi

We write~$(\zeta_t)_{t \ge 0}$ for an interchange-and-contact process on $\mathbb{Z}^d$ with \emph{infection rate} $\lambda$ and \emph{interchange rate} $\sfv$. The name `interchange-and-contact process' is explained by the following two points:
\begin{itemize}
\item[-] \emph{Interchange:} For~$\zeta \in \{0,\stateh,\statei\}^{\mathbb Z^d}$, define~$\xi^\zeta \in \{0,1\}^{\mathbb Z^d}$ by
\[\xi^\zeta(x) = \begin{cases} 1&\text{if } \zeta(x) \in \{\stateh,\statei\};\\ 0&\text{if }\zeta(x)=0 \end{cases}\]
Then, the process~$(\xi^{\zeta_t})_{t \ge 0}$ is an \emph{interchange process} (also known as \emph{stirring process}): sites can be either vacant (state 0) or occupied (state 1), and the dynamics is governed by the third rule in the list above. 
(Depending on the point of view, this process could also be regarded as an \emph{exclusion process}, but we will not adopt this perspective, because we would like to have individual particles performing random walks on~$\mathbb Z^d$, as in Definition~\ref{def_ctrw} below). 
\item[-] \emph{Contact:} Since particles are never created or destroyed by the interchange-and-contact dynamics, the subset~$\Omega_{\mathrm{full}}:=\{\stateh,\statei\}^{\mathbb Z^d}$ of the state space~$\Omega:=\{0,\stateh,\statei\}^{\mathbb Z^d}$ is left invariant. 
For~$\zeta \in \Omega_{\mathrm{full}}$, define~$\pi^\zeta \in \{0,1\}^{\mathbb Z^d}$ by
\[
\pi^\zeta(x) = \begin{cases} 0&\text{if } \zeta(x) = \stateh;\\ 1&\text{if } \zeta(x)=\statei. \end{cases}
\]
If the parameter~$\sfv$ is zero, then the process~$(\pi^{\zeta_t})_{t \ge 0}$ reduces to the Harris contact process. 
\end{itemize}
An exposition on the contact process can be found in~\cite{Lig13}. 
For now, let us only recall that it undergoes a phase transition: there exists~$\lambda_c^{\mathrm{CP}} \in (0,\infty)$ such that, if the process starts from finitely many infections, then the infection goes extinct almost surely if and only if~$\lambda \le \lambda_c^{\mathrm{CP}}$. 

\subsection{Background}
The case where the process evolves on~$\Omega_{\mathrm{full}}$, but~$\sfv$ is allowed to be positive, also corresponds to an existing model in the literature, called the \emph{contact process with stirring}, which we now briefly survey.

In~\cite{de1986reaction}, De Masi, Ferrari and Lebowitz studied the effect of introducing a stirring mechanism on spin systems governed by Glauber-type dynamics. 
They proved that, as the rate of stirring is taken to infinity, the system converges to a solution of an associated reaction-diffusion equation. 
%In this context, it is reasonable to expect that threshold parameter values converge to `mean-field' values tied to the limiting equation.

The contact process with stirring was introduced by Durrett and Neuhauser in~\cite{durrett1994particle}.  Let~$\lambda_c^{\mathrm{CPS}}(\sfv)$ denote the supremum of the values of~$\lambda$ for which, starting from finitely many infected particles, and evolving  with infection rate~$\lambda$ and interchange rate~$\sfv$, the process goes extinct almost surely.
In~\cite{durrett1994particle} it is proved that
\begin{equation}\label{eq_first_convergence_lambda}\lim_{\sfv \to\infty} \lambda_c^{\mathrm{CPS}}(\sfv) = \frac{1}{2d}.\end{equation}
This is to be expected: the associated mean-field setting is a genealogical process in which each infection is regarded as an individual entity in a population where, independently, entities die with rate~1 and give birth to a new entity with rate~$2d\lambda$.
The associated threshold value of~$\lambda$ is then~$1/(2d)$.

Allowing for sites to be vacant, as we do for the interchange-and-contact process, introduces a very significant layer of complexity to the model.
The contact process dynamics is sensitive to the spatial inhomogeneities in the medium, and even if we were given which sites contain infected particles at a given time, fully describing the system would require the knowledge of which sites in $\mathbb{Z}^d$ were occupied or vacant throughout its prior evolution.
This makes the interchange-and-contact process more akin to models of contact process on dynamic random environments, in the spirit of the works of Broman~\cite{broman2007stochastic}, Steif and Warfheimer~\cite{steif2008critical}, Remenik~\cite{remenik2008contact} and Linker and Remenik~\cite{LR}. The latter studies a \emph{contact process on dynamical bond percolation}, defined as the classical contact process on~$\mathbb{Z}^d$ (with no motion of the infection), except that edges of~$\mathbb Z^d$ can be open or closed for the transmission of the infection. 
Edges evolve as independent two-state Markov chains, that jump from closed to open with rate~$p \sfv$, and from open to closed with rate~$(1-p)\sfv$, where~$p \in (0,1]$.

A critical threshold~$\lambda^{\mathrm{CPDP}}_c(p,\sfv)$ can be defined for the contact process on dynamical percolation, similarly to that of the previously discussed models say, using the process started from a single infection at the origin, and the environment in equilibrium  (though it turns out that the initial configuration is not important, as long as the initial set of infected sites is non-empty and finite).
Among several other results, Linker and Remenik proved that
\begin{equation}\label{eq_second_convergence_lambda}
\lim_{\sfv \to\infty} \lambda^{\mathrm{CPDP}}_c(p,\sfv) = \tfrac{1}{p}\lambda_c^{\mathrm{CP}}.
\end{equation}
This is justified by the observation that, when~$\sfv$ is very large, the edge dynamics mixes much quicker than the evolution of the contact process, so it is almost as if each time an edge were used by the infection, its state could be resampled independently of everything else, with probability~$p$ of being open and~$1-p$ of being closed. 
This amounts to a thinning with retention density~$p$ of the infection parameter. It should also be mentioned that more general environment dynamics have been considered by Seiler and Sturm in~\cite{seiler2023contact}.

\subsection{Main result}
We consider the interchange-and-contact process~$(\zeta_t)_{t \ge 0}$ with parameters~$\lambda$ and~$\sfv$. We take the initial configuration~$\zeta_0$ as the random configuration with
\begin{equation}\label{eq_how_it_starts}
    \zeta(0)=\statei \quad \text{and} \quad \zeta(x)=\begin{cases}
        \stateh &\text{with probability }p;\\0&\text{with probability }1-p,
    \end{cases} \quad \text{independently for }x \in \mathbb{Z}^d \backslash \{0\}.
\end{equation}
This is a natural choice, as the product Bernoulli measure is stationary for the interchange dynamics; we only perturb it at the origin to ensure that there is an infection at the start. Denoting by~$\mathbb P_{\lambda,\sfv,p}$ a probability measure under which this process is defined, the probability of survival and the critical infection threshold for survival are defined as
\begin{align*}
\uptheta(\lambda,\sfv,p)
    &:=\mathbb P_{\lambda,\sfv,p}\big(\text{for all $t$ there exists $x$ such that }\zeta_t(x)=\statei\big),\\
\lambda_c(\sfv,p)
    &:= \inf\{\lambda > 0: \uptheta(\lambda,\sfv,p) > 0 \},
\end{align*}
respectively. We can now state our main result, a convergence result as in~\eqref{eq_first_convergence_lambda} for the case $p \in (0,1)$.
\begin{theorem}
	\label{thm_main}
    For any~$p \in (0,1)$, we have
\begin{equation*}
    \lim_{\sfv \to \infty} \lambda_c(\sfv,p) = \frac{1}{2dp}.
\end{equation*}
\end{theorem}
%For $p=1$ the corresponding result reduces to \eqref{eq_first_convergence_lambda}, which has been proved in \cite{durrett1994particle}.
Interestingly, Theorem \ref{thm_main} incorporates both phenomena from the convergences in~\eqref{eq_first_convergence_lambda} and~\eqref{eq_second_convergence_lambda}, namely, the appearance of the mean-field threshold rate and the thinning of the infection parameter, respectively. Here, the thinning is due to a proportion~$1-p$ of the transmissions being lost due to targeting vacant sites.

In Section~\ref{ss_ideas}, we discuss the technical challenges involved in establishing this result. We then discuss our methods of proof, which in broad terms involve splitting the convergence into two regimes (extinction and survival), according to values of~$\lambda$ and~$p$ that are kept fixed as~$\sfv$ is taken to infinity.
Both regimes are analyzed through renormalization techniques.

Although our focus on this paper is exclusively the limit as~$\sfv \to \infty$, many other directions of investigation may naturally be considered for this model. 
To mention one of them, in analogy with~\cite{LR}, it would be interesting to study whether the model exhibits \emph{immunity}, meaning that there are values of~$p$ and~$\sfv$ for which the infection goes extinct almost surely \emph{regardless of the value of $\lambda$}.
%It is reasonable to expect that immunity does hold -- for instance, if~$p$ and~$\sfv$ are both very small, and then the infection can only spread among small islands of particles, which very rarely communicate with each other.

\subsection{Motivation and related works}
Mobility of agents is a desirable feature in models of growth and epidemics, and several works have addressed this feature in the literature. 
For models in which the agents move as independent random walks and transmit an infection, notable contributions include the works of Kesten and Sidoravicius~\cite{
kesten2005spread,kesten2006phase,kesten2008shape}, B\'erard and Ram\'irez \cite{Berard2016front}, Baldasso and Stauffer~\cite{baldasso2022local,baldasso2023local}, and Dauvergne and Sly~\cite{dauvergne2022sir,dauvergne2023spread}. 
%The frog model (see the survey by Popov~\cite{popov2003frogs}) can also be seen as a model in which an infection is spread by mobile agents.

Models in which the motion of the infection-spreading agents is not independent have also been considered. 
Infected particles move as a zero-range process in a work by Baldasso and Teixeira~\cite{baldasso2020spread}, and as an exclusion process in a work by Jara, Moreno and Ram\'irez~\cite{jara2008front}. 
The latter model shares only superficial similarities with ours since there are no recoveries, and the mechanism for spreading the infection involves the jumps in the exclusion process.

As mentioned earlier, the interchange-and-contact process may be regarded as the contact process on a dynamical random environment. 
The contact process on both static and dynamic random environments has been a very active topic of research over the last two decades.
In the static setting, it has been shown that degree inhomogeneities in the graph gives rise to a very rich behavior; see for instance~\cite{chatterjee2009contact, mountford2013metastable, bhamidi2021survival}, and the recent survey~\cite{valesin2024contact}. 
Introducing dynamics in the environment, raises the question of whether the effects of inhomogeneity persist, alongside with other interesting lines of investigation; see for instance~\cite{jacob2017contact, jacob2025contact, jacob2024metastability, cardona2024contact, leite2024contact, schapira2025contact, fernley2025contact}.

Concerning the convergence \eqref{eq_first_convergence_lambda} for the contact process with fast stirring, more refined results have  been obtained.
In~\cite{konno1995asymptotic}, Konno proved that
\[
0 < \liminf_{\sfv \to \infty} \frac{\lambda_c^{\mathrm{CPS}}(\sfv)-\frac{1}{2d}}{f(\sfv)} \le \limsup_{\sfv \to \infty} \frac{\lambda_c^{\mathrm{CPS}}(\sfv)-\frac{1}{2d}}{f(\sfv)} < \infty,
\]
where~$f(\sfv) = \sfv^{-1}$ if~$d \ge 3$,~$f(\sfv)=\log(\sfv)\sfv^{-1}$ if~$d=2$, and~$f(\sfv)=\sfv^{-1/3}$ if~$d = 1$.
For~$d \ge 3$, more is known: putting together the main results of Katori~\cite{katori1994rigorous} and Berezin and Mytnik~\cite{berezin2014asymptotic}, it holds that~$\lim_{\sfv \to \infty} \sfv \cdot (\lambda_c^{\mathrm{CPS}}(\sfv) - \tfrac{1}{2d}) = (G(0,0)-1)/(2d)$, where~$G(0,0)$ is the Green function of discrete-time simple random walk on~$\mathbb Z^d$. 
For~$d=2$, results in the same spirit are available, albeit not achieving precision down to the limiting constant, in~\cite{berezin2014asymptotic} and~\cite{levitvalesin}. 
It is an interesting line of research to obtain refinements of this kind for the convergence given in Theorem \ref{thm_main}.

\subsection{Ideas of proof}
\label{ss_ideas}

To prove Theorem \ref{thm_main} we will establish separately the following:
\begin{align}
    \label{eq_main_ext}
    \text{for all } \lambda, p \text{ with }2dp\lambda < 1, \text{ there exists }\sfv_0 > 0 \text{ such that } \uptheta(\lambda,\sfv,p)=0 \text{ for all }\sfv \ge \sfv_0;\\
    \label{eq_main_surv}
    \text{for all } \lambda, p \text{ with }2dp\lambda > 1, \text{ there exists }\sfv_1 > 0 \text{ such that } \uptheta(\lambda,\sfv,p)>0 \text{ for all }\sfv \ge \sfv_1.
\end{align}
The proof of these two points share broad similarities: both begin with a microscopic analysis and proceed to a renormalization scheme, which employs decoupling tools.

The microscopic analysis assumes for the most part that the environment of particles in which the infection spreads is close to equilibrium (product Bernoulli measure with density~$p$), as should be the case when the process starts.
It then exploits the assumption on~$\lambda$ and~$p$ to establish that the infection behaves subcritically in the case of~\eqref{eq_main_ext} and supercritically in the case of~\eqref{eq_main_surv}.

The guiding principle for either direction is that as~$\sfv$ goes to infinity, the set of infected particles behaves similarly to a branching random walk with death rate~1 and birth rate~$2d\lambda p$, at least while there are not too many infections. 
When there are too many infected particles, one observes \emph{collisions}, that is, transmission attempts towards particles that are already infected.
This makes the approximation by branching random walks inaccurate.  

Apart from the occurrence of collisions, the heterogeneity in the environment of particles is another important factor that contributes to the inaccuracy of the branching random walk approximation.
As the process evolves and additional information on the environment is revealed, one may find regions where the density deviates significantly from equilibrium. 
Renormalization comes in as a tool to establish that these regions are sufficiently rare to be neglected.

Our renormalization approach follows the standard framework of tiling space-time into boxes, classifying each as ``good'' or ``bad'' based on the behavior of the process within them, and iteratively coarse-graining to construct higher-scale boxes, which are similarly classified. 
The construction is designed such that the probability of a box being bad decreases rapidly with increasing scales, a property established through a recursive argument.

At the bottom scale, an upper bound on the probability of a box being bad is obtained using the microscopic analysis mentioned earlier. 
For higher scales, the probability of observing bad boxes is controlled by the likelihood of encountering a pair of bad boxes at the preceding scale. 
A mild decoupling estimate enables us to demonstrate that the process behaves approximately independently in boxes that are sufficiently separated in space and time. 
This yields a contracting sequence of probabilities for bad boxes across scales.

This decoupling estimate is a key ingredient in our analysis deserving further discussion.
Given the oriented nature of the model (due to the time component), it is important to distinguish between ``horizontal decoupling'' (between pairs of boxes that are well-separated in space, but possibly not in time) and ``vertical decoupling'' (distant in time, possibly not in space).

Horizontal decoupling in our setting follows from the fact that both particles and infection cannot traverse the distance between well-separated boxes within the relevant time frame.
However, as the interchange rate~$\sfv$ goes to infinity, this becomes a very delicate requirement, imposing a careful choice for the scale progression used in the renormalization.
Having taken care of this difficulty, the horizontal decoupling can be obtained using standard large deviations bounds on the speed of random walks and spreading processes.

In contrast, to derive a useful vertical decoupling is substantially more complex.
To address this, we develop a refined version of the \emph{sprinkling} procedure by Baldasso and Teixeira (Theorem 1.5 in~\cite{BT}). 
It consists in randomly introducing  particles into the system across successive scales that mitigates dependencies, thereby facilitating decoupling.
We needed to develop a subtle improvement that allows for deterministic initial states, rather than random and stationary, as in~\cite{BT}.
See Lemma~\ref{lem_coupling_rate_one} below, and its proof in the appendix. 
A similar refinement of the decoupling of~\cite{BT}, allowing for deterministic initial configurations, has recently been developed by Conchon-Kerjan, Kious and Rodriguez in~\cite{kious2024sharp}.

Although the discussion in the previous few paragraphs refers to both the extinction and the survival regimes (\eqref{eq_main_ext} and~\eqref{eq_main_surv}, respectively), our treatments of these regimes are largely distinct, and we discuss them separately now.
\smallskip

\textbf{Extinction: microscopic analysis.} 
This item is handled in Section~\ref{s_micro_ext}. 
As discussed before, we use the term `collision' to refer to attempts to infect already-infected particles, which cause the infection dynamics in our model to deviate from that of a branching random walk.
In the context of proving extinction, this  does not pose any concern: collisions can only contribute to extinction, so they may be ignored. 
This substantially simplifies the analysis in this regime.

Assume that~$2dp\lambda < 1$,~$\sfv$ is large, and the process starts as in~\eqref{eq_how_it_starts}. Let~$\sigma_1 < \sigma_2 < \cdots$ denote the times at which the number of infected vertices changes (it necessarily increases or decreases by one unit each time, and is absorbed at zero). Let~$M_n$ denote the number of infected particles at time~$\sigma_n$. 
Roughly speaking, we argue that, unless the environment surrounding the infected particles exhibits an atypically high density above $p$, then~$(M_n)_{n \ge 1}$ is stochastically dominated from above by a birth-and-death chain biased towards zero.
To prove this, we use the fact that particle motion occurs in a much faster time-scale than epidemics events (transmissions and recoveries); hence, in between the times~$(\sigma_n)$, particles are well-mixed, so when transmissions attempt take place, their target location is approximately in equilibrium (that is, vacant with probability~$1-p$ and particle with probability~$p$).

We run the process for time~$\log^3(\sfv)$, which is sufficient for the dominating birth-and-death chain to be absorbed, within a spatial box with radius~$\sqrt{\sfv}\log^4(\sfv)$.
This radius exceeds the distance across which the infection can propagate during the~$\log^3(\sfv)$ time horizon.
It is also an adequate size for ensuring that the particle density remains well-controlled.
\smallskip

\textbf{Extinction: renormalization.} This item is the topic of Section~\ref{s_proof_extinction}. 
The renormalization scheme we apply in that section is identical to that of our earlier work~\cite{HUVV}, on the contact process on dynamical percolation. 
It involves (half-)crossings of space-time boxes by infection paths in the Poisson graphical construction. 
This follows standard lines: if a box is crossed by an infection path, then we can find two boxes of the lower level inside it that are also crossed; furthermore, these boxes can be taken far apart from each other, and the number of ways that they can be chosen is bounded above by a suitable quantity.
This leads to a recursion in~$N$ on the probability that a box of scale~$N$ is crossed, which is used to prove that this probability tends to zero.
\smallskip

\textbf{Survival: microscopic analysis.} This is done in Section~\ref{s_micro_prop}. 
In contrast with the extinction regime, collisions play a relevant role in the survival regime.
In fact, the infection attempts that are missed due to collisions could potentially drive the process below the supercritical regime of the mean-field model.
A careful analysis is needed to rule that out.

Suppose that~$2dp\lambda > 1$, that~$\sfv$ is large, and that the process starts as in~\eqref{eq_how_it_starts}.
We consider a space-time box of the form~$[-\sqrt{\sfv}\log^2(\sfv),\sqrt{\sfv}\log^2(\sfv)]^d \times [0,h_0]$, where the  height~$h_0$ is taken sufficiently large, depending on~$\lambda$ but not on~$\sfv$. 
We aim to say that, for some sufficiently small~$\varepsilon_0 > 0$ suitably chosen, if 
\begin{enumerate}[a)]
    \item there are~$\sfv^{\varepsilon_0}$ infections in~$[-\sqrt{\sfv},\sqrt{\sfv}]^d$ at time~0, and
    \item the density of particles in~$[-\sqrt{\sfv}\log^2(\sfv),\sqrt{\sfv}\log^2(\sfv)]^d$ at time~0 is close to~$p$, 
\end{enumerate}
then with high probability the infection spreads well up to the top, meaning that, at time $h_0$ there are at least~$\sfv^{\varepsilon_0}$ infections in each of the (overlapping) boxes
\[
[-2\sqrt{\sfv},0] \times [-\sqrt{\sfv},\sqrt{\sfv}]^{d-1},\quad [-\sqrt{\sfv},\sqrt{\sfv}]^{d} \quad \text{and}\quad [0,2\sqrt{\sfv}]\times [-\sqrt{\sfv},\sqrt{\sfv}]^{d-1}.
\]
To demonstrate this, we construct a refined coupling between the set of infected vertices in the interchange-and-contact process and a branching random walk.
By examining the scaling limit of the branching random walk under a diffusive scale, namely, the branching Brownian motion, we establish that the branching random walk spreads effectively and fills the boxes. 
This allows us to show that, with high probability, the interchange-and-contact process behaves similarly.

If the density of particles were always close to~$p$, then the environment would be favorable to the spread of the infection, and the box-to-box propagation would readily provide a ``block argument'' enabling comparison with oriented percolation, as in~\cite{durrett1984oriented}.
Since this is not the case, renormalization is required to address the effect of low-density regions.
\smallskip

\textbf{Survival: renormalization.} This is the topic of Section~\ref{s_surv_renorm}. 
The renormalization scheme required for the survival regime is more involved than the one for extinction.
One key difference is that in the survival regime, the scales grow faster than exponentially, unlike the exponential growth in the extinction regime. 
This adjustment is necessary to account for the reduced propagation speed of good boxes, as noted in Remark \ref{rem_slow} in Section~\ref{s_surv_renorm}. A more significant distinction lies in the vertical decoupling, which is far more intricate in the survival regime. 
In the extinction regime, a box from the bottom scale is assigned the status of good or bad depending solely on infection paths within that box.
These paths are determined by contact interactions (transmissions and recoveries) and the particle configuration (vacant versus occupied) inside the box. 
While the particle configuration is influenced by dynamics outside the box, this dependence can be managed using the standard decoupling method from~\cite{BT}, since it involves the interchange process exclusively.

In the survival regime, the definition of a bad box is much more complex.
We avoid detailing it here but note that it relies on information both inside and outside the box, involving both the interchange dynamics and the infection’s behavior.
This is a significant complication, because revealing information about the healthy-infected status of a particle, implies that the assumption that the rest of the particles are in a product Bernoulli measure is no longer valid. 
Fortunately, our refined version of the decoupling method from~\cite{BT} addresses this by eliminating the need for the particle configuration to be in equilibrium.

\smallskip

The above paragraphs cover the content from Sections~\ref{s_micro_ext} to~\ref{s_surv_renorm}. 
In Section~\ref{s_prelim}, we present several preliminary tools to deal with the interchange and interchange-and-contact process.

\subsection{Notation}
We write~$\mathbb N = \{1,2,\ldots\}$ and~$\mathbb N_0 = \{0,1,\ldots\}$.
For a set~$A$, we let~$|A|$ denote the cardinality of~$A$.
Given~$x \in \mathbb Z^d$ and~$r > 0$, we let%~$B_x(r)$ be the~$\ell_\infty$-ball in~$\mathbb Z^d$ with center~$x$ and radius~$r$, that is
\begin{equation*}
    B_x(r) := \{y \in \mathbb Z^d:\; |y-x|_\infty \le r\}.
\end{equation*}
For~$x,y \in \mathbb Z^d$, we write~$x\sim y$ if~$x$ and~$y$ are nearest neighbors.
Given a set~$S$ and~$\eta \in S^{\mathbb Z^d}$, for each~$x \in \mathbb Z^d$ we define the translation~$\eta \circ \theta(x)$ given by~$(\eta \circ \theta(x))(y)=\eta(y+x)$.

\section{Preliminary constructions and tools}
\label{s_prelim}

This section compiles tools and results used throughout the paper. 
Section~\ref{s_rws} covers basic facts about simple random walks on~$\mathbb Z^d$. 
Section~\ref{s_prelim_interchange} contains the construction of the interchange process, and important bounds for it, including the refinement of the decoupling method from~\cite{BT}. Section~\ref{s_prelim_interchange_and_contact} presents the graphical construction of the interchange-and-contact process, as well as decoupling bounds for it.

\subsection{Random walk notation and estimates}\label{s_rws}
% In this section, we provide a few estimates on random walks that will be useful throughout the paper.
\begin{definition}\label{def_ctrw}
	Let~$(\mathsf{p}(x,y, t): x,y \in \mathbb Z^d,\; t \ge 0)$, denote the transition function of a continuous-time random walk on~$\mathbb{Z}^d$ which jumps from~$x$ to each~$y \sim x$ with rate~$1$, that is,~$\mathsf p$ satisfies
\begin{equation}
\mathsf p(x,y,0)=\mathds{1}_{\{x=y\}},\qquad \frac{\mathrm{d}}{\mathrm{d}t} \mathsf p(x,y,t) = \triangle \mathsf p(x, y
,t),\; t > 0,
 \label{e:ctrw}
\end{equation}
where $
    \triangle f(y) = \sum_{z \sim y} (f(z)-f(y)).
$
\end{definition}
The maximal inequality below provides some control on the trajectory of the random walk.
\begin{lemma}\label{lem_kallenberg}
    For a continuous-time random walk~$(X_t)_{t \ge 0}$ on~$\mathbb Z^d$ with $X_0=0$ and whose transition function satisfies~\eqref{e:ctrw}, we have
	\begin{equation*}
		\label{eq_bound_rw_reach}
		\mathbb P\Bigl(\max_{0\le s \le t} \|X_s\| > x\Bigr) \le 2d\exp\Bigl\{-\frac12 x \log\left(1+\frac{x}{t} \right) \Bigr\},\quad t > 0,\; x > 0.
	\end{equation*}
\end{lemma}
\begin{proof}
    Since the projections of $X_t$ onto each of the coordinates are simply independent continuous-time random walks on~$\mathbb Z$, the desired inequality follows from the case~$d=1$ together with a union bound.
    We thus assume that $d=1$. 
    
    We use a concentration inequality for continuous-time martingales that follows from {Theorem~26.17} in~\cite{kallenberg1997foundations}. 
    Let us briefly explain what is involved. Let~$(M_t)_{t \ge 0}$ be a continuous-time martingale with respect to its natural filtration and $(\langle M_t \rangle)_{t \ge 0}$ its predictable quadratic variation -- the almost-surely unique process that is adapted to the filtration~$(\sigma(\{M_s:s < t\}))_{t \ge 0}$ and is such that~$(M_t^2 - \langle M_t \rangle)_{t \ge 0}$ is a martingale. 
 Assume that the jumps of~$(M_t)$ have absolute value bounded above by a constant~$\kappa > 0$. 
 Fix~$t > 0$ and assume that there is a constant~$\sigma^2_t$ such that~$\langle M_t \rangle \le \sigma_t^2$ almost surely. 
 Then,
	\begin{equation}\label{eq_kal}
		\mathbb P \left(\max_{0 \le s \le t} |M_s-M_0| > x \right) \le 2\exp\left\{ - \frac12 \frac{x}{\kappa} \log \left(1 + \frac{\kappa x}{\sigma_t^2} \right) \right\},\quad x > 0.
	\end{equation}

	Now, let~$(X_t)_{t \ge 0}$ be the continuous-time random walk on~$\mathbb Z$ with $X_0 = 0$ and whose transition function satisfies~\eqref{e:ctrw}. 
    It is readily seen that~$(X_t)$ is a martingale whose jumps have absolute value equal to 1 and that~$\langle X_t \rangle = t$. The desired inequality follows immediately from~\eqref{eq_kal}.
\end{proof}

We now turn to the control of the probability that two independent random walks starting at distance at most $\ell$ meet before time $\ell^2$.
\begin{definition}[The probability of meeting]
\label{defi:meet}
	Given~$x, y \in \mathbb Z^d$, let~$\mathbb P_{x,y}$ be a probability measure under which we have two independent continuous-time random walks~$(X_t)_{t\ge 0}$ and~$(X_t')_{t \ge 0}$, both with jump rate~$1$ (as in Definition~\ref{def_ctrw}), with~$X_0 = x$ and~$X_0' = y$.
	For any~$\ell \in \mathbb N$, we let
	\begin{equation*}
		\mathrm{meet}(\ell):=\inf\big\{\mathbb P_{x,y}(\exists s \le \ell^2:  X_s = X_s' ):\;x,y \in B_0(\ell)\big\}.
	\end{equation*}
	\end{definition} 

\begin{lemma}\label{lem_meet}
	There exists~$c:=c(d)> 0$ such that for any~$\ell \in \mathbb N$,
	\begin{equation}\label{eq_meet}
		\mathrm{meet}(\ell) \ge \begin{cases}
c&\text{if }d=1;\\
c/\log(\ell)&\text{if }d=2;\\
c/\ell^{d-2}&\text{if }d \ge 3.
            \end{cases}
	\end{equation}
\end{lemma}
%\marginpar{explicitar $c = c(d)?$}
\begin{proof}
  Since the process  $(Z_t)_{t \ge 0}$ with $Z_t:=X_t-X'_t$ is a continuous-time simple random walk on $\mathbb Z^d$ with jump rate 2, starting at $Z_0:=X_0-X_0'$, it suffices to show the lower bound on the r.h.s.\ of \eqref{eq_meet} for the probability that this random walk hits the origin up to time $\ell^2$ uniformly on the initial $z \in B_0(2\ell)$, for all large $\ell$. 
  But we can write $Z_t=Y(N_t)$ where $(Y(n))_{n\ge 0}$ is a discrete-time simple random walk with $Y(0)=Z_0$ and $(N_t)_{t \ge 0}$ is an independent Poisson process with rate 2. 
  Conditioning on $N_t$ and using that $N_t/t$ converges a.s.\ to 2 we are left with an estimate for the discrete-time random walk. This estimate is given exactly in~\cite[Lemma~2.4]{popov}. 
\end{proof}

\subsection{The interchange process}
\label{s_prelim_interchange}
\begin{definition}[Partial order]
	Let~$\Lambda \subseteq \mathbb Z^d$. We endow~$\{0,1\}^{\Lambda}$ with the partial order~$\le$ defined by declaring~$\xi \le \xi'$ when~$\xi(x) \le \xi'(x)$ for all~$x$. For two probability measures~$\mu$ and~$\mu'$ on~$\{0,1\}^{\mathbb Z^d}$, we write~$\mu \preceq \mu'$ if~$\mu$ is stochastically dominated by~$\mu'$ with respect to this partial order (similarly, if~$\xi,\xi'$ are random configurations, we write~$\xi \preceq \xi'$ if the law of~$\xi$ is stochastically dominated by that of~$\xi'$).
\end{definition}

%\begin{definition}[The measure~$\pi_p^A$] Given~$p \in [0,1]$ and a finite set~$A \subseteq \mathbb Z^d$, we let~$\pi^A_p$ be the measure on~$\{0,1\}^{\mathbb Z^d}$ given by~$\pi_p^A(\cdot):= \pi_p(\cdot \mid \{\xi:\xi \equiv 1 \text{ on }A\})$, that is, for~$\xi \sim \pi^A_p$, we have~$\xi(x) = 1$ for all~$x \in A$, and~$\{\xi(x): x \notin A\}$ are independent, $\mathrm{Bernoulli}(p)$.
%\end{definition}
\subsubsection{Graphical representation and interchange flow}

\begin{definition}[Graphical representation and flow of the interchange process] \label{def_interchange_flow}
	A \emph{graphical representation of the interchange process} with rate~$\mathsf v > 0$ is a collection
	\[\mathcal J:= (\mathcal J_{\{x,y\}}: \{x,y\} \text{ is an edge of }\mathbb Z^d),\]
	of independent Poisson point processes~$\mathcal J_{\{x,y\}}$  on~$[0,\infty)$ with intensity~$\mathsf v$. 
 Arrivals of these Poisson processes are called \emph{jump marks}. 
	Given a realization of~$\mathcal J$, we define the \emph{interchange flow}~$\Phi(x,s,t)=\Phi_{\mathcal J}(x,s,t)$ as follows. For any~$x \in \mathbb Z^d$ and~$s \ge 0$,~$t \mapsto \Phi(x,s,t)$ is the (almost surely well-defined) function satisfying $\Phi (x,s,s)=x$ and, for $t > s$, 
	\begin{align*}
		&\Phi(x,s,t-) = y,\; t \in \mathcal J_{\{y,z\}} &\Longrightarrow &&\Phi(x,s,t)= z; \\
		&\Phi(x,s,t-) = y,\; t \notin \cup_{z \sim y} \mathcal J_{\{y,z\}} &\Longrightarrow  &&\Phi(x,s,t)= y.
	\end{align*}
\end{definition}

Note that, for any~$s\le t$, the function~$x \mapsto \Phi(x,s,t)$ is a random permutation of~$\mathbb Z^d$. It is straightfoward to check that, for any~$x$ and~$s$,~$[s,\infty) \ni t \mapsto \Phi(x,s,t)$ has the distribution of a continuous-time random walk that starts at~$x$ at time~$s$, and jumps to each neighboring position with rate~$\mathsf v$.

{ For~$s > 0$ and~$t \in [0,s)$, we  define~$\Phi(x,s,t)$ as the unique~$y \in \mathbb Z^d$ such that~$\Phi(y,t,s)=x$. 
With this,~$\Phi(x,s,t)$ is now defined for all~$s \ge 0$ and all~$t \ge 0$.
Note that~$(\Phi(x,0,s):x \in \mathbb Z^d,\;0 \le s \le t)$ has the same law as~$(\Phi(x,t,t-s):x \in \mathbb Z^d,\;0\le s \le t)$. 
This property is known as the \emph{self-duality} of the interchange flow.}

Given~$\xi_0 \in \{0,1\}^{\mathbb Z^d}$ and a realization of~$\mathcal J$ with flow~$\Phi$, we obtain the interchange process by setting, for any~$x \in \mathbb Z^d$ and~$t \ge 0$,
\begin{equation*}
	\xi_t(x) = \xi_0(\Phi(x,t,0)).
\end{equation*}

	 We will also need the following estimate. 
\begin{lemma}
	\label{lem_time_together}
	There exists~$C > 0$ such that for any~$\mathsf v > 0$, if~$\Phi$ is the flow of the interchange process with rate~$\mathsf v$, then for any~$t \ge 0$ and any~$x,y \in \mathbb Z^d$, we have
	\[\mathbb E\left [\int_0^t \mathds{1}\{ \Phi(x,0,s) \sim \Phi(y,0,s)\}\;\mathrm{d}s \right] \le \begin{cases} C \sqrt{t/\mathsf v} &\text{if } d =1;\\ C\log(\mathsf{v}t)/\mathsf{v} & \text{if } d = 2; \\ C/\mathsf v&\text{if } d \ge 3.\end{cases}\]
\end{lemma}
\begin{proof}
	%{\color{red} (To be done.)}
Applying Proposition 1.7 in \cite[Chapter VIII]{Lig} we have
\[
\mathbb P\left ( \Phi(x,0,s) \sim \Phi(y,0,s)\right) \leq \mathbb P\left (X^x_{s\mathsf v} \sim X^y_{s\mathsf v}\right),
\]
where $X^x_s$ and $X^y_s$ denote the positions at time $s$ of two independent, unit rate, simple symmetric random walks, starting at $x$ and $y$ respectively.
The conclusion follows easily from classical estimates on random walks.
\end{proof}

\subsubsection{Discrepancy and spatial decoupling}
\begin{definition}[The discrepancy probability for the interchange process]
\label{defi:discr_ip}
	Let~$\Phi$ be the interchange flow with rate~$\sfv = 1$. 
 We then write, for every~$\ell, L \in \mathbb N$ with~$\ell < L$ and~$t > 0$,
\begin{equation*}
	\mathrm{discr}^{\mathrm{ip}}(\ell, L, t) := \mathbb P(\exists x \in \partial B_0(L),\;0 \le s < s' \le t:\; \Phi(x,s,s') \in \partial B_0(\ell)).
\end{equation*}
\end{definition}
The reason we call this a \emph{discrepancy} probability is as follows: if~$(\xi_t)$ and~$(\xi_t')$ are two interchange processes obtained from the same graphical representation, and~$\xi_0(x) = \xi_0'(x)$ for all~$x \in B_0(L)$, then
\begin{equation}\label{eq_inclusion_discr}
\{\exists x \in B_0(\ell),\; s \in [0,t]: \; \xi_s(x) \neq \xi'_s(x)\} 
    \subseteq \biggl\{
    \begin{array}{c}
         \text{$\exists x \in \partial B_0(L),$}  \\
          \text{$0 \le s < s' \le t$}
    \end{array}
    :\; \Phi(x,s,s') \in \partial B_0(\ell) \biggr\}.
\end{equation}
As a consequence, we have the following.
\begin{lemma}\label{lem_covariances_0}
Let~$(\xi_t)_{t \ge 0}$ be the interchange process with rate~$\mathsf v=1$. Let~$\ell \in \mathbb N$,~$x_1,x_2 \in \mathbb Z^d$ with~$\|x_1 - x_2\| \ge 2\ell + 2$, and~$t > 0$. For~$i = 1,2$, let~$A_i$ be an event whose occurrence depends only on~$\{\xi_s(y): (y,s) \in B_{x_i}(\ell) \times [0,t]\}$. Then,
\begin{equation*}
	|\mathrm{Cov}(\mathds{1}_{A_1},\mathds{1}_{A_2})| \le 4\, \mathrm{discr}^{\mathrm{ip}}(\ell, \lfloor\tfrac12\|x_1-x_2\|\rfloor , t).
\end{equation*}
\end{lemma}
\begin{proof}
Let~$L:=\lfloor \tfrac12\|x_1-x_2\|\rfloor > \ell$.  Assume that the interchange process is obtained from a graphical representation and let~$\Phi$ be the associated interchange flow. 
    For~$i=1,2$, let
    \[
    E_i = \{\text{there is no~$y \in \partial B_{x_i}(L)$,~$z \in \partial B_{x_i}(\ell)$, and~$s < s' < t$ such that~$z = \Phi(y,s,s')$}\}.
    \]
    Note that the occurrence of~$E_i$ only depends on the graphical representation inside~$B_{x_i}(L) \times [0,t]$. Moreover, using~\eqref{eq_inclusion_discr}, we see that if~$E_i$ occurs, then the values~$\{\xi_s(y):(y,s) \in B_{x_i}(\ell)\times [0,t]\}$ can be determined from~$\{\xi_0(y):y \in B_{x_i}(L)\}$ and from the graphical representation inside~$B_{x_i}(L) \times [0,t]$.
    
    Conditioning on~$\xi_0$ and using the above considerations, as well as the fact that~$(B_{x_1}(L) \times [0,t]) \cap (B_{x_2}(L) \times [0,t]) = \varnothing$, we write
\begin{align*}
\mathbb P(A_1 \cap A_2 \mid \xi_0) = \mathbb P(A_1 \cap E_1 \mid \xi_0) \cdot \mathbb P(A_2 \cap E_2 \mid \xi_0) + \mathbb P(A_1 \cap A_2 \cap (E_1 \cap E_2)^c \mid \xi_0). 
\end{align*}
Note that~$\mathbb P(E_i^c) = \mathbb P(E_i^c\mid \xi_0)= \mathrm{discr}^{\mathrm{ip}}(\ell,L,t)$, where the second equality follows from Definition \ref{defi:discr_ip}.
Hence,
\begin{align*}
|\mathbb P(A_1 \cap E_1 \mid \xi_0) \cdot \mathbb P(A_2 \cap E_2 \mid \xi_0) - \mathbb P(A_1 \mid \xi_0) \cdot \mathbb P(A_2 \mid \xi_0)| 
    &\le 2 \mathrm{discr}^{\mathrm{ip}}(\ell, L, t) \\
\text{and}\hspace{4cm}
\mathbb P(A_1 \cap A_2 \cap (E_1 \cap E_2)^c \mid \xi_0)
    &\le 2 \mathrm{discr}^{\mathrm{ip}}(\ell, L, t).
\end{align*}
This proves that $|\mathrm{Cov}(\mathds{1}_{A_1},\mathds{1}_{A_2} \mid \xi_0)| \le 4 \mathrm{discr}^{\mathrm{ip}}(\ell, L, t)$
and the result now follows from integrating with respect to~$\xi_0$.
\end{proof}

% We now want to establish an upper bound for the discrepancy probability. 
% The following is an auxiliary step.
The next lemma helps us to establish an upper bound for the discrepancy probability. 
\begin{lemma}\label{lem_prelim_discr_ip}
   Let~$\Phi$ be the interchange flow with rate~$\sfv = 1$. For any~$t \ge 1$ and~$\ell \in \mathbb N$, we have
\begin{equation}\label{eq_com_prep}
    \mathbb P(\exists s, s':\; 0 \le s \le s' \le t,\; \| \Phi(0,s,s')\| \ge \ell) \le 8ed^2t \cdot \exp\Big\{ -\ell \cdot \log\Bigl(1 + \frac{\ell}{2t} \Bigr)\Big\}. 
\end{equation}
\end{lemma}
\begin{proof}
Let $\epsilon >0$, 
\[\mathcal S:= \bigl\{ s' \in [0,t+\epsilon]: \max_{s \in [0,s']} \|\Phi(0,s,s')\| \ge \ell \bigr\}
\quad \text{and} \quad \sigma := \inf \mathcal S,\]
with the usual convention that $\inf\emptyset=+\infty$. 
The l.h.s.\ of~\eqref{eq_com_prep} equals $\mathbb P(\sigma \leq t)$. On this event, let $\mathcal Y \in \mathbb Z^d$ be the random position where $\max_{s \in [0,\sigma]} \|\Phi(0,s,\sigma)\|$ is achieved. Set
\begin{equation*}
		\mathcal A:= \{\sigma \leq t, \;\text{there is no jump mark from }\mathcal Y  \text{ in the time interval } [\sigma, \sigma +\epsilon]\}.
	\end{equation*}
    Notice that on~$\mathcal A$, $\mathrm{Leb}(\mathcal S) \ge \epsilon$,
	and by the strong Markov property,
$\mathbb P(\mathcal A \mid \sigma \leq t) = e^{-2d\epsilon}$,
  so
\begin{equation}\label{eq_sigma_A_Leb}
			\mathbb P(\sigma \leq t ) = e^{2d\epsilon} \cdot \mathbb P(\mathcal A ) \le \frac{e^{2d\epsilon}}{\epsilon} \cdot \mathbb E[\mathrm{Leb}(\mathcal S)]
            \le \frac{e^{2d\epsilon}(t+\epsilon)}{\epsilon}\cdot\ \sup_{\mathclap{s' \in [0,t+\epsilon]}}\ \mathbb P(s' \in \mathcal S).
	\end{equation}
    
	Lemma~\ref{lem_kallenberg} allows us to uniformly bound the supremum on the r.h.s. by
	\begin{equation*}
		\mathbb P\Bigl(\max_{s \in [0,s']} \|\Phi(0,s,s')\| \ge \ell \Bigr) \le 2d \exp \Bigl\{- \ell \cdot \log \Bigl(1+ \frac{\ell}{t+\epsilon} \Bigr) \Bigr\}.
	\end{equation*}
	We take~$\epsilon := (2d)^{-1}$. 
	To make the formulas cleaner, we add the assumption that~$t \ge 1$ and  bound~$t+\epsilon \le 2t$, so the r.h.s. above is bounded by the r.h.s. in~\eqref{eq_com_prep}.
\end{proof}

%\end{proof}
We are now ready to establish the desired bound on the discrepancy.
\begin{lemma} \label{lem_first_rw_discr}
	For any~$t \ge 1$,~$\ell, L \in \mathbb N$ with~$L  \ge \ell + 2$, we have
	\begin{equation}\label{eq_with_t_1}
	\mathrm{discr}^{\mathrm{ip}}(\ell, L, t) \le 16ed^3 t (2L+1)^{d-1} \exp \Bigl\{- (L- \ell) \cdot \log \Bigl(1 + \frac{L-\ell}{2t} \Bigr) \Bigr\}.
	\end{equation}
\end{lemma}
\begin{proof}
The statement follows from Lemma~\ref{lem_prelim_discr_ip} and the union bound
	\begin{equation*}
		\mathrm{discr}^{\mathrm{ip}}(\ell,L,t) \le 2d(2L +1)^{d-1} \cdot \mathbb P\left(\exists s,s': 0 \le s < s' \le t,\; \|\Phi(0,s,s')\| \ge L-\ell \right).\hspace{5mm}\mbox{\qedhere}
	\end{equation*}
\end{proof}

\subsubsection{Domination by product measures and temporal decoupling}
\label{ss_domination_product}
\begin{definition}[The functions~$g^{\uparrow}$ and~$g^{\downarrow}$] \label{def_gs}
	Let $\ell, L \in \mathbb N$ with $\ell < L$, $t > 0, p \in [0,1]$, and  $\xi \in \{0,1\}^{\mathbb Z^d}.$
	Let~$(\xi_t)_{t \ge 0}$ be the interchange process with rate~$\sfv=1$ with~$\xi_0=\xi$, and set
	\begin{align*}
		&g^{\uparrow}(\ell, L, t,p,\xi):= \mathbb P\left( \begin{array}{l} \text{for some } s \le t \text{ and some box $B$ with radius $\ell$ contained} \\ \,\,\,\,\,\,\, \text{in $B_0(L)$, we have } |\{y \in B: \xi_s(y) = 1\}| > p |B|\end{array}\right),\\[.1cm]
		&g^{\downarrow}(\ell, L, t,p,\xi):= \mathbb P\left( \begin{array}{l} \text{for some } s \le t \text{ and some box $B$ with radius $\ell$ contained} \\ \,\,\,\,\,\,\, \text{in $B_0(L)$, we have } |\{y \in B: \xi_s(y) = 1\}| < p |B|\end{array}\right).
	\end{align*}
\end{definition}

\begin{lemma}[Stochastic domination between interchange processes] \label{lem_coupling_rate_one}
	Given~$\xi, \xi' \in \{0,1\}^{\mathbb Z^d}$, there exists a probability space in which there are two graphical representations of the interchange process with rate one, with the following property. 
    For any spatial scales $\ell, L \in \mathbb{N} \text{ with } \ell < L, \text{ times } t,T > 0 \text{ with } t \le T, \text{and parameter } p \in [0,1)$,
	we have\vspace{-2mm}
	\begin{equation*}
		\xi'_s(x) \ge \xi_s(x) \quad \text{for all } (x,s) \in B_0(L/4) \times [t,T]
	\end{equation*}	outside an event of probability at most
	\begin{align}\label{eq_error_prob_coupling}
		g^\uparrow(\ell, L, t, p, \xi) &+ g^{\downarrow}(\ell, L,  t, p, \xi') + \mathrm{err}_{\mathrm{coup}}(\ell, L, t, T),\qquad \text{where}\\
\label{eq_err_coup}
		\mathrm{err}_{\mathrm{coup}}(\ell, L, t, T) &:=|B_0(L/2)|\cdot \left( 1 - \mathrm{meet}(\ell) \right)^{\lfloor t/\ell^2 \rfloor} + \mathrm{discr}^{\mathrm{ip}}(L/4, L/2, T). 
	\end{align}
\end{lemma}
% \marginpar{\tiny E: a separação das densidades não fica no lema seguinte? Parece-me que a prova do lema 2.7 está OK, mas há comentários (antes da prova) que podem dar essa ideia e talvez devam ser editados}

This is obtained from a coupling method introduced in~\cite{BT}. Due to some particularities of our context, we provide some details of the proof in Appendix~\ref{sec:decoupling}.

\begin{remark}
It will be useful to bound
\begin{equation}\label{eq_better_bound_meet}
\left( 1 - \mathrm{meet}(\ell) \right)^{\lfloor t/\ell^2 \rfloor} \le \begin{cases} \exp\{-ct/\ell^2\}&\text{if }d=1;\\[.1cm]
\exp\{-ct/(\ell^2\log(\ell))\}&\text{if }d =2;\\[.1cm]
\exp\{-c t/\ell^d\}&\text{if }d \ge 3
\end{cases}\qquad \text{ for all } \ell \in \mathbb N,\; t \ge \ell^2,
\end{equation}
where~$c$ is the constant appearing at Lemma~\ref{lem_meet} {divided by 2}.
This is obtained by using~$1-x \le e^{-x}$, bounding~$\lfloor t/\ell^2 \rfloor \ge t/(2\ell^2)$ when~$t \ge\ell^2$, and using the dimension-dependent lower bound for~$\mathrm{meet}(\ell)$ given by Lemma~\ref{lem_meet}.
\end{remark}
 
\begin{definition}[The measures~$\pi_p$ and~$\pi_p^A$]\label{def_pi_p_A} 
	Let~$p \in [0,1]$, and~$\pi_p$ be the $\mathrm{Bernoulli}(p)$ product measure over vertices of~$\mathbb Z^d$. 
    Given~$A \subseteq \mathbb{Z}^d$, we let~$\pi_p^A$ be the measure~$\pi(\cdot \mid \{\xi:\xi(x)=1\text{ for all } x \in A\})$.
\end{definition}

\begin{lemma}\label{lem_integral_gs}
	Let~$\ell, L \in \mathbb N$ with~$\ell < L$,~$t > 0$, and~$p, p' \in [0,1]$ with~$p < p'$. Then, 
	\[\int_{\{0,1\}^{\mathbb Z^d}} g^\uparrow(\ell,L,t,p',\xi)\;\pi_p(\mathrm{d}\xi)\quad \text{and}\quad \int_{\{0,1\}^{\mathbb Z^d}} g^\downarrow(\ell,L,t,p,\xi)\;\pi_{p'}(\mathrm{d}\xi) \]
	are both smaller than
	\begin{align*}
	(2L+1)^d \cdot	 \left( e(2\ell+2)^d t + e \right) \cdot  \exp\left\{-2 (2\ell+1)^d (p'-p)^2 \right\}.	\end{align*}
\end{lemma}

\begin{proof}
    We only prove the bound for the first integral, as the second may be treated in the same way.  
    Given~$\xi \in \{0,1\}^{\mathbb Z^d}$, let
	\[f(\xi) := \mathds{1}\left\{\text{there is a box $B$ with radius~$\ell$ contained in $B_0(L)$ such that $|\xi \cap B|/|B| > p'$} \right\}.\]
	Let~$(\xi_s)_{s \ge 0}$ be the interchange process with rate~$\mathsf v = 1$ started from a random configuration~$\xi_0 \sim \pi_p$. Defining~$\tau := \inf\{s \ge 0: f(\xi_s) = 1\}$, we have
\[
	\int_{\{0,1\}^{\mathbb Z^d}} g^\uparrow(\ell,L,t,p',\xi)\;\pi_p(\mathrm{d}\xi) = \mathbb P(\tau \le t).
\]

	Letting~$(\mathcal F_t)_{t \ge 0}$ denote the natural filtration of the process, we claim that for every~$\epsilon > 0$,
	\begin{equation}
		\label{eq_first_time_trick}
\text{on } \{\tau < \infty\}, \quad  \mathbb P(f(\xi_s) = 1 \text{ for all } s \in [\tau, \tau + \epsilon] \mid \mathcal F_\tau) \ge \exp\{-(2\ell + 2)^d\epsilon \}.
	\end{equation}
	To see this, we argue as follows. On~$\{\tau < \infty\}$, there is a ball~$B \subset B_0(L)$ with radius~$\ell$ such that~$|\xi_\tau \cap B|/|B| > p'$. The number of edges intersecting this ball is~$(2\ell + 2)^d$, and if none of these edges has an interchange jump in the time interval~$[\tau, \tau + \epsilon]$, then we have~$f(\xi_s) = 1$ for all~$s \in [\tau,\tau + \epsilon]$.
	We now bound, for arbitrary~$\epsilon > 0$,
    \begin{align*}
		\mathbb E \left[\int_0^{t+\epsilon} \hspace{-4mm}f(\xi_s)\;\mathrm{d}s\right] &\ge \epsilon \cdot \mathbb P(\tau \le t,\; f(\xi_s) = 1\ \forall s \in [\tau,\tau+\epsilon])
		\ge \epsilon \cdot \exp\{-(2\ell + 2)^d\epsilon \} \cdot \mathbb P(\tau \le t).
	\end{align*}
	where the second inequality follows from~\eqref{eq_first_time_trick} and the strong Markov property. Taking~$\epsilon = (2\ell + 2)^{-d}$, rearranging and using Fubini's theorem, this gives
	\[\mathbb P(\tau \le t) \le e (2\ell + 2)^d \cdot \int_0^{t+(2\ell+2)^{-d}} \mathbb E[f(\xi_s)]\mathrm{d}s = e(2\ell+2)^d \cdot \left( t + (2\ell + 2)^{-d}\right) \cdot \mathbb E[f(\xi_0)],\]
	where the equality holds because~$\pi_p$ is stationary for the interchange process.
    By a union bound over all boxes of radius~$\ell$ contained in~$B_0(L)$, we have
   \[
 \mathbb E[f(\xi_0)] \le   (2L+1)^d\cdot \mathbb P(\mathrm{Bin}((2\ell + 1)^d,p)> (2\ell + 1)^d p').
   \]
   Using Hoeffding's Inequality (see ~\cite [Sec. 2.6]{BLM}), the probability appearing on the r.h.s.\ is bounded above by
$\exp\left\{-2 (2\ell + 1)^d (p' - p)^2 \right\}$. This completes the proof. 
\end{proof}

\subsection{The interchange-and-contact process}
\label{s_prelim_interchange_and_contact}
\begin{definition}[The measure~$\hat \pi^A_p$] \label{def_pi_hat} Let~$p \in [0,1]$ and~$A \subset \mathbb Z^d$. We define~$\hat \pi_p^A$ as the measure on~$\{0,\stateh, \statei\}^{\mathbb Z^d}$ such that, if~$\zeta \sim \hat \pi_p^A$, then~$\zeta(x) = \statei$ for all~$x\in A$, and outside~$A$, independently at each vertex~$x$,~$\zeta(x)$ equals~$\stateh$ with probability~$p$ and~$0$ with probability~$1-p$.
\end{definition}
\begin{definition}[Projection]\label{def_projections}
    Given~$\Lambda \subseteq \mathbb Z^d$ and~$\zeta \in \{0,\stateh,\statei\}^{\Lambda}$, we define~$\xi^\zeta \in \{0,1\}^{\Lambda}$ by setting\vspace{-2mm}
    \[\xi^\zeta(x) = \begin{cases} 1&\text{if } \zeta(x) \in \{\stateh,\statei\};\\ 0&\text{otherwise.} \end{cases}\]
\end{definition}
\subsubsection{Graphical representation, infection paths and containment flow}

\begin{definition}[Graphical representation of the interchange-and-contact process] \label{def_graph_cpip}
The \emph{graphical representation of the interchange-and-contact process} with jump rate~$\mathsf v >0$ and infection rate~$\lambda > 0$ is a collection~$H$ of independent Poisson point processes on~$[0,\infty)$, as follows:
\begin{itemize}
	\item for each edge~$\{x,y\}$ of~$\mathbb Z^d$, a process~$\mathcal J_{\{x,y\}}$ with rate~$\mathsf v$ (\emph{jump marks});
	\item for each vertex~$x$ of~$\mathbb Z^d$, a process~$\mathcal{R}_x$ with rate~$1$ (\emph{recovery marks});
	\item for each ordered pair~$(x,y)$ of vertices of~$\mathbb Z^d$ with~$x \sim y$, a process~$\mathcal T_{(x,y)}$ with rate~$\lambda$ (\emph{transmission marks}).
\end{itemize}
\end{definition}

As is the case for the classical contact process, the interchange-and-contact process can be obtained from an initial configuration and the graphical representation, using the notion of \emph{infection paths}.

\begin{definition}[Infection path]
   Let~$(\xi_t)_{t \ge 0}$ be the interchange process started from~$\xi_0$ and graphical representation~$H$. 
   Let~$s < s'$. An \emph{infection path} is a function~$\gamma:[s,s'] \to \mathbb Z^d$ such that
   \begin{align}\label{eq_containm3}
       t \notin \mathcal R_{\gamma(t)} \quad &\text{for all }t \in [s,s'],\\
       \label{eq_containm4}
\xi_t(\gamma(t)) = 1 \quad &\text{for all }t \in [s,s'],
   \end{align}
   and such that there exist times~$s=s_0 < s_1 < \cdots < s_n < s_{n+1}=s'$ with
\begin{align}\label{eq_containm1}
    s_j \in \mathcal T_{(\gamma(s_j-),\gamma(s_j))} \quad &\text{for all } j \in \{1,\ldots,n\}\\
    \label{eq_containm2}
    \gamma(t)=\Phi(\gamma(s_j),s_j,t) \quad &\text{for all }  j \in \{0,\ldots, n\}, \;t \in [s_j,s_{j+1}).
\end{align}
If~$\gamma(s)=x$ and~$\gamma(s')=y$, we say that~$\gamma$ is an infection path from~$(x,s)$ to~$(y,s')$.
\end{definition}
Property~\eqref{eq_containm3} says that the path does not touch any recovery mark. Property~\eqref{eq_containm4}
means that it only passes by space-time points that are occupied by particles in the interchange process.   Property~\eqref{eq_containm1} means that it may jump by following transmission marks, and~\eqref{eq_containm2} that in between those transmission jump times, it must follow the interchange flow.
We emphasize that, unlike in the classical contact process, here the notion of infection path depends \emph{both} on the graphical representation \emph{and on the initial configuration}. This is natural, since infections are tied to particles. 

Given a realization of the graphical representation~$H$ and an initial configuration~$\zeta_0 \in \{0, \stateh, \statei\}^{\mathbb Z^d}$, we can now construct the interchange-and-contact process~$(\zeta_t)_{t \ge 0}$ as follows. Let~$\xi_0=\xi^{\zeta_0}$, as in Definition~\ref{def_projections}, and let~$(\xi_t)_{t\ge 0}$ be the interchange process started from~$\xi_0$ and constructed from (the jump marks in)~$H$. Then, for~$t \ge 0$ and~$y \in \mathbb Z^d$, set~$\zeta_t$ as follows:
\begin{itemize}
    \item if~$\xi_t(y)=0$, then~$\zeta_t(y)=0$;
    \item if~$\xi_t(y)=1$ and there is~$x \in \mathbb Z^d$ such that~$\zeta_0(x)=\statei$ and there is an infection path from~$(x,0)$ to~$(y,t)$, set~$\zeta_t(y)=\statei$;
    \item otherwise, set~$\zeta_t(y)=\stateh$.
\end{itemize}

The fact that infection paths now depend on the initial configuration complicates the analysis of the process. It is convenient to define a broader class of paths, which satisfy~\eqref{eq_containm1} and~\eqref{eq_containm2} above, but not necessarily~\eqref{eq_containm3} or~\eqref{eq_containm4}. In particular, the removal of~\eqref{eq_containm4} eliminates the dependence on the initial configuration.
\begin{definition}[Containment path and flow]
    \label{def_containment_try}
    Let~$H$ be a graphical representation of the interchange-and-contact process. Let~$s < s'$. A \emph{containment path} is a function~$\gamma:[s,s'] \to \mathbb Z^d$ such that there exist~$s_1 < \cdots < s_n$ such that~\eqref{eq_containm1} and~\eqref{eq_containm2} hold. We define the \emph{containment flow} $\Psi(x,s,t) = \Psi_H(x,s,t)$, for~$x \in \mathbb Z^d$ and~$t \ge s \ge 0$ by letting\vspace{-1mm}
	\begin{equation*}
		\Psi(x,s,t):= \{y \in \mathbb Z^d: \text{there is a containment path from $(x,s)$ to $(y,t)$}\}.
	\end{equation*}
    { Given~$A \subseteq \mathbb Z^d$, we write}\vspace{-1mm}
    \begin{equation}\label{eq_flow_from_A}
    { \Psi^A_t:=\{y \in \mathbb Z^d:\text{there is a containment path from }(x,0) \text{ to }(y,t) \text{ for some } x \in A\}.}
    \end{equation}
\end{definition}
We have thus defined a set-valued process~$(\Psi(x,s,t))_{t \ge s}$  with~$\Psi(x,s,s) = \{x\}$, {as well as~$(\Psi^A_t)_{t\ge 0}$ with~$\Psi^A_0 = A$}. As usual, we will abuse notation and treat~$\Psi(x,s,t)$ {and~$\Psi^A_t$} as elements of~$\{0,1\}^{\mathbb Z^d}$, by associating a set with its indicator function.

Using the definition of containment paths, it is easy to check that~$(\Psi(x,s,t))_{t \ge s}$ is a spin system which behaves as a contact process with stirring with no recoveries, and in a situation where the lattice is completely occupied by particles. 
 For future use, it will be useful to spell out how it obeys the instructions of the graphical representation:
	\begin{itemize}
		\item[$\mathrm{(r1)}$] if~$t \in \mathcal{J}_{\{w,z\}}$, then            
			\begin{align*}
				[\Psi(x,s,t)](w) &= [\Psi(x,s,t-)](z),\\
				[\Psi(x,s,t)](z) &= [\Psi(x,s,t-)](w), \\
				[\Psi(x,s,t)](u) &= [\Psi(x,s,t-)](u)\;\forall u \notin \{w,z\};
			\end{align*}
	\item[$\mathrm{(r2)}$] recovery marks have no effect;
	\item[$\mathrm{(r3)}$] if~$t \in \mathcal{T}_{(w,z)}$, then
			\begin{align*}
				&[\Psi(x,s,t)](z) = \begin{cases} 1 &\text{if } [\Psi(x,s,t-)](w) = 1;\\ [\Psi(x,s,t-)](z) &\text{otherwise;} \end{cases}\\
					&[\Psi(x,s,t)](u) = \ \ [\Psi(x,s,t-)](u) \quad \forall u \neq z.
			\end{align*}
	\end{itemize}

The reason we use the word `containment' is given by the following lemma, whose proof is elementary and thus omitted:
\begin{lemma}
	\label{lem_contain}
	Let~$(\zeta_t)_{t \ge 0}, (\zeta_t')_{t \ge 0}$ be interchange-and-contact processes built from the same graphical representation~$H$ and started from~$\zeta_0,\zeta_0'\in \{0,\stateh,\statei\}^{\smash{\mathbb Z^d}}\hspace{-2mm}$, respectively. Letting 
	$A := \{x: \zeta_0(x) \neq \zeta'_0(x)\}$, we have
	\[\{x: \zeta_t(x) \neq \zeta'_t(x)\} \subseteq  \Psi^A_t\quad \text{for all } t \ge 0.\]
	In particular, for any~$(y,t) \in \mathbb Z^d \times [0,\infty)$, 
	% \begin{equation*}
	% 	\text{if } \zeta_0 \equiv \zeta_0' \text{ on } \{x: \Psi(x,0,t) \ni y \}, \quad \text{then } \zeta_t(y) = \zeta_t'(y).
	% \end{equation*}
    if $\zeta_0 \equiv \zeta_0'$ on $\{x: \Psi(x,0,t) \ni y \}$, then $\zeta_t(y) = \zeta_t'(y)$.
\end{lemma}

\subsubsection{Growth of the containment flow}

Recall the random walk transition kernel~$\mathsf p$ (Definition~\ref{def_ctrw}).

\begin{lemma}\label{lem_primeira_dom}
Let~$\Psi$ be the containment flow of an interchange-and-contact process with parameters~$\mathsf v$ and~$\lambda$. For any~$t > 0$ and any~$x\in \mathbb Z^d$, we have
	\begin{equation}\label{eq_containment_estimate}
		\mathbb P\big(x \in \Psi^{\{0\}}_t \big) \le e^{2d\lambda t}\cdot \mathsf p(0,x,(\mathsf v + \lambda)t).
	\end{equation}
\end{lemma}
\begin{proof}
	Using the same graphical representation under which the containment flow is defined, we define an auxiliary process~$(\kappa_{t})_{t \ge 0}$ taking values in~$(\mathbb N_0)^{\mathbb Z^d}$ as follows. We let~$\kappa_{0}(0) = \mathds{1}_{\{0\}}$. The instructions in the graphical representation have the following effects for~$(\kappa_{t})$: rules (r1) and (r2) after Definition~\ref{def_containment_try} are applied in the same way, while rule (r3) is replaced by
	\begin{itemize}
		\item[$\mathrm{(r3')}$] if~$t \in \mathcal{T}_{(w,z)}$, then
				$\kappa_t(u) = \begin{cases} \kappa_{t-}(w)+\kappa_{t-}(z) &\text{if } u =z; \\ \kappa_{t-}(u)&\text{otherwise}.\end{cases}$	
	\end{itemize}
% Intuitively, simultaneously, each particle sitting at~$w$ creates one child particle at~$z$.
It can be readily checked that, for any~$t \ge 0$, $\Psi^{\{0\}}_t \leq \kappa_{t}$, so	$\mathbb{P}\big(x \in \Psi^{\{0\}}_t \big) \le \mathbb{E}\left[\kappa_{t}(x)\right]$. A standard generator computation shows that the function
	$(t,x) \mapsto \mathbb E[\kappa_{t}(x)],\, (t,x) \in [0,\infty)\times \mathbb Z^d$ solves
	\begin{equation*}
		\left\{
			\begin{array}{l} \frac{\mathrm{d}}{\mathrm{d}t} f(t,x) = (\mathsf v + \lambda) \triangle f(t,x) +2d\lambda f(t,x);\\[.2cm]
				f(0,x) = \mathds{1}_{\{0\}}(x)	,		\end{array}
			\right.
	\end{equation*}
 whose unique solution is given by $(t,x) \mapsto e^{2d\lambda t}\cdot \mathsf p(0,x,(\mathsf v + \lambda) t),\, (t,x) \in [0,\infty)\times \mathbb Z^d$.
\end{proof}

%^{\red Desconfio que esta definição de $Z^A_t$ se encaixe melhor lá na Sec 2 onde fala de containement. Até porque essencialmente o mesmo, desde a perspectiva dual, ja apareceu. Acho que já falamos disso. Ver lemma \ref{lem_contain}. Penso que ficaria mais curta se feita lá. Idem para o Lemma \ref {lem_good_event_Z} e de fato toda a subsec \ref {ss_bound_containment}  Outra coisa que poderia perturbar o leitor é o uso de letra latina em vez de grega para essas configurações sem cura. Não é importante, obviamente, e não queremos gastar tempo, mas para o  leitor poderia encaixar melhor, como $\Psi^A_t$. Escrevendo como conjunto e algumas expressões ficariam mais curtas também.}

We now turn to proving a bound for the containment process that will be useful in Section~\ref{s_micro_prop}. 
We need a couple of extra definitions.
First, for a fixed finite set~$A \subseteq \mathbb Z^d$ define
\begin{equation}\label{eq_def_of_TA}
	T^A:= \inf\big\{t>0: \text{ there are $x \sim y$ with~$x,y \in \Psi^A_{t-}$ and $t \in \mathcal{T}_{(x,y)}$}\big\},
\end{equation}
that is,~$T^A$ is the first time when there is a transmission mark from an infected particle towards another infected particle in~$(\Psi^A_t)$.  
Second, we define
\begin{equation}
\label{eq_def_K_t_A}
\mathcal K_t^A := \int_0^t \sum_{\{x,y\}:x\sim y} \mathds{1}\big\{\Psi^A_s(x) = \Psi^A_s(y) = 1\big\}\;\mathrm{d}s,\qquad t \ge 0.
\end{equation}

\begin{lemma}
	\label{lem_good_event_Z}
    Let~$\lambda > 0$ and~$h > 0$. 
	If~$\varepsilon > 0$ is small enough (depending on~$\lambda,h$) and~$\mathsf v$ is large enough (depending on~$\lambda,h,\varepsilon$), the following holds. Consider the interchange-and-contact process with parameters~$\lambda$ and~$\sfv$. For all~$A \subseteq B_0(\sqrt{\mathsf v})$ with~$|A| \le \mathsf v^\varepsilon$, we have
	\begin{equation*}
		\mathbb P(|\Psi^A_{h}| \le \mathsf v^{3\varepsilon},\; \mathcal K_{h}^A \le \mathsf v^{-1/4},\; T^A > h) > 1-\mathsf{v}^{-\varepsilon}.
	\end{equation*}
\end{lemma}
\begin{proof}
	The process~$\big(|\Psi^A_t|\big)_{t \ge 0}$ is stochastically dominated by a pure-birth process in which each existing individual gives birth to a new individual with rate~$2d\lambda$. For this larger process, if the initial population has~$|A|$ individuals, then the expected population size at time~$h$ is~$|A|\exp\{2d\lambda h\}$. 
    Hence, by Markov's inequality,
	\begin{equation}
		\label{eq_bound_first_bad}
		\mathbb P \big(|\Psi^A_{h}| > \mathsf v^{3\varepsilon}\big) \le \frac{|A|\exp\{2d\lambda h\}}{\mathsf v^{3\varepsilon}} \leq \frac{\exp\{2d\lambda h\}}{\mathsf v^{2\varepsilon}}.
	\end{equation}
	Before giving our next bound, we introduce notation.
    Let~$0<s_1 < s_2 < \cdots$ be the times at which~$(|\Psi^A_t|)_{t \ge 0}$ increases. For each~$j$, there is some vertex~$z_j$ such that~$z_j \notin \Psi^A_{s_j-}$ and~$z_j \in \Psi^A_{s_j}$. Next, take an enumeration~$A = \{y_1,\ldots,y_m\}$, with~$m = |A|$. We then define~$(x_j,t_j)_{j \ge 1}$ by setting
	\begin{align*}&(x_1,t_1) = (y_1,0),\;\ldots,\; (x_m,t_m) = (y_m,0),\\ &\hspace{3cm} (x_{m+1},t_{m+1}) = (z_1,s_1),\; (x_{m+2},t_{m+2}) = (z_2,s_2),\;\ldots.\end{align*}
		 In words, these are either the pairs of the form~$(x,0)$, where~$x \in A$, or the pairs of the locations and times when new infections enter the process~$(\Psi_t^A)_{t \ge 0}$.
         We then define, for each~$i < j \le |\Psi^A_{h}|$,
		\[\sigma(i,j):= \int_{t_j}^{h\vee {t_j}} \mathds{1}\{\Phi(x_i,t_i,s) \sim \Phi(x_j,t_j,s) \} \;\mathrm{d}s,\]
		that is,~$\sigma(i,j)$ is the amount of time until~$h$ that the interchange flow starting from~$(x_i,t_i)$ spends neighboring the one starting from~$(x_j,t_j)$. 
        Note that, in case~$t_j\ge h$, we have~$\sigma(i,j):=0$. 
        We then have
	\begin{equation*}
		\mathcal K^A_{h} = \sum_{1 \le i < j} \sigma(i,j).
	\end{equation*}
	We can then bound
	\begin{align*}
	\mathbb P\big(|\Psi^{A}_{h}| \le \mathsf v^{3\varepsilon},\; \mathcal{K}^A_{h} > \mathsf v^{-1/4} \big)
        &\le\!\! \sum_{1 \le i < j \le \mathsf v^{3 \varepsilon}} \mathbb P\left(\sigma(i,j) > \frac{\mathsf v^{-1/4}}{\mathsf v^{6\varepsilon}} \right) \le \mathsf v^{\tfrac14 + 6 \varepsilon}\cdot \!\!\sum_{1 \le i < j \le \mathsf v^{3 \varepsilon}} {\mathbb E[\sigma(i,j)]},
	\end{align*}
	by a union bound and Markov's inequality. By Lemma~\ref{lem_time_together} (in the worst case~$d = 1$), each expectation on the r.h.s. is smaller than~$C\sqrt{h/\mathsf v}$. We then obtain
	\begin{equation}
		\label{eq_bound_second_bad}
		\mathbb P\big(|\Psi^{A}_{h}| \le \mathsf v^{3\varepsilon},\; \mathcal{K}^A_{h} > \mathsf v^{-1/4} \big) \le C\sqrt{h}\cdot \mathsf v^{-\tfrac14 + 12 \varepsilon}.
	\end{equation}

	As the last step, we now want to bound~$\mathbb P(\mathcal{K}_{h}^A \le \mathsf v^{-1/4},\; T^A \le h)$. To do this, we first observe that the process 
	\[M_t:= \mathds{1}\{T^A \le t\} - 2\lambda \cdot \mathcal{K}^A_{t \wedge T^A},\qquad t \ge 0\]
	is a martingale, since before~$T^A$, a transmission that could trigger~$T^A$ occurs with rate 
    \[
    \lambda \cdot |\{(x,y):\;x,y \in \Psi^A_t,\; x\sim y\}|.
    \]

	Further define the stopping time~$\mathcal \kappa := \inf\{t \ge 0:\;\mathcal{K}^A_t > \mathsf v^{-1/4}\}$, and note that the stopped process $(M_{t \wedge \kappa})_{t\ge 0}$
	% \[M_{t \wedge \kappa} = \mathds{1}\{T^A \le t \wedge \kappa\} - 2\lambda \cdot \mathcal{K}^A_{t \wedge T^A \wedge \kappa},\qquad t \ge 0\]
	is also a martingale. Then,
	\begin{align*}
		0 = \mathbb E[M_0] = \mathbb E[M_{h_0 \wedge \kappa}] &= \mathbb P(T^A\!\! \le h_0 \!\wedge\! \kappa) \!-\! 2\lambda \cdot \mathbb E[\mathcal K^A_{h_0 \wedge T^A\!\! \wedge \kappa}]\ge \mathbb P(T^A\! \le h_0 \!\wedge\! \kappa) - 2\lambda \cdot \mathsf v^{\!\!-\frac14}.
	\end{align*}
	We thus obtain
	\begin{equation}
		\label{eq_bound_third_bad}
		\mathbb P(\mathcal K^A_{h_0} \le \mathsf v^{-1/4},\; T^A \le h_0) \le \mathbb P(T^A \le h_0 \wedge \kappa) \le 2\lambda \cdot \mathsf v^{-1/4}.
	\end{equation}
	To conclude, if~$\varepsilon > 0$ is small enough and $\mathsf v$ is large enough, then the r.h.s.s of~\eqref{eq_bound_first_bad},~\eqref{eq_bound_second_bad} and~\eqref{eq_bound_third_bad} are much smaller than~$\mathsf v^{-\varepsilon}$, so the proof is complete.
\end{proof}

\subsubsection{Discrepancy and spatial decoupling}
We define a second discrepancy probability.

\begin{definition}[The discrepancy probability for the interchange-and-contact process]\label{def_discr_icp}
Let~$H$ be the graphical representation for an interchange-and-contact process with parameters~$\mathsf v$ and~$\lambda > 0$, defined under some probability measure~$\mathbb P$. Given~$\ell, L \in \mathbb N$ with~$\ell < L$ and~$t > 0$, we define
	\begin{equation*}
		\mathrm{discr}^{\mathrm{icp}}_{\mathsf v, \lambda}(\ell, L, t) := \mathbb P\left(\begin{array}{c}\text{there exist }x \in \partial B_0(L),\; y \in \partial B_0(\ell) \text{ and } s,s' \in [0,t] \\\text{ with }  0 \le s < s' \le t \text{ such that } y \in \Psi(x,s,s')\end{array}\right).
	\end{equation*}
\end{definition}
Note that the event defining~$\mathrm{discr}^{\mathrm{icp}}_{\mathsf v, \lambda}(\ell, L, t)$ depends only on the Poisson processes of $H$ associated to vertices and edges inside the ball~$B_0(L)$. 
% Indeed, the event occurs if and only if there is a containment path that starts at~$\partial B_0(L) \times [0,t]$, stays inside~$B_0(L)$ and reaches~$\partial B_0(\ell) \times [0,t]$.
The following lemma is analogue to Lemma~\ref{lem_covariances_0}. 
\begin{lemma}\label{lem_covariances}
Let~$(\zeta_t)_{t \ge 0}$ be the interchange-and-contact process with parameters~$\mathsf v$ and~$\lambda$. Let~$\ell \in \mathbb N$,~$x_1,x_2 \in \mathbb Z^d$ with~$\|x_1 - x_2\| \ge 2\ell + 2$, and~$t > 0$. For~$i = 1,2$, let~$A_i$ be an event whose occurrence depends only on~$\{\zeta_s(y): (y,s) \in B_{x_i}(\ell) \times [0,t]\}$. Then,
\begin{equation*}
	|\mathrm{Cov}(\mathds{1}_{A_1},\mathds{1}_{A_2})| \le 4 \mathrm{discr}^{\mathrm{icp}}_{\mathsf v, \lambda}(\ell, \lfloor\tfrac12\|x-y\|\rfloor , t).
\end{equation*}
\end{lemma}
The proof uses Lemma~\ref{lem_contain}, and we decide to omit it because it is very similar to the proof of Lemma~\ref{lem_covariances_0}.
Our next goal is to obtain a bound for~$\mathrm{discr}^{\mathrm{icp}}_{\sfv,\lambda}(\ell,L,t)$, similarly to Lemma~\ref{lem_first_rw_discr}. This will be significantly more involved in this case, and will require preliminary bounds.

\begin{lemma}\label{seg_primeira_dom}
Let~$\Psi$ be the containment flow associated to a graphical representation of the interchange-and-contact process with parameters~$\mathsf v$ and~$\lambda$. For any~$t \ge 1$ and any~$x\in \mathbb Z^d$, we have
	\begin{align}\label{eq_first_bound_domination}
		\mathbb P\Bigl(\! x \in \bigcup_{\smash{\mathclap{s:0\le s \le t}}} \Psi(0,0,s) \Bigr) 
        &\le 8de\max(2d\mathsf v,1)\! \cdot t e^{4d\lambda t}\!\! \cdot \exp \Bigl\{\!- \frac{\|x\|}{2} \log\Bigl(\!1 + \frac{\|x\|}{2(\mathsf v + \lambda) t} \Bigr)\! \Bigr\}\\
\label{eq_second_bound_domination}
		\mathbb P\Bigl(\!x \in \ \bigcup_{\mathclap{s,s': 0 \le s < s' \le t}}\  \Psi(0,s,s') \Bigr)
        &\le 16 de^2\! \max(4d^2 \mathsf v^2\!,1) \cdot t e^{8d\lambda t}\!\! \cdot \exp\Bigl\{\!- \frac{\|x\|}{2} \log\Bigl(1 + \frac{\|x\|}{4(\mathsf v + \lambda) t} \Bigr) \!\Bigr\}.
		\end{align}

\end{lemma}
\begin{proof}
	Fix~$x\in \mathbb Z^d$ and let~$\tau_x := \inf\{t: x \in \Psi(0,0,t)\}$.  For any~$t \ge 0$ and~$\epsilon > 0$, we have
	\begin{equation*}
		\mathbb E\Bigl[\int_0^{t +\epsilon} \hspace{-2mm}\Psi(0,0,s) \;\mathrm{d}s \Bigr] \ge \epsilon \cdot \mathbb P \left(\tau_x \le t,\; x \in \Psi(0,0,s) \text{ for all } s \in [\tau_x,\tau_x + \epsilon]\right)
		\ge \epsilon \cdot \mathbb P(\tau_x \le t)\cdot e^{-2d\mathsf v\epsilon},
	\end{equation*}
	where the second inequality follows from the strong Markov property (we impose that there is no jump mark involving~$x$ in the time interval~$[\tau_x,\tau_x+\epsilon]$).
Then, rearranging and using Fubini's theorem,
	\[
		\mathbb P(\tau_x \le t) \le \frac{e^{2d\mathsf v \epsilon}}{\epsilon} \int_0^{t + \epsilon}  \mathbb P(x \in \Psi(0,0,s)) \;\mathrm{d}s.
	\]
	We take~$\epsilon := \min(1,\tfrac{1}{2d\mathsf v})$, so that~$e^{2d\mathsf v \epsilon}/\epsilon \le e\max(2d\mathsf v, 1)$. For simplicity we add the assumption that~$t \ge 1$, so that we can bound~$t + \epsilon \le 2t$. Also using Lemma~\ref{lem_primeira_dom} to bound the probability inside the integral, we obtain
	\[
		\mathbb P(\tau_x \le t) \le e\max(2d\mathsf v,1) \cdot 2t e^{4d\lambda t} \cdot \max_{0\le s \le 2t} \mathsf p(0,x,(\mathsf v + \lambda)s).
	\]
	Using Lemma~\ref{eq_bound_rw_reach}, the above maximum is bounded by
\[
	2d \exp \Bigl\{- \frac12 \|x\| \log\Bigl(1 + \frac{\|x\|}{2(\mathsf v + \lambda) t} \Bigr) \Bigr\}.
\]
This completes the proof of~\eqref{eq_first_bound_domination}.
We can obtain~\eqref{eq_second_bound_domination} from~\eqref{eq_first_bound_domination} proceeding similarly to the proof of Lemma~\ref{lem_prelim_discr_ip}. Again take~$\epsilon:=\min(1,\tfrac{1}{2d\mathsf v})$. 
 Repeating the steps leading to~\eqref{eq_sigma_A_Leb},we obtain 
	\begin{align*}
		\mathbb P\Bigl(x \in \bigcup_{\substack{s,s': 0 \le s < s' \le t}}  \Psi(0,s,s') \Bigr) \le \frac{e^{2d\mathsf v \epsilon}}{\epsilon} \cdot  \max_{s \in [0,t+\epsilon]} \mathbb P\Bigl(x \in \bigcup_{s':s \le s' \le t+\epsilon}\Psi(0,s,s') \Bigr).
	\end{align*}
By~\eqref{eq_first_bound_domination}, the maximum on the r.h.s. is smaller than
\[
8de \max(2d\sfv,1) \cdot (t+\epsilon)e^{4d\lambda (t+\epsilon)} \cdot \exp \Bigl\{-\frac12 \|x\| \log \Bigl(1+ \frac{\|x\|}{2(\sfv + \lambda)(t+\epsilon)} \Bigr) \Bigr\}.
\]
Using again that~$e^{2d\mathsf v \epsilon}/\epsilon \le e\max(2d\mathsf v, 1)$ and~$t+\epsilon \le 2t$, we complete the proof.
\end{proof}

\begin{proposition}
	\label{prop_discrepancy}
	For any~$\mathsf v > 0$,~$\lambda > 0$,~$\ell, L \in \mathbb N$ with~$\ell < L$ and~$t \ge 1$, we have
	\begin{equation}\label{eq_bound_discrepancy_icp}
		\mathrm{discr}^{\mathrm{icp}}_{\mathsf v, \lambda}(\ell, L, t) \le 64d^3e^2 \max(4d^2 \mathsf v^2,1) \cdot (9\ell L)^{d-1} \cdot t e^{8d\lambda t} \cdot \exp \Bigl\{-\frac12 (L-\ell) \log\Bigl( 1 + \frac{L-\ell}{4(\mathsf v + \lambda )t}\Bigr) \Bigr\}.
		\end{equation}
\end{proposition}
\begin{proof}
	This follows from the union bound
	\begin{equation*}\label{eq_discr_icp_fin}
		\mathrm{discr}^{\mathrm{icp}}_{\mathsf v, \lambda}(\ell, L, t) \le \sum_{x \in \partial B_0(L)} \; \sum_{y \in \partial B_0(\ell)} \mathbb P(\exists s,s' \in [0,t] \text{ with } s < s' \text{ and } y \in \Psi(x,s,s')),
	\end{equation*}
    and~\eqref{eq_second_bound_domination}, together with the estimate
        $|\partial B_0(r)|\le 2d \cdot (2r+1)^{d-1} \le 2d \cdot (3r)^{d-1}$.\qedhere
\end{proof}

\section{Lack of microscopic propagation below the mean-field threshold}
\label{s_micro_ext}

Our goal in this section is to prove the following: 
\begin{proposition}\label{prop_extended_surv}
	Let~$\lambda > 0$ and~$p \in [0,1)$ be such that~$\lambda < 1/(2dp)$. The following holds if~$\mathsf v$ is large enough. Assume that~$(\zeta_t)_{t \ge 0}$ is the interchange-and-contact process with parameters~$\sfv$ and~$\lambda$, started from a random configuration~$\zeta_0$ such that~$\xi^{\zeta_0} \sim \pi_p$. Then, the probability that there is an infection path starting at~$(0,0)$ and ending at~$\mathbb Z^d \times \{\log^3(\sfv)\}$ is smaller than~$3\exp\{-\log^2(\sfv)\}$.
\end{proposition}

In order to prove Proposition \ref{prop_extended_surv}, we will need several preliminary results. 
For~$\sfv > 0$, let
\begin{equation}
	\label{eq_def_extinction_bottom_scale}
L_0 = L_0(\sfv) := \sqrt{\sfv}\log^4(\sfv).
\end{equation}

\begin{lemma}[Up-and-down lemma]\label{lem_av_rw_bal}
    Let~$p, p' \in [0,1]$ with~$p < p'$. The following holds for~$\mathsf v > 0$ large enough. Let~$A \subseteq {\mathbb Z^d}$ be such that
    \begin{equation}\label{eq_my_density_assumption}
	    \frac{|A \cap B_x(\mathsf v^{1/10})|}{|B_x(\mathsf v^{1/10})|} < p\quad \text{for any }x \in B_0(L_0),
    \end{equation}
and let~$\Phi$ be an interchange flow with rate~$\mathsf v$.
    Fix~$u \in B_0(\tfrac12L_0)$, $\mathsf e \in \mathbb Z^d$  with $\mathsf e \sim 0$ and $T \in [\mathsf v^{-1/2},\log(\mathsf v)]$. Let
   \[\mathcal Y:= \Phi(\Phi(u,0,T)+\mathsf e, T,0),\] 
    that is,~$\mathcal Y$ is the (unique) element of~$\mathbb Z^d$ such that~$\Phi(\mathcal Y,0,T) = \Phi(u,0,T)+\mathsf e$. Then,
\begin{equation*}
    \mathbb P(\mathcal Y \in A) < p'.
\end{equation*}
\end{lemma} 

Figure~\ref{fig_up_and_down} illustrates Lemma \ref{lem_av_rw_bal}. 
The red dots on the bottom of the picture represent the set~$A$. 
We assume that the local density of~$A$ within~$B_0(L_0)$ is not larger than~$p$, meaning that inside any box of radius~$\sfv^{1/10}$ inside~$B_0(L_0)$, its density remains below~$p$. 
The blue trajectory is the path of the interchange flow started at~$u$ at time 0; we imagine that we reveal it first, from bottom to top. 
The trajectory in red is the interchange flow which at time 0 is at some point~$\mathcal Y$ and at time~$T$ is at~$v:=\Phi(u,0,T)+\mathsf e$. 
We imagine that we reveal it after the blue one, from top to bottom. 
The red path (traversed from top to bottom) is thus~$(\Phi(v,T,T-s), T-s)_{0\le s \le T}$. 
If we were to ignore the information provided when the blue path is revealed, the red path would simply have the law of a random walk.
Therefore, the probability that it lands on a point of~$A$ would not be much higher than~$p$, due the the local density assumption.
We then need to argue that this is true even when taking into account the information revealed from the blue path.

\begin{figure}[H]
\begin{center}
\setlength\fboxsep{0cm}
\setlength\fboxrule{0.01cm}
\fbox{\includegraphics[width=0.6\textwidth]{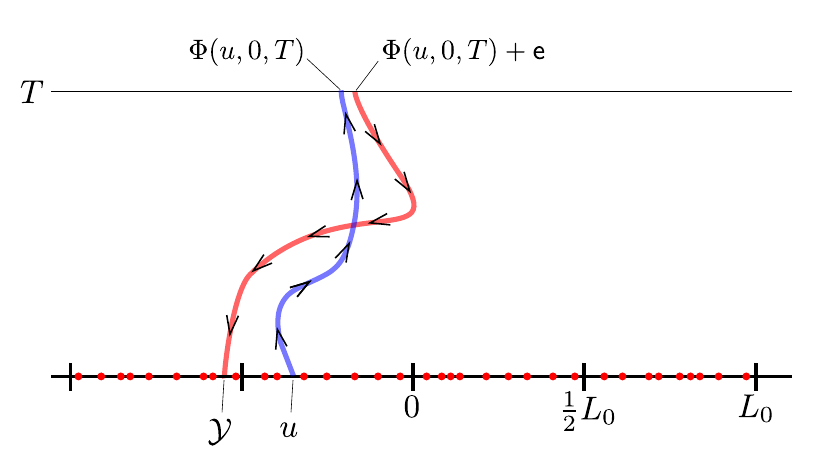}}
\end{center}
\caption{Trajectories involved in the statement of the Up-and-down lemma (Lemma~\ref{lem_av_rw_bal}). }
\label{fig_up_and_down}
\end{figure}

The proof of Lemma~\ref{lem_av_rw_bal} is not too difficult and is deferred to Appendix~\ref{appendix_rw_ip} since it requires some preparation involving some bounds for the interchange process and coupling interchange particles with independent random walks. 

For the remainder of this section, fix~$\lambda > 0$ and~$p \in [0,1)$ such that~$2dp\lambda < 1$.  As before, we denote by~$(\zeta_t)_{t \ge 0}$ the interchange-and-contact process with parameters~$\lambda$ and~$\sfv$. The initial configuration will be specified in each context; whenever it is not specified, it is irrelevant.
We will often assume that~$\sfv$ is large, and will take~$L_0=L_0(\sfv)=\sqrt{\sfv}\log^4(\sfv)$ as in~\eqref{eq_def_extinction_bottom_scale}.

\begin{definition}
    The number of infected particles in a configuration is given by
\begin{equation*}
    \mathfrak{i}(\zeta):=|\{x \in \mathbb Z^d:\; \zeta(x) = \statei\}|,\quad \zeta \in \{0,\stateh,\statei\}^{\mathbb Z^d}.
\end{equation*}
\end{definition}

We now fix~$p_0,p_1$ with $p_1 >p_0 > p$ and $2dp_1\lambda < 1$.

\begin{definition}
   We define the following sets of configurations, all depending on~$\mathsf v$:
\begin{align*}
    \Xi_\mathrm{dens}(\mathsf v)
    &:= \Bigl\{
\zeta \in \{0,\stateh,\statei\}^{\mathbb Z^d}: \exists x \in B_0(L_0):\; \frac{|\xi^\zeta \cap B_x(\mathsf v^{1/10})|}{|B_x(\mathsf v^{1/10})|} \ge p_0
    \Bigr\},\\
    \Xi_\mathrm{dist}(\mathsf v)
    &:=  \bigl\{
\zeta \in \{0,\stateh,\statei\}^{\mathbb Z^d}: \exists x \in B_0(\tfrac12L_0)^c:\; \zeta(x)= \statei
    \bigr\},\\
\Xi_{\mathrm{inf}}(\mathsf v)
    &:= \bigl\{ \zeta \in \{0,\stateh,\statei\}^{\mathbb Z^d}: \;\mathfrak{i}(\zeta) > \log^3(\sfv) \bigr\}.
\end{align*}

\end{definition}

\begin{lemma}\label{lem_one_step_ext}
The following holds if~$\mathsf v$ is large enough. Assume that~$(\zeta_t)_{t \ge 0}$ starts from a deterministic configuration~$\zeta_0  \notin \Xi_{\mathrm{dens}} \cup \Xi_{\mathrm{dist}}\cup \Xi_{\mathrm{inf}}$ which contains at least one {infected particle}. Let
\[
\sigma:=  \inf\{t:\exists x: \zeta_{t-}(x)=\statei,\; t \in \mathcal R_x \cup (\cup_{y\sim x}\mathcal T_{(x,y)}) \},
\]
that is, the first time an infected particle recovers or attempts to transmit the infection. Then,
\begin{align}
    \label{eq_increment_minus}
    \mathbb P(\mathfrak{i}(\zeta_\sigma)=\mathfrak{i}(\zeta_0)- 1)
        &= \frac{1}{1+2d\lambda};\\
    \label{eq_increment_plus}
    \mathbb P(\mathfrak{i}(\zeta_\sigma)=\mathfrak{i}(\zeta_0)+ 1)
        &\le \frac{2d\lambda}{1+2d\lambda}p_1.
    \end{align}
\end{lemma}
\begin{proof}
Let~$\mathcal B$ be the event that the stopping time~$\sigma$ is triggered by an infection arrow, that is,~$\mathcal B$ is the event that there are~$\mathcal X,\mathcal Y \in \mathbb Z^d$ with~$\mathcal X \sim \mathcal Y$ such that~$\zeta_{\sigma-}(\mathcal X)=\statei$ and~$\sigma \in  \mathcal T_{(\mathcal X,\mathcal Y)}$. 
	Note that~$\mathbb P(\mathcal B) = \frac{2d\lambda}{1+2d\lambda}$, and we have~$\mathfrak{i}(\zeta_\sigma) \in \{\mathfrak{i}(\zeta_0),\mathfrak{i}(\zeta_0)+1\}$ on~$\mathcal B$, and~$\mathfrak{i}(\zeta_{\sigma}) =  \mathfrak{i}(\zeta_0)-1$ on~$\mathcal B^c$. This already proves~\eqref{eq_increment_minus}. 
    Next, we observe that
    \begin{equation}
        \{\mathfrak{i}(\zeta_\sigma)=\mathfrak{i}(\zeta_0)+ 1\} = \mathcal B \cap \{\zeta_{\sigma-}(\mathcal Y) = \stateh\}.\label{eq_my_first_ctnmt}
    \end{equation}
    
    On~$\mathcal B$, let~$\mathcal X_0$ and~$\mathcal Y_0$ be the (unique) points of~$\mathbb Z^d$ such that
    $\Phi(\mathcal X_0,0,\sigma) = \mathcal X,\ \Phi(\mathcal Y_0,0,\sigma) = \mathcal Y$.
    Define
\[
A_{\tinystateh}:=\{x:\zeta_0(x)=\stateh\},\quad A_{\tinystatei}:=\{x:\zeta_0(x)=\statei\},\quad A:= A_{\tinystateh} \cup A_{\tinystatei}
\]
and note that
\begin{equation}\label{eq_my_second_ctnmt}
        \mathcal B \cap \{\zeta_{\sigma-}(\mathcal Y) = \stateh\} \subseteq \mathcal B \cap \{\zeta_{\sigma-}(\mathcal Y) \in \{\stateh,\statei\}\} 
 =\mathcal B \cap \{\mathcal Y_0 \in A\}.
    \end{equation}

    Recalling that each infected particle recovers with rate one and attempts to transmit the infection with rate~$\lambda$ to each neighbor, we make the following observations:
\begin{itemize}
    \item $\sigma$ follows the exponential distribution with parameter~$\mathfrak{i}(\zeta_0)\cdot (1+2d\lambda)$;
    \item $\sigma$ is independent of~$\mathcal B$ and of~$\mathds{1}_{\mathcal B}\cdot (\mathcal X,\mathcal X_0,\mathcal Y,\mathcal Y_0)$;
    \item the interchange jumps~$(\mathcal J_{\{x,y\}}:x,y\in\mathbb Z^d,\;x\sim y)$ are independent of~$\sigma$ and of~$\mathcal B$.
\end{itemize}
    
%Moreover,
%\begin{equation*}
%    \mathbb P(\mathcal B \cap \{\mathcal X_0 = x,\; \mathcal Y = \mathcal X + \mathsf e\}) = \frac{2d\lambda}{1+2d\lambda} \cdot \frac{1}{|I_0|}\cdot \frac{1}{2d}\quad \text{for any } x \in I_0,\; \mathsf e \sim 0.
%\end{equation*}

By Lemma~\ref{lem_av_rw_bal} (using $p=p_0$ and $p'=p_1$), which is applicable with our current choice of~$A$ by the assumption that~$\zeta_0 \notin\Xi_{\mathrm{dens}}$, for any~$x\in A_{\smash{\tinystatei}}$ and~$\mathsf e \in \mathbb Z^d$ with~$\mathsf e \sim 0$, we have
\begin{equation}\label{eq_nice_bound_B}
\mathbb P( \mathcal Y_0 \in A \mid \mathcal B \cap \{\mathcal X_0 = x,\; \mathcal Y = \mathcal X + \mathsf e,\;\sigma \in [\mathsf v^{-1/2},\log(\mathsf v)]\}  ) \le p_1
\end{equation}
if~$\mathsf v$ is large.
We are now ready to conclude. Using \eqref{eq_my_first_ctnmt} and \eqref{eq_my_second_ctnmt} we bound
\begin{align*}
\mathbb P(\mathfrak{i}(\zeta_\sigma)=\mathfrak{i}(\zeta_0)+ 1)
&\le \mathbb P(\mathcal B \cap \{\mathcal Y_0 \in A\}) 
\\
&\le \mathbb P(\sigma \notin [\mathsf v^{-1/2},\log(\mathsf v)]) + \mathbb P(\mathcal B \cap \{\mathcal Y_0 \in A,\;\sigma \in [\mathsf v^{-1/2},\log(\mathsf v)]\}).
\end{align*}
	We have~$\sigma \sim \mathrm{Exp}(\mathfrak{i}(\zeta_0) \cdot (1+2d\lambda))$. From the assumptions on~$\zeta_0$, we have~$1 \le \mathfrak{i}(\zeta_0) \le \log^3(\sfv)$, so~$\mathsf v^{-1/2} \ll (\mathfrak{i}(\zeta_0)\cdot (1+2d\lambda))^{-1} \ll \log(\mathsf v)$. Consequently, the first probability on the r.h.s. above can be made as small as desired by taking~$\mathsf v$ large. We bound the second probability as follows: 
\begin{align*}
	&\sum_{\mathclap{x \in A_{\tinystatei}}}\;\; \sum_{\mathsf e \sim 0} \mathbb P( \mathcal B \cap \{\mathcal Y_0 \in A,\;\mathcal X_0 = x,\; \mathcal Y = \mathcal X + \mathsf e,\;\sigma \in [\mathsf v^{-1/2},\log(\mathsf v)]\}  ) \\[-2mm]
	&\le p_1 \cdot \sum_{x \in A_{\tinystatei}}\; \sum_{\mathsf e \sim 0} \mathbb P( \mathcal B \cap \{\mathcal X_0 = x,\; \mathcal Y = \mathcal X + \mathsf e,\;\sigma \in [\mathsf v^{-1/2},\log(\mathsf v)]\}  ) \\[-2mm]
    &= p_1 \cdot \mathbb P(\mathcal B \cap \{\sigma \in [\mathsf v^{-1/2},\log(\mathsf v)]\})  \le p_1 \cdot \mathbb P(\mathcal B) = p_1 \cdot \frac{2d\lambda}{1+2d\lambda},
\end{align*}
where the first inequality follows from~\eqref{eq_nice_bound_B}. The proof of~\eqref{eq_increment_plus} is now complete.
\end{proof}

 \begin{lemma}
 \label{lem_count_i_martingale}
	The following holds if~$\mathsf v$ is large enough. Assume that~$(\zeta_t)_{t \ge 0}$ starts from a deterministic configuration $\zeta_0 \notin \Xi_{\mathrm{dens}} \cup \Xi_{\mathrm{dist}}$ with~$\mathfrak{i}(\zeta_0)=1$. Then,  we have
\begin{equation*}
	\mathbb P(\mathfrak{i}(\zeta_{\log^3(\sfv)}) \neq 0,\; \zeta_t \notin \Xi_{\mathrm{dens}} \cup \Xi_{\mathrm{dist}} \text{ for all }t \in [0,\log^3(\sfv)]) \le \exp\{-\log^2(\sfv)\}.
\end{equation*}
\end{lemma}
\begin{proof} Let~$T:=\log^3(\sfv)$.
Let~$\sigma_0 \equiv 0$, and as in the proof of Lemma~\ref{lem_one_step_ext}, define
\[
\sigma_1:= \inf\{t:\exists x: \zeta_{t-}(x)=\statei,\; t \in \mathcal R_x \cup (\cup_{y\sim x}\mathcal T_{(x,y)}) \}.
\]
Recursively, we define~$\sigma_{n+1}$ by setting~$\sigma_{n+1}=\infty$ on~$\{\sigma_n = \infty\}$ and
\[
\sigma_{n+1}:= \inf\{t>\sigma_n:\exists x: \zeta_{t-}(x)=\statei,\; t \in \mathcal R_x \cup (\cup_{y\sim x}\mathcal T_{(x,y)}) \}
\]
on~$\{\sigma_n < \infty\}$. We now define three bad events:
\begin{align*}
\mathcal A_1 &:= \{T < \sigma_{\lfloor T/2 \rfloor} < \infty\},\\
\mathcal A_2 &:= \{\sigma_{\lfloor T/2\rfloor} < \infty,\; \zeta_{\sigma_n} \notin \Xi_\mathrm{dens} \cup \Xi_\mathrm{dist} \cup \Xi_{\mathrm{inf}} \text{ for } n = 0,\ldots, \lfloor T/2\rfloor \},\\
\mathcal A_3 &:= \{\exists n^*: \; \zeta_{\sigma_{n^*}} \in \Xi_\mathrm{inf} \text{ and } \zeta_{\sigma_n} \notin \Xi_{\mathrm{dens}} \cup \Xi_{\mathrm{dist}} \cup \Xi_{\mathrm{inf}} \text{ for } n =0,\ldots, n^*-1\}.
\end{align*}
A moment's thought reveals that
\begin{equation*}
\{\mathfrak{i}(\zeta_{T}) \neq 0,\; \zeta_t \notin \Xi_{\mathrm{dens}} \cup \Xi_{\mathrm{dist}} \text{ for all }t \in [0,T]\} \subseteq \mathcal A_1 \cup \mathcal A_2 \cup \mathcal A_3.
\end{equation*}

We now proceed to give upper bounds for the probabilities of the three bad events.\smallskip

\textbf{Bound on~$\mathbb P(\mathcal A_1)$.} Similarly to what was observed in the proof of Lemma~\ref{lem_one_step_ext}, conditionally on the event~$\{\sigma_n < \infty,\; \mathfrak{i}(\zeta_{\sigma_n}) > 0\}$, the law of~$\sigma_{n+1}-\sigma_n$ is exponential with parameter~$(2d\lambda + 1)\cdot \mathfrak{i}(\zeta_{\sigma_n})$. More precisely, letting~$(\mathcal F_t)_{t\ge 0}$ be the filtration generated by the graphical representation, we have 
\[\text{on }\{\sigma_n\! < \infty,\;\mathfrak{i}(\zeta_{\sigma_n}\!)\! > \!0\},\ \ \
\mathbb P(\sigma_{n+1}-\sigma_n > x \mid \mathcal F_{\sigma_n})  = \exp\{-(2d\lambda \!+\!1)\cdot \mathfrak{i}(\zeta_{\sigma_n})\cdot x\},\ \ \ x > 0.
\]
Since~$(2d\lambda + 1) \cdot \mathfrak{i}(\zeta_{\sigma_n}) \ge 1$ when~$\mathfrak{i}(\zeta_{\sigma_n}) > 0$, we can stochastically dominate~$\mathrm{Exp}((2d\lambda + 1) \cdot \mathfrak{i}(\zeta_{\sigma_n}))$ by~$\mathrm{Exp}(1)$ on this event, so
\vspace{-2mm}
\[\text{on }\{\sigma_n < \infty,\;\mathfrak{i}(\zeta_{\sigma_n}) > 0\},\quad 
\mathbb E\bigl[e^{\theta \cdot (\sigma_{n+1}-\sigma_n)} \mid \mathcal F_{\sigma_{n}}\bigr]
\le \frac{1}{1-\theta},\quad \theta \in (0,1).
\]
Noting that for~$n \ge 1$ we have~$\{\sigma_n < \infty\} \subseteq \{\sigma_{n-1}<\infty,\;\mathfrak{i}(\zeta_{\sigma_{n-1}}) > 0\}$, we can bound
\begin{align*}
\mathbb E\bigl[e^{\theta \sigma_n} \!\cdot \mathds{1}\{\sigma_n < \infty\}\bigr]
&\le \mathbb E\Bigl[ e^{\theta \sigma_{n-1}} \!\cdot \mathds{1}\{\sigma_{n-1}<\infty,\;\mathfrak{i}(\zeta_{\sigma_{n-1}})>0\} \cdot \mathbb E\bigl[e^{\theta(\sigma_{n}-\sigma_{n-1})} \mid \mathcal F_{\sigma_{n-1}} \bigr]\Bigr]\\
&\le \frac{1}{1-\theta}\cdot \mathbb E\bigl[ e^{\theta \sigma_{n-1}} \cdot \mathds{1}\{\sigma_{n-1}<\infty,\;\mathfrak{i}(\zeta_{\sigma_{n-1}})>0\}\bigr].
\end{align*}
Iterating this gives
$\mathbb E[\exp\{\theta \sigma_n\} \cdot \mathds{1}\{\sigma_n < \infty\}] \le (\frac{1}{1-\theta})^n$.
Then,
\begin{align*}
    \mathbb P(\mathcal A_1)
    &\le \mathbb E\Bigl[ \frac{\exp\{\theta \cdot \sigma_{\lfloor T/2\rfloor}\}}{\exp\{\theta \cdot T\}} \cdot \mathds{1}_{\mathcal A_1}\Bigr]
    \le \exp\{-\theta \cdot T\}\cdot \Bigl(\frac{1}{1-\theta}\Bigr)^{\lfloor T/2\rfloor}\hspace{-2mm}
    \le \exp\Bigl\{-\Bigl(\theta -\frac{1}{2}\log\Bigl(\frac{1}{1-\theta} \Bigr) \Bigr)T \Bigr\}. 
\end{align*}
By taking~$\theta=1/2$, this gives
\begin{equation}\label{eq_bound_with_delta_1}
	\mathbb P(\mathcal A_1) \le \exp\Bigl\{-\frac{1-\log(2)}{2}\cdot T\Bigr\} = \exp\Bigl\{-\frac{1-\log(2)}{2}\cdot \log^3(\sfv) \Bigr\}.
\end{equation}

\textbf{Bound on~$\mathbb P(\mathcal A_2)$.}
% \iffalse 
% Since~$2d\lambda p_1 <1$, we can  choose~$\delta>0$ small enough that
% \[e^{\delta} \cdot \frac{2d\lambda p_1}{1+2d\lambda} +  \frac{2d\lambda (1-p_1)}{1+2d\lambda} +  e^{-\delta} \cdot \frac{1}{1+2d\lambda} < 1.\] 
% We also take~$\delta' > 0$ small, so that
% \[
% e^{\delta'} \cdot \left( e^{\delta} \cdot \frac{2d\lambda p_1}{1+2d\lambda} +  \frac{2d\lambda (1-p_1)}{1+2d\lambda} +  e^{-\delta} \cdot \frac{1}{1+2d\lambda} \right) < 1.
% \]
% By~\eqref{eq_increment_minus} and~\eqref{eq_increment_plus}, which are applicable since~$\zeta_0 \notin \Xi_{\mathrm{dens}} \cup \Xi_{\mathrm{dist}} \cup \Xi_{\mathrm{inf}}$,
% this then gives
% \[
% \mathbb E[\exp\{ \delta \cdot [\mathfrak{i}(\zeta_{\sigma_1}) - \mathfrak{i}(\zeta_{\sigma_0})] + \delta' \}] < 1.
% \]
% Similarly, 
% \begin{align*}
%     \text{on }\{\zeta_{\sigma_n} \notin \Xi_{\mathrm{dens}} \cup \Xi_{\mathrm{dist}} \cup \Xi_{\mathrm{inf}}\},\qquad \mathbb E[\exp\{\delta \cdot [\mathfrak{i}(\zeta_{\sigma_{n+1}}) - \mathfrak{i}(\zeta_{\sigma_{n}})] +\delta' \} \mid \mathcal F_{\sigma_n}] < 1. 
% \end{align*}
% This shows that, letting~$\nu:=\inf\{m:\zeta_{\sigma_m} \in \Xi_\mathrm{dens} \cup \Xi_{\mathrm{dist}} \cup \Xi_{\mathrm{inf}}\}$, the process
% \begin{equation*}
%     M_n := \exp\{ \delta \cdot \mathfrak{i}(\zeta_{\sigma_{n\wedge \nu}}) + \delta' \cdot (n \wedge \nu)   \},\quad n \in \mathbb N_0
% \end{equation*}
% is a supermartingale with respect to the filtration~$(\mathcal F_{\sigma_n})_{n \in \mathbb N_0}$.
% \fi
Since on $\mathcal A_2$ we have~$\zeta_0 \notin \Xi_{\mathrm{dens}} \cup \Xi_{\mathrm{dist}} \cup \Xi_{\mathrm{inf}}$, we can use~\eqref{eq_increment_minus} and~\eqref{eq_increment_plus} to get a uniform estimate on the moment generating function of the random variable $\mathfrak{i}(\zeta_{\sigma_1}) - \mathfrak{i}(\zeta_{\sigma_0})$:
\begin{equation*}
\frac{\mathrm d}{\mathrm d t}\, \mathbb E\Bigl[e^{t [\mathfrak{i}(\zeta_{\sigma_1})\! - \mathfrak{i}(\zeta_{\sigma_0})]}\Bigr]
    = e^{t}  \mathbb P (\mathfrak{i}(\zeta_{\sigma_1}) - \mathfrak{i}(\zeta_{\sigma_0}) = 1) - e^{-t} \frac{1}{1\!+\!2d\lambda}
    \le e^{t} \frac{2d\lambda p_1}{1\!+\!2d\lambda} - e^{-t} \frac{1}{1\!+\!2d\lambda}.
\end{equation*}
Since $2d\lambda p_1 < 1$, if we take $0 < t < \delta := \frac12 \log (\frac{1}{2d\lambda p_1})$ it follows that the derivative above is negative, implying $\mathbb E[e^{\delta [\mathfrak{i}(\zeta_{\sigma_1}) - \mathfrak{i}(\zeta_{\sigma_0})]}] < \mathbb E[e^{0}] = 1$. Moreover, taking $\delta' > 0$ small, we obtain
\[
\mathbb E[\exp\{ \delta \cdot [\mathfrak{i}(\zeta_{\sigma_1}) - \mathfrak{i}(\zeta_{\sigma_0})] + \delta' \}] < 1.
\]
Similarly, 
\begin{align*}
    \text{on }\{\zeta_{\sigma_n} \notin \Xi_{\mathrm{dens}} \cup \Xi_{\mathrm{dist}} \cup \Xi_{\mathrm{inf}}\},\qquad \mathbb E[\exp\{\delta \cdot [\mathfrak{i}(\zeta_{\sigma_{n+1}}) - \mathfrak{i}(\zeta_{\sigma_{n}})] +\delta' \} \mid \mathcal F_{\sigma_n}] < 1. 
\end{align*}
This shows that, letting~$\nu:=\inf\{m:\zeta_{\sigma_m} \in \Xi_\mathrm{dens} \cup \Xi_{\mathrm{dist}} \cup \Xi_{\mathrm{inf}}\}$, the process
\begin{equation*}
    M_n := \exp\{ \delta \cdot \mathfrak{i}(\zeta_{\sigma_{n\wedge \nu}}) + \delta' \cdot (n \wedge \nu)   \},\quad n \in \mathbb N_0
\end{equation*}
is a supermartingale with respect to the filtration~$(\mathcal F_{\sigma_n})_{n \in \mathbb N_0}$.

Let~$\bar{n}:=\lfloor T/2 \rfloor$. On~$\mathcal A_2$, we have~$\sigma_{\bar n  \wedge \nu} = \sigma_{\bar n}$, so~$M_{\bar n} = \exp\{\delta \cdot \mathfrak{i}(\zeta_{\sigma_{\bar n}})+ \delta' \cdot  \bar n\} \ge e^{\delta' \cdot \bar n}$. Then,
\begin{align}
	e^{\delta' \cdot \bar n} \cdot \mathbb P(\mathcal A_2) 
    &\le \mathbb E[M_{\bar n} \cdot \mathds{1}_{\mathcal A_2}] \le \mathbb E[M_{\bar n}] \le \mathbb E[M_0] = e^{\delta \cdot \mathfrak{i}(\zeta_0)} = e^{\delta},\quad \text{which gives}\nonumber\\
    \label{eq_bound_with_delta_2}
	\mathbb P(\mathcal A_2) 
    &\le \exp\{\delta -\delta' \cdot \bar n \} = \exp\{ \delta -\delta' \cdot  \lfloor \log^3(\sfv)/2\rfloor  \}.
\end{align}
% \begin{equation}\label{eq_bound_with_delta_2}
% 	\mathbb P(\mathcal A_2) \le \exp\{\delta -\delta' \cdot \bar n \} = \exp\{ \delta -\delta' \cdot  \lfloor \log^3(\sfv)/2\rfloor  \}.
% \end{equation}

\textbf{Bound on~$\mathbb P(\mathcal A_3)$.} 
Let~$\nu':= \inf\{m: \zeta_{\sigma_m} \in \Xi_{\mathrm{inf}}\}$. On~$\mathcal A_3$, we have~$\nu' < \infty$,~$\nu' \le \nu$, and~$\mathfrak{i}(\zeta_{\nu'}) > \log^3(\sfv)$, so~$M_{\nu'} \ge \exp\{\delta \cdot \log^3(\sfv)\}$. Hence, using the optional stopping theorem,
\begin{align*}
	\exp\{\delta \cdot \log^3(\sfv)\} \cdot \mathbb P(\mathcal A_3) \le \mathbb E[M_{\nu'} \cdot \mathds{1}_{\mathcal A_3}] \le \mathbb E[M_{\nu'} \cdot \mathds{1}_{\{\nu' < \infty\}}] \le \mathbb E[M_0] = e^{\delta},
\end{align*}
which gives
\begin{equation}\label{eq_bound_with_delta_3}
	\mathbb P(\mathcal A_3) \le \exp\{-\delta(\log^3(\sfv)-1)\}.
\end{equation}

The result now follows from~\eqref{eq_bound_with_delta_1},~\eqref{eq_bound_with_delta_2} and~\eqref{eq_bound_with_delta_3}, by taking~$\mathsf v$ large enough.
\end{proof}

\begin{proof}[Proof of Proposition~\ref{prop_extended_surv}]
Letting ~$T:=\log^3(\sfv)$, we bound $\mathbb P(\mathfrak{i}(\zeta_{T}) \neq 0)$ by the sum
    \begin{equation}\label{eq_three_terms_path}
          \mathbb P(\mathfrak{i}(\zeta_{T}) \neq 0,\; \zeta_t \notin \Xi_{\mathrm{dens}} \cup \Xi_{\mathrm{dist}} \forall t \in [0,T]) +\mathbb P(\exists t \le T: \zeta_t \in \Xi_{\mathrm{dens}}) + \mathbb P(\exists t \le T: \zeta_t \in \Xi_{\mathrm{dist}}).
    \end{equation}
	By Lemma \ref{lem_count_i_martingale}, the first term on \eqref{eq_three_terms_path} is smaller than~$\exp\{-\log^2(\sfv)\}$. 
    The second term is
\begin{align*}
\mathbb P(\exists t \le T: \zeta_t \in \Xi_{\mathrm{dens}}) = \int_{\{0,1\}^{\mathbb Z^d}}g^\uparrow(\sfv^{1/10},\sqrt{\sfv}\log^4(\sfv),\log^3(\sfv), p_0,\zeta_0)\;  \pi_p(\mathrm{d}\zeta_0),
\end{align*}
recalling Definition~\ref{def_gs}. By Lemma~\ref{lem_integral_gs}, the second term is smaller than
\[
(2\sqrt{\sfv}\log^4(\sfv)+1)^d \cdot (e(2\sfv^{\frac{1}{10}}\!+2)^d \cdot \log^3(\sfv) + e) \cdot \exp \bigl\{- 2(2 \sfv^{\frac{1}{10}}\!+1)^d \cdot (p_0-p)^2\bigr\} \ll e^{-\log^2(\sfv)}.
\]
Finally, recalling the definition of the containment flow (Definition~\ref{def_containment_try}), we have
\begin{align*}
\mathbb P(\exists t \le T: \zeta_t \in \Xi_{\mathrm{dist}})
\le \sum_{\substack{x \in \partial B_0(\lfloor \sqrt{\sfv} \log^4(\sfv)\rfloor})} \mathbb P(x \in \Psi(0,0,s) \text{ for some }s \le \log^3(\sfv)).
\end{align*}
By Lemma~\ref{seg_primeira_dom}, all terms in the sum of the r.h.s. are smaller than
\[
8de \max(2d\sfv,1) \cdot \log^3(\sfv)e^{4d\lambda \log^3(\sfv)}\cdot \exp\Bigl\{-\frac12 \lfloor \sqrt{\sfv}\log^4(\sfv) \rfloor \cdot \log \Bigl(1 + \frac{\lfloor \sqrt{\sfv}\log^4(\sfv) \rfloor}{2(\sfv + \lambda)\log^3(\sfv)} \Bigr) \Bigr\}.
\]
Then,~$\mathbb P(\exists t \le T: \zeta_t \in \Xi_\mathrm{dist})$ is smaller than~$|B_0(\sqrt{\sfv}\log^4(\sfv))|$ times the expression above. When~$\sfv$ is large enough, this is smaller than~$\exp\{-\log^2(\sfv)\}$, completing the proof.
\end{proof}

\section{Proof of Theorem~\ref{thm_main}: extinction}\label{s_proof_extinction}

\subsection{Renormalization scheme}

\subsubsection{Boxes and half-crossings}
We will apply the same renormalization scheme as in Section~2 of~\cite{HUVV}, involving \emph{half-crossings} of space-time boxes; let us briefly explain it.

We want to discuss events involving infection paths, so we fix a realization of the graphical construction~$H$ of the interchange-and-contact process and an initial configuration~$\xi_0$ of the interchange process.

Let~$\mathsf{e}_1,\ldots,\mathsf e_d$ denote the vectors in the canonical basis of~$\mathbb R^d$, and let~$\langle \cdot,\cdot \rangle$ denote the inner product of~$\mathbb R^d$.
\begin{definition}
    Let~$x=(x^1,\ldots,x^d) \in \mathbb Z^d$,~$\ell \in \mathbb N$,~$t \ge 0$ and~$h > 0$. Let~$\mathcal Q:=B_x(\ell) \times [t,t+h]$.
    \begin{itemize}
            \item A \emph{temporal half-crossing of~$\mathcal Q$} is an infection path~$\gamma:[t,t+\tfrac{h}{2}]\to \mathbb Z^d$ such that~$\gamma(s) \in B_x(\ell)$ for all~$s$.
            \item A \emph{spatial half-crossing of~$\mathcal Q$ in the direction~$i$} is an infection path~$\gamma:[s_1,s_2]\to \mathbb Z^d$ such that~$(\gamma(s),s) \in \mathcal Q$ for all~$s$, and 
            \begin{align*}
                \text{either }&\langle \gamma(s_1), \mathsf e_i\rangle = x^i,\quad \langle \gamma(s_2), \mathsf e_i\rangle = x^i + \ell,\quad \langle \gamma(s), \mathsf e_i\rangle \in [x^i,x^i+\ell] \; \forall s;\\
                \text{or }&\langle \gamma(s_1), \mathsf e_i\rangle = x^i+\ell,\quad \langle \gamma(s_2), \mathsf e_i\rangle = x^i,\quad \langle \gamma(s), \mathsf e_i\rangle \in [x^i,x^i+\ell] \; \forall s;\\
                \text{or }&\langle \gamma(s_1), \mathsf e_i\rangle = x^i,\quad \langle \gamma(s_2), \mathsf e_i\rangle = x^i-\ell,\quad \langle \gamma(s), \mathsf e_i\rangle \in [x^i-\ell,x^i] \; \forall s;\\
                \text{or }&\langle \gamma(s_1), \mathsf e_i\rangle = x^i-\ell,\quad \langle \gamma(s_2), \mathsf e_i\rangle = x^i,\quad \langle \gamma(s), \mathsf e_i\rangle \in [x^i-\ell,x^i] \; \forall s.
            \end{align*}
            \item A \emph{half-crossing of~$\mathcal Q$} is any of the above (temporal half-crossing or spatial half-crossing in any direction). If it exists, we say that~$\mathcal Q$ \emph{is half-crossed}.
    \end{itemize}
\end{definition}

\textbf{Renormalization scales.} We take
\begin{equation*}
    L_N:=128^N \cdot L_0=128^N \cdot \sqrt{\sfv}\log^4(\sfv),\quad h_N:=128^N \cdot 2\log^3(\sfv),\qquad N \in \mathbb N_0.
\end{equation*}
The reason for the constant $128$ will be given in Remark~\ref{rem:constants_extinction} below.
We write
\begin{equation*}
    \mathcal Q_N(x,t):= B_x(L_N) \times [t,t+h_N],\quad x \in \mathbb Z^d,\; t \ge 0.
\end{equation*}
The following is a particular case of Lemma~2.5 of~\cite{HUVV} (with slightly different notation and weaker constants), by taking $\alpha=\beta=128$ in equation~(2.5) therein.

    \begin{lemma}[Cascading half-crossings] \label{lem_cascade}
    Let~$N \in \mathbb N$,~$x \in \mathbb Z^d$ and~$t \ge 0$. There exists an integer~$k \le 255^{2d}(2d+1)$ and~$(x_1,s_1),(y_1,t_1),\ldots,(x_k,s_k),(y_k,t_k) \in \mathcal Q_N(x,t)$ with the following properties:
    \begin{itemize}
        \item $\mathcal Q_{N\!-\!1}(x_1,s_1), \mathcal Q_{N\!-\!1}(y_1,t_1), \ldots, \mathcal Q_{N\!-\!1}(x_k,s_k),\mathcal Q_{N-1}(y_k,t_k)$ are contained in~$\mathcal Q_N(x,t)$;
        \item for all~$i$, we have either~$\|x_i-y_i\| \ge 4L_{N-1}$ or~$|s_i-t_i| \ge 2h_{N-1}$;
        \item if~$\mathcal Q_N(x,t)$ is half-crossed, then there is $i$ such that~$\mathcal Q_{N-1}(x_i,s_i),$ and $\mathcal Q_{N-1}(y_i,t_i)$ are both half-crossed.
    \end{itemize}
\end{lemma}

%{\color{red}
%Motivo da expressão $255^{2d}(2d+1)$ acima: o nosso outro artigo diz, no Lemma 2.5:
%\[(2\alpha - 1)^{2d} + 2d \cdot ((2\beta-1)\cdot (2\alpha - 1)^{d-1})^2.\]
%Quando~$\alpha =\beta$, isso dá
%\[
%(2\alpha-1)^{2d} \cdot (2d+1). 
%\]
%Aqui temos~$\alpha = \beta = 128$.\\
%}

\textbf{Choice of constants and notation:} For the rest of this section, fix~$\lambda > 0$ and $p \in (0,1)$ with~$2dp\lambda < 1$. These are the values of~$\lambda$ and~$p$ for which we will prove~\eqref{eq_main_ext}. Then, fix~$p_0$ slightly larger than~$p$ so that~$2dp_0 \lambda < 1$. Also define
\begin{equation}
    p_N:= \big(1-2^{-N}\big)p + 2^{-N}p_0,\quad N \in \mathbb N.
\end{equation}
We denote by~$(\zeta_t)_{t\ge 0}$ the interchange-and-contact process with parameters~$\lambda$ and~$\sfv$. The initial configuration will be specified in each context; whenever it is not specified, it is irrelevant.\smallskip

Our estimates from the previous sections will readily give us:
\begin{lemma}\label{lem_base_extinction}
    The following holds for~$\sfv$ large enough. If~$\xi_0$ is stochastically dominated by~$\pi_{p_0}$, then, for any~$x$ and~$t$, the probability that~$\mathcal Q_0(x,t)$ is half-crossed is smaller than~$e^{-\log^{3/2}(\sfv)}$.
\end{lemma}
\begin{proof}
It suffices to prove the statement for~$(x,t)=(0,0)$. To see this, note that the event of half-crossing of a space-time box only depends on the realization of the interchange process and the recovery marks and transmission arrows of the graphical representation, all inside the box. The graphical representation is invariant under space-time shifts, and the assumption that~$\xi_0$ is stochastically dominated by~$\pi_{p_0}$ implies that~$\xi_t$ is stochastically dominated by the same measure, for all~$t$.

So we proceed with~$(x,t)=(0,0)$. The probability of a temporal half-crossing of~$\mathcal Q_0(0,0)$ is smaller than the probability that for some~$y \in B_0(L_0)$, there is an infection path starting at~$y$ at time 0 and reaching time~$h_0/2=\log^3(\sfv)$. By Proposition~\ref{prop_extended_surv} and a union bound, this probability is smaller than~$|B_0(L_0)| \cdot 3 \exp\{-\log^2(\sfv)\} \ll \exp\{-\log^{3/2}(\sfv)\}$.

Let us now treat spatial-half crossings. Recall Definition~\ref{def_containment_try} of containment flow. The probability of a spatial half-crossing of~$\mathcal Q_0(0,0)$ is bounded from above by
\begin{align*}
    \sum_{\substack{x \in B_0(L_0),\\ y:\|x-y\|=\lfloor L_0\rfloor}}\mathbb P \biggl(y \in \bigcup_{s,s':0\le s < s' \le h_0} \Psi(x,s,s') \biggr).
\end{align*}
By a union bound and~\eqref{eq_second_bound_domination}, this is smaller than
\begin{align*}
 (2L_0+1)^{2d} \cdot   64d^3e^2 \sfv^2   \cdot h_0e^{8d\lambda h_0} \cdot \exp\Bigl\{-\frac12 \lfloor L_0 \rfloor \log \Bigl( 1+ \frac{\lfloor L_0 \rfloor}{4(\sfv + \lambda)h_0}\Bigr) \Bigr\} 
\end{align*}
Recalling that~$L_0=\sqrt{\sfv}\log^4(\sfv)$ and~$h_0=2\log^3(\sfv)$, it is easily checked that this is much smaller than~$\exp\{-\log^{3/2}(\sfv)\}$ when~$\sfv$ is large enough, completing the proof.
\end{proof}

\subsection{Induction step}
Let
\begin{align*}
    \delta_0 &:= \exp[- \log^{3/2} (\sfv)], \\
    \delta_N &:=(255^{2d}(4d+2))^{-N-1},\quad N \ge 1.
\end{align*} 
We are interested in establishing the following, for~$\sfv$ large enough (uniformly over~$N$):\\[.1cm]
\noindent \textbf{Half-crossing estimate at scale~$N$ ($\mathrm{HC}_N$)}:
\begin{equation}
    \label{eq_hce}
    \tag{$\mathrm{HC}_N$}
    \begin{split}
    \xi^{\zeta_0} \text{ is stochastically dominated by }\pi_{p_N}\quad \Longrightarrow \quad \mathbb P(\mathcal Q_N(x,t) \text{ is half-crossed}) < \delta_N\; \forall (x,t).
    \end{split}
\end{equation}
This will be done by induction on~$N$. The two key ingredients are horizontal and vertical decoupling estimates, which we now state. They are proved in Section~\ref{ss_induction_ext}.
\begin{lemma}[Horizontal decoupling]
	\label{lem_horizontal_decoupling_ext}
	Let~$N \in \mathbb N_0$ and assume that~\eqref{eq_hce} is satisfied. Assume that~$\xi_0$ is stochastically dominated by~$\pi_{p_{N}}$.
	Then, for any~$(x,s),(y,t) \in \mathbb Z^d \times [0,\infty)$ with~$\|x-y\| \ge 4L_{N}$ and~$|s-t| \le 2h_{N}$, we have
	\begin{equation*}
		\mathbb P(\mathcal Q_N(x,s) \text{ and } \mathcal Q_N(y,t) \text{ are both half-crossed}) \le \delta_N^2 + \mathsf v^{-2^N}.
	\end{equation*}
\end{lemma}
\begin{lemma}[Vertical decoupling]
	\label{lem_vertical_decoupling_ext}
	Let~$N \in \mathbb N_0$ and assume that~\eqref{eq_hce} is satisfied. Assume that~$\xi_0$ is stochastically dominated by~$\pi_{p_{N+1}}$.
	Then, for any~$(x,s),(y,t) \in \mathbb Z^d \times [0,\infty)$ with~$|s-t| > 2h_{N}$, we have
	\begin{equation*}
		\mathbb P(\mathcal Q_N(x,s) \text{ and } \mathcal Q_N(y,t) \text{ are both half-crossed}) \le \delta_N^2 + \mathsf v^{-2^N}.
	\end{equation*}
\end{lemma}

Putting these statements together, we obtain:
\begin{proposition}
    \label{prop_hcn}
If~$\sfv$ is large enough, then \eqref{eq_hce} holds for every $N \in \mathbb{N}_0$.
\end{proposition}
\begin{proof}
We write~$\mathfrak C_d:=255^{2d}(2d+1)$, so that~$\delta_N=(2\mathfrak C_d)^{-N-1}$. Firstly, we prove that we can take~$\sfv$ sufficiently large so that~\eqref{eq_hce} holds for~$N=0$ and $N=1$. The case $N=0$ follows from Lemma~\ref{lem_base_extinction}. It is also useful to take~$\sfv$ large so that $\sfv^{-1} \le (2\mathfrak{C}_d)^{-3}$, which implies
\begin{equation*}
    \sfv^{-2^N} \le  (2\mathfrak C_d)^{-N-3} \quad \text{for all } N \in \mathbb N
\end{equation*}
by induction. Next, we check that~\eqref{eq_hce} also holds for $N=1$, since by Lemmas~\ref{lem_cascade}, \ref{lem_horizontal_decoupling_ext} and~\ref{lem_vertical_decoupling_ext}, the probability that~$\mathcal Q_{1}(x,t)$ is half-crossed is at most
\begin{equation*}
\mathfrak C_d \cdot (\delta_0^2 + \sfv^{-1}) \le \mathfrak C_d \cdot (\exp\{- 2 \log^{3/2} \sfv \} + \sfv^{-1})
\end{equation*}
and can be made smaller than $\delta_1 = (2\mathfrak{C}_d)^{-2}$ by increasing $\sfv$ if needed.

Finally, assume that~\eqref{eq_hce} has been proved for some~$N \ge 1$, and assume that~$\xi_0$ is stochastically dominated by~$\pi_{p_{N+1}}$. Let~$(x,t) \in \mathbb Z^d \times [0,\infty)$. By a union bound using Lemma~\ref{lem_cascade}, the induction hypothesis and Lemmas~\ref{lem_horizontal_decoupling_ext} and~\ref{lem_vertical_decoupling_ext}, the probability that~$\mathcal Q_{N+1}(x,t)$ is half-crossed is at most
\begin{align*}
\mathfrak C_d \cdot (\delta_N^2 + \sfv^{-2^N}) &\le \mathfrak C_d \cdot \big((2\mathfrak C_d)^{-2N-2} +  (2\mathfrak{C}_d)^{-N-3} \big)\\
&= 2^{-2N-2}\mathfrak{C}_d^{-2N-1} + 2^{-N-3} \mathfrak{C}_d^{-N-2} \le (2\mathfrak C_d)^{-N-2} = \delta_{N+1},
\quad \text{since $N \ge 1$.}\qedhere
\end{align*}
\end{proof}

\begin{proof}[Proof of Theorem~\ref{thm_main},~\eqref{eq_main_ext}]

Let~$H$ be a graphical construction for the interchange-and-contact process. 
Let~$\xi_0 \in \{0,1\}^{\mathbb Z^d}$ be distributed as~$\pi_p$, and let~$\zeta_0$ be given by
\[
\zeta_0(x)= \begin{cases}
0&\text{if } x \neq 0 \text{ and }\xi_0(x)=0;\\
    \stateh &\text{if } x \neq 0 \text{ and }\xi_0(0)=1;\\
    \statei &\text{if } x =0.
\end{cases}
\]
Now, we will consider the set of infection paths induced by~$H$ and~$\xi_0$. Let~$A$ be the event that for all~$t$, there is an infection path started at the origin at time 0 and reaching time~$t$. We will also consider the interchange-and-contact process obtained from~$H$ and started at~$\zeta_0$, denoted~$(\zeta_t)_{t \ge 0}$.
We then have
\begin{align*}
    \mathbb P(A) = p \cdot \mathbb P(A \cap \{\xi_0(0)=1\} \mid \{\xi_0(0)=1\}) = p \cdot \mathbb P(\forall t \; \exists x: \; \zeta_t(x)=\statei) = p\cdot \Theta(\lambda, \sfv, p).
\end{align*}
Hence, to show that~$\Theta(\lambda, \sfv, p)=0$, it suffices to show that~$\mathbb P(A)=0$. We do this now.

For any~$N\in \mathbb N$, the event~$A$ is contained in the event that there is an infection path starting at~$(0,0)$ and leaving the box~$\mathcal Q_N(0,0)$. This event is in turn contained in the event that~$\mathcal Q_N(0,0)$ is half-crossed, which has probability smaller than~$\delta_N$, since~$\xi_0$ has law~$\pi_{p}$ (and hence is stochastically dominated by~$\pi_{p_N}$). The result now follows, since~$\delta_N \xrightarrow{N \to \infty} 0$.
\end{proof}

\subsubsection{Half-crossing estimates: induction step}\label{ss_induction_ext}
\begin{proof}[Proof of Lemma~\ref{lem_horizontal_decoupling_ext}]
Assume that~$\xi_0$ is dominated by~$\pi_{p_N}$. Fix~$(x,s),(y,t)$ as in the statement. We assume without loss of generality that~$s \le t \le s+2h_N$. Letting
\[
X:=\mathds{1}\{\mathcal Q_N(x,s) \text{ is half-crossed}\}, \quad Y:=\{\mathcal Q_N(y,t) \text{ is half-crossed}\}
\]
it suffices to prove that $\mathrm{Cov}(X,Y)| \le \sfv^{-2^N}$.
We define the space-time boxes
\[
\mathcal{B}_x := B_x(L_N) \times [s,s+3h_N] \supset \mathcal Q_N(x,s),\quad \mathcal{B}_y:=B_y(L_N) \times [s,s+3h_N] \supset \mathcal{Q}(y,t).
\]
We let~$\mathcal{F}$ denote the~$\sigma$-algebra generated by the interchange process inside these boxes, that is,
\[\mathcal{F}:=\sigma(\{\xi_r(z):\; (z,r) \in \mathcal{B}_x \cup \mathcal{B}_y\});\]
we also let~$\mathcal{G}$ denote the~$\sigma$-algebra generated by the Poisson processes of transmission and recovery marks inside~$\mathcal{B}_x \cup \mathcal{B}_y$. We note that~$X$ and~$Y$ are measurable with respect to~$\sigma(\mathcal{F},\mathcal{G})$. Additionally, by Lemma~\ref{lem_covariances_0}, we have
\begin{align*}
    \mathrm{Cov}(X,Y \mid \mathcal{G}) &\le 4 \mathrm{discr}^{\mathrm{ip}}(L_N, \tfrac12  \lfloor \|x-y\| \rfloor, 3 \sfv h_N)
\end{align*}
(note that the factor~$\sfv$ appears in the third argument of~$\mathrm{discr}^{\mathrm{ip}}$ because this discrepancy is defined for the interchange process with rate~1).
Since~$\mathrm{discr}^{\mathrm{ip}}(\ell,L,h)$ is non-increasing in~$L$ and~$\|x-y\| \ge 4L_N$, the r.h.s. above is smaller than
\begin{align*}
        4 \mathrm{discr}^{\mathrm{ip}}(L_N, \tfrac32 L_N    , 3 \sfv h_N) \stackrel{\eqref{eq_with_t_1}}{\le} 16 e d^3 \cdot 3\sfv h_N(3L_N+1)^{d-1} \exp\Bigl\{-\frac12 L_N \cdot \log \Bigl(1+ \frac{\tfrac12 L_N}{6\sfv h_N} \Bigr)\Bigr\}.
\end{align*}
Recalling that~$L_N=128^N \cdot \sqrt{\sfv}\log^4(\sfv)$ and~$h_N=128^N \cdot 2 \log^3(\sfv)$, when~$\sfv$ is large enough, the above is much smaller than~$\exp\{- c\cdot 128^N\cdot \log^4(\sfv)\}$ for some~$c > 0$ not depending on~$\sfv$ or~$N$. When~$\sfv$ is large enough (uniformly in~$N$), this is much smaller than~$\sfv^{-2^N}$. 
\end{proof}

Before we prove Lemma~\ref{lem_vertical_decoupling_ext}, we will need a preliminary lemma. We will use the decoupling method presented in Section~\ref{ss_domination_product}. We will apply the functions~$g^\uparrow$ and~$g^\downarrow$ and~$\mathrm{err}_{\mathrm{coup}}$ defined in that section. To make the notation cleaner, we abbreviate the sets of parameters for these functions:
\begin{align}
    \label{eq_choice_of_thetaN_ext}&\Theta_N := (\ell = h_N^{1/(2d+1/3)},\; L = 4L_{N},\; t = \mathsf v h_N,\; p = \tfrac12(p_N+p_{N+1})),\\
    &\Theta_N' := (\ell = h_N^{1/(2d+1/3)},\; L = 4L_{N},\; t = \mathsf v h_N,\; T = 2\mathsf v h_N).\label{eq_def_ThetaNp_ext}
\end{align}
\begin{remark}\label{rem:constants_extinction}
The choice of the constant $128 = 2^7$ and the exponent $\frac{1}{2d+1/3}$ in \eqref{eq_choice_of_thetaN_ext} and \eqref{eq_def_ThetaNp_ext} are tied together so that the bounds in \eqref{eq_why_128} and \eqref{eq_why_128_2} hold uniformly for all $N \ge 0$.
\end{remark}

\begin{lemma}\label{lem_outro_feio}
Let~$N \in \mathbb N_0$ and assume that~\eqref{eq_hce} holds.
    For every (deterministic)~$\xi_0 \in \{0,1\}^{\mathbb Z^d}$, the probability that~$\mathcal Q_N(0,h_N)$ is half-crossed is smaller than
    \begin{equation}
\label{eq_ext_all_terms}
    \delta_N + g^\uparrow(\Theta_N,\xi_0) + \int g^\downarrow(\Theta_N, \xi') \pi_{p_N}(\mathrm{d}\xi') + \mathrm{err}_{\mathrm{coup}}(\Theta_N').
\end{equation}
\end{lemma}
\begin{proof}
Fix~$\xi_0 \in \{0,1\}^{\mathbb Z^d}$. Let~$\xi_0'$ be distributed as~$\pi_{p_N}$, but condition on its value for now (so it will initially be treated as deterministic). We use the coupling given by Lemma~\ref{lem_coupling_rate_one} to construct interchange processes~$(\xi_t)_{t \ge 0}$ and~$(\xi'_t)_{t \ge 0}$ started from~$\xi_0$ and~$\xi_0'$, respectively, and such that~$\xi'_t(x) \ge \xi_t(x)$ for all~$(x,t) \in \mathcal Q_N(0,h_N)$ outside an event of probability at most
\[
g^\uparrow(\Theta_N,\xi_0) + g^\downarrow(\Theta_N,\xi_0')+\mathrm{err}_\mathrm{coup}(\Theta_N').
\]
Note that this coupling provides a construction for the interchange processes. On top of that, independently, we take recovery marks and transmission arrows as in Definition~\ref{def_graph_cpip}. It now makes sense to consider infection paths with respect either to~$(\xi_t)$ or to~$(\xi'_t)$.

Then, the probability that~$\mathcal Q_N(0,h_N)$ is half-crossed with respect to~$(\xi_t)$ is smaller than 
\[
\mathbb P(\mathcal Q_N(0,h_N) \text{ is half-crossed by }(\xi'_t)) + \mathbb P(\text{$\exists (x,t) \in \mathcal Q_N(0,h_N)$ such that $\xi_t(x) > \xi'_t(x)$}).
\]
By the assumption that~\eqref{eq_hce} holds, integrating the first probability above with respect to~$\xi_0' \sim \pi_{p_N}$ yields a value smaller than~$\delta_N$. Integrating the second probability with respect to~$\xi_0'$ gives the remaining terms in~\eqref{eq_ext_all_terms}.
\end{proof}

\begin{proof}[Proof of Lemma~\ref{lem_vertical_decoupling_ext}]
Assume that~$\xi_0$ is stochastically dominated~$\pi_{p_{N+1}}$. Fix~$(x,s)$ and~$(y,t)$ with~$t > s+2h_N$. We abbreviate
$\tilde \xi := \xi_{t-h_N} \circ \theta(y)$ and define
\[
a:= \int g^\uparrow (\Theta_N, \xi) \pi_{p_{N+1}}(\mathrm{d} \xi)
\quad\text{and the event}\
\mathcal A:= \left\{g^\uparrow (\Theta_N, \tilde\xi) > \sqrt{a} \right\}.
\]

Since~$\xi_0$ is stochastically dominated by~$\pi_{p_{N+1}}$, the same applies to~$\tilde{\xi}$. Hence, by Markov's inequality and monotonicity of~$g^\uparrow$,
\[
\mathbb P(\mathcal A) \le a^{-1/2} \cdot \mathbb E[g^\uparrow(\Theta_N,\tilde{\xi})] \le a^{-1/2}\cdot \int g^\uparrow(\Theta_N,\xi) \pi_{p_{N+1}}(\mathrm{d}\xi) = \sqrt{a}.
\]

Next, for each~$r \ge 0$, let~$\mathcal F_r$ denote the~$\sigma$-algebra generated by~$\xi_0$ and the restriction of the graphical representation to the time interval~$[0,r]$. Lemma~\ref{lem_outro_feio} implies that
\begin{align*}
&\mathbb P(\mathcal Q_{N}(y,t) \text{ is half-crossed} \mid \mathcal F_{t-h_N})  \leq \delta_N + g^\uparrow(\Theta_N,\tilde \xi) + \int g^\downarrow(\Theta_N, \xi) \pi_{p_N}(\mathrm{d}\xi) + \mathrm{err}_{\mathrm{coup}}(\Theta_N').
\end{align*}
Hence,
\begin{align*}
    \text{on }\mathcal A^c, \quad &\mathbb P (\mathcal Q_N(y,t) \text{ is half-crossed} \mid \mathcal F_{t-h_N}) \le \delta_N + \mathcal E,
\end{align*}
where $\mathcal E:= \sqrt{a} + \int g^\downarrow(\Theta_N, \xi) \pi_{p_N}(\mathrm{d}\xi) + \mathrm{err}_{\mathrm{coup}}(\Theta_N')$. We are now ready to bound
\begin{align}
   \nonumber&\mathbb P(\mathcal Q_N(x,s) \text{ and } \mathcal Q_N(y,t) \text{ are both half-crossed}) \\
   \nonumber&= \mathbb E[\mathds{1}\{\mathcal Q_N(x,s) \text{ is half-crossed}\} \cdot \mathbb P (\mathcal Q_N(y,t) \text{ is half-crossed} \mid \mathcal F_{t-h_N}) ]\\
   \nonumber&\le \mathbb P(\mathcal A)+ (\mathcal E + \delta_N) \cdot \mathbb P(\mathcal Q_N(x,s) \text{ is half-crossed})
   \le \sqrt{a} + \mathcal E \delta_N + \delta_N^2  \\
   &\le \sqrt{a} + \mathcal E + \delta_N^2. \label{eq_need_to_go_back}
\end{align}
We now turn to bounding all the error terms that we have gathered along the way.\\[.2cm]
\textbf{Bound on $\sqrt{a}$.} Recalling the definition of~$\Theta_N$ in~\eqref{eq_choice_of_thetaN_ext} and using Lemma~\ref{lem_integral_gs}, we have
\begin{align}
a
% &=\int g^\uparrow (\Theta_N, \xi) \;\pi_{p_{N+1}}(\mathrm{d} \xi)\nonumber \\
&\le (8L_N+1)^d \!\cdot \Bigl(e \Bigl(2 h_N^{\frac{1}{2d+1/3}}\!\!+\!2\Bigr)^d \!\!\!\cdot \sfv h_N \!+\! e\Bigr)\cdot \exp\Bigl\{-2 \cdot  \Bigl(2 h_N^{\frac{1}{2d+1/3}}\!\!+\!1\Bigr)^d \!\!\!\cdot (p_{N+1}-p_N)^2 \Bigr\}. \label{eq_bound_on_a}
\end{align}
Recall that~$L_N=128^N \cdot \sqrt{\sfv} \log^4(\sfv)$, $h_N = 128^N \cdot 2 \log^3(\sfv)$ and~$p_{N+1}-p_N=2^{-N-1}(p_0-p)$. Hence, 
\begin{equation*}
(2 h_N^{1/(2d+1/3)}+1)^d 
    \ge c_d \cdot (128^N \log^3(\sfv))^{d/(2d+1/3)}
    \ge c_d \cdot 8^N \log^{9/7} (\sfv),
\end{equation*}
for some positive constant $c_d$ and $\sfv$ sufficiently large (uniformly in $N$), since $\frac{d}{2d+1/3} \ge \frac37$ for $d \ge 1$.
As a consequence, in the exponent of~\eqref{eq_bound_on_a} we have
\begin{align}
\nonumber
2 \cdot  (2 h_N^{1/(2d+1/3)}+1)^d \cdot (p_{N+1}-p_N)^2 
    &\ge c_d \cdot (8^N \log^{9/7} (\sfv)) (4^{-N}(p_0-p)^2)\\
\label{eq_why_128}
    &= c_d (p_0-p)^2 \cdot 2^N \log^{9/7} (\sfv).    
\end{align}
A moment's reflection shows that if~$\sfv$ is large enough (depending on~$d$ and $p_0-p$, but uniformly over~$N$), then~$a$ (and also~$\sqrt{a}$) is much smaller than~$\sfv^{-2^N}$. 

\medskip
\noindent\textbf{Bound on $\int g^\downarrow(\Theta_N,\xi)\pi_{p_N}(\mathrm{d}\xi)$.} The exact same bound as in the previous item, using Lemma~\ref{lem_integral_gs}, shows that this is also much smaller than~$\sfv^{-2^N}$.\\[.2cm]
\textbf{Bound on $\mathrm{err}_{\mathrm{coup}}(\Theta_N')$.} Recall from~\eqref{eq_err_coup} that
\[
\mathrm{err}_{\mathrm{coup}}(\ell,L,t,T) := |B_0(L/2)|\cdot \left( 1 - \mathrm{meet}(\ell) \right)^{\lfloor t/\ell^2 \rfloor} + \mathrm{discr}^{\mathrm{ip}}(L/4, L/2, T),
\]
and recall from~\eqref{eq_def_ThetaNp_ext} that
$\Theta_N' := (\ell = h_N^{1/(2d+1/3)},\; L = 4L_{N},\; t = \mathsf v h_N,\; T = 2\mathsf v h_N)$.
Hence,
\begin{align}\label{eq_error_coup_fin}
    \mathrm{err}_{\mathrm{coup}}(\Theta_N')= |B_0(2L_{N})| \cdot (1-\mathrm{meet}(h_N^{\frac{1}{2d+1/3}}))^{\lfloor \mathsf v h_N/h_N^{\frac{2}{2d+1/3}}\rfloor} + \mathrm{discr}^{\mathrm{ip}}(L_N,2L_N,2\sfv h_N).
\end{align}

Plugging in~$t=\mathsf v h_N$ and~$\ell=h_N^{1/(2d+\frac13)}$ in~\eqref{eq_better_bound_meet},
we have
\[
(1-\mathrm{meet}(\ell))^{\lfloor t/\ell^2\rfloor} \le \begin{cases}
\exp\{-c \mathsf v h_N^{1/7}\}&\text{if }d=1;\\[.2cm]
\exp\{-c \mathsf v h_N^{7/13}/\log(h_N)\}&\text{if }d=2;\\[.2cm]
\exp\{-c \mathsf v h_N^{(3d+1)/(6d+1)}\}&\text{if }d\ge 3.
\end{cases}
\]
Since~$\frac{7}{13} > \frac{1}{7}$ and~$\frac{3d+1}{6d+1} > \frac{1}{7}$ for all~$d \ge 3$, and also using~$|B_0(L/2)|=|B_0(2L_N)|=(4L_N+1)^d$, we obtain, for all~$d \ge 1$: 
\begin{align}
\label{eq_why_128_2}
    |B_0(L/2)|\cdot \left( 1 - \mathrm{meet}(\ell) \right)^{\lfloor t/\ell^2 \rfloor} 
    \le (4L_{N}+1)^d \cdot \exp\bigl\{- c \mathsf v h_N^{1/7}\bigr\}.
\end{align}

Since~$L_N=128^N \cdot \sqrt{\sfv} \log^4(\sfv)$ and~$h_N=128^N \cdot 2 \log^3(\sfv)$, the r.h.s. is much smaller than~$\sfv^{-2^N}$.
We now turn to the discrepancy term of~\eqref{eq_error_coup_fin}. Using Lemma~\ref{lem_first_rw_discr}, we bound
\begin{equation}
    \label{eq_final_discrepancy_bound}
\mathrm{discr}^{\mathrm{ip}}(L_N,2L_N,2\sfv h_N) \le 16ed^3 \cdot 2\sfv h_N \cdot (2L_N+1)^{d-1} \cdot \exp\Bigl\{-L_N \cdot \log \Bigl( 1 + \frac{L_N}{4\sfv h_N} \Bigr)\Bigr\}.
\end{equation}
We have~$\frac{L_N}{4\sfv h_N} = \frac{\log(\sfv)}{8\sqrt{\sfv}}$, so, using the bound~$\log(1+x) \ge x/2$ for~$x > 0$ small enough,
\[
L_N \cdot \log \Bigl( 1 + \frac{L_N}{4\sfv h_N} \Bigr) \ge \frac{1}{16} \cdot 128^{N} \cdot \log^5(\sfv).
\]
Using this, we see that the r.h.s. of~\eqref{eq_final_discrepancy_bound} is much smaller than~$\sfv^{-2^N}$.\\[.2cm]
This concludes the treatment of all error terms. 
Going back to~\eqref{eq_need_to_go_back}, we have thus proved that
\begin{equation*}
\hspace{20mm}P(\mathcal Q_N(x,s) \text{ and } \mathcal Q_N(y,t) \text{ are both half-crossed}) \le \delta_N^2 + \sfv^{-2^N}.\hspace{18mm}\mbox{\qedhere}   
\end{equation*}
\end{proof}

\section{Microscopic propagation above the mean-field threshold}
\label{s_micro_prop}
\subsection{Propagation from random configuration}

The main goal of this section is to prove the following.

\begin{proposition}[Propagation from random configuration]\label{prop_estrela_vermelha_new}
	Let~$\lambda > 0$ and~$p \in (0,1]$ be such that~$2d p \lambda > 1$. There exist~$h_0 > 0$ and~$\varepsilon_0 \in (0,1)$ (which can be taken as small as desired) such that the following holds for~$\sfv$ large enough. Let~$A \subseteq B_0(\sqrt{\sfv})$ with~$|A|\ge \sfv^{\varepsilon_0}$, and let~$(\zeta_t)_{t \ge 0}$ be the interchange-and-contact process on~$\mathbb Z^d$ with parameters~$\lambda$ and~$\sfv$ started from a random configuration with law~$\hat \pi^A_p$. Then,
	\begin{equation*}
		\mathbb P\left(
		\begin{array}{l} 
		\begin{array}{l} 
		\zeta_{h_0} \text{ has more than $\mathsf v^{\varepsilon_0}$ infected vertices inside}\\ \text{each of } B_{-\lfloor \sqrt{\mathsf v} \rfloor \mathsf e_1}(\sqrt{\mathsf v}),\; B_{0}(\sqrt{\mathsf v}),\; B_{\lfloor \sqrt{\mathsf v} \rfloor \mathsf e_1}(\sqrt{\mathsf v})  \end{array}  \end{array} \right) > 1-  \mathsf v^{-\varepsilon_0/2} ,
	\end{equation*}
		where~$\mathsf{e}_1:=(1,0,\ldots,0) \in \mathbb Z^d$.
\end{proposition}

The proof of this proposition will require significant preparatory work, to be carried out in the following subsections. In Section~\ref{ss_bbm}, we consider  two auxiliary processes, namely \emph{branching Brownian motions} and \emph{branching random walks}. Using branching Brownian motions as a stepping stone, we prove that branching random walks satisfy a statement analogous to Proposition~\ref{prop_estrela_vermelha_new}, see Corollary~\ref{cor_multiple_brw} below.  In Section~\ref{ss_coupling_of_brw}, we construct a coupling between the interchange-and-contact process and branching random walks and prove Proposition~\ref{prop_estrela_vermelha_new}. Finally, in Section~\ref{ss_propagation_deterministic} we will state and prove a version of Proposition~\ref{prop_estrela_vermelha_new} in which the initial configuration of the process is deterministic.

\subsection{Branching Brownian motion and branching random walk}\label{ss_bbm}
In this section, we introduce branching Brownian motion and branching random walks. 
Our aim is to obtain a propagation result for the latter, Corollary~\ref{cor_multiple_brw} below. 
In order to prove it, we will appeal to an analogous result for branching Brownian motion (which can be obtained from previous work by Biggins~\cite{Biggins}), and exploit the fact that branching Brownian motion is the scaling limit of branching random walks.

\begin{definition}[Branching Brownian motion]
Let~$\beta > 0$. We consider a process of particles moving in~$\mathbb R^d$, with branching and deaths, as follows. At time~0, there are finitely many particles sitting at points of~$\mathbb R^d$. At any given time, existing particles behave independently of each other. Particles move as standard Brownian motions, and also (independently of the motion) die with rate 1 and split into two with rate~$\beta$ (in this latter case, the two new particles are placed in the same location that the parent was occupying). We represent a configuration~$\mathscr{B}$ of this process as a sum of Dirac measures,~$\mathscr B = \sum_{i=1}^m \delta_{x_i}$, where~$m$ is the number of particles in~$\mathscr B$ and~$x_1,\ldots,x_m$ are their locations (enumerated in some arbitrary way). The process is then denoted by~$(\mathscr{B}_t)_{t \ge 0}$.
\end{definition}

Since each particle in this process dies with rate 1 and is replaced by two particles with rate~$\beta$, the extinction probability~$q$ is the smallest solution in~$[0,1]$ of~$q = \frac{1}{1+\beta} + \frac{\beta}{1+\beta}\cdot q^2$. Hence,~$q = 1$ if~$\beta \le 1$, and~$q = \frac{1}{\beta}$ otherwise.

\begin{lemma}
	\label{lem_bbm}
	Let~$\beta > 1$. For any~$k \in \mathbb N$ and~$\alpha \in (0,1-\frac{1}{\beta})$, there exists~$h > 0$ (depending on~$\beta, k, \alpha$) such that the following holds. Let~$(\mathscr B_t)_{t \ge 0}$ denote the branching Brownian motion described above with parameter~$\beta$ and started from a single particle at the origin. Then,
	\begin{equation*}
		\mathbb P (\mathscr{B}_h(B_x^{\mathbb R^d}(\tfrac12)) \ge k \text{ for all } x \in B^{\mathbb R^d}_0(8)) > \alpha.
	\end{equation*}
\end{lemma}
% \marginpar{\tiny E: dependência de $h$ em $d$ a ser mencionada? \\ D: Acho que não estamos mencionando dependência em $d$ em nenhum lugar. Podemos dizer em algum lugar no início que $d$ fica fixo e não indicamos dependência em $d$.\\
% E: de acordo}

\begin{proof}
  { The proof of this lemma is indeed a simple consequence of Theorem~3 and Corollary~4 in \cite{Biggins}, whose proofs, as commented at the end of that paper, are essentially the same for the branching Brownian motion or branching random walk. One gets that as $t \to \infty$  and for each fixed $r$ there exists $c(r)>1$ so that $\liminf_{t \to \infty}\mathscr{B}_t(B^{\mathbb R^d}_x(r)) /c(r)^t$ is positive and this holds  uniformly for $x$ over compacts. Now it suffices to take $r>0$ sufficiently small and a finite set $F$ so that each $B^{\mathbb R^d}_x(1/2)$ with $x \in B^{\mathbb R^d}_0(8)$ contains at least one ball $B^{\mathbb R^d}_u(r)$ for some $u \in F$.}
\end{proof}

\begin{definition}[Branching random walk]
Let~$\beta > 0$ and~$\mathsf v > 0$. We consider a process of particles moving in~$\mathbb Z^d$, with branching and deaths, as follows.
Initially, there are finitely many particles sitting at points of~$\mathbb Z^d$. At any given time, particles behave independently of each other; they jump to each neighboring position with rate~$\mathsf{v}$, die with rate 1, and split into two (which are placed in the same location) with rate~$\beta$.  A configuration~$\eta$ of this process is represented as a sum of Dirac measures on~$\mathbb Z^d$, representing the locations of existing particles. The process is then denoted~$(\eta_{\sfv,t})_{t \ge 0}$ (omitting~$\beta$ from the notation), or simply by~$(\eta_{t})_{t \ge 0}$ when~$\sfv$ is clear from the context.
\end{definition}

Fix~$\beta > 1$. Define 
\[\bar \eta_{\sfv,t} := \sum_x \eta_{\sfv,t}(x) \cdot \delta_{\lfloor x/\sqrt{\mathsf v}\rfloor},\quad t \ge 0,\]
that is,~$(\bar \eta_{\sfv,t})$ is obtained from~$(\eta_{\sfv,t})$ by scaling space by~$\frac{1}{\lfloor\sqrt{\mathsf v}\rfloor}$. The convergence
\begin{equation*}
	(\bar\eta_{\sfv,t})_{t \ge 0} \xrightarrow[\text{(d)}]{\mathsf v \to \infty} (\mathscr{B}_t)_{t \ge 0},
\end{equation*}
follows from Donsker's theorem, where the limiting branching Brownian motion also has reproduction rate~$\beta$, and the convergence is with respect to the Skorohod topology on the space of c\`adl\`ag trajectories on finite point measures on~$\mathbb R^d$. As a consequence of this convergence and of Lemma~\ref{lem_bbm}, we obtain the following.
\begin{lemma}
	\label{lem_brw}
	Let~$\beta > 1$. For any~$k \in \mathbb N$ and~$\alpha \in (0, 1 - \frac{1}{\beta})$, there exist~$h > 0$ such that the following holds for~$\mathsf v$ large enough (both $h, \sfv$ depending on~$\beta,k,\alpha$). Let~$(\eta_t)_{t \ge 0}$ denote the branching random walk described above with parameters~$\beta$ and~$\mathsf v$, started from a single particle at the origin. Then,
	\begin{equation*}
		\mathbb P ( \eta_{h}(B_x(\tfrac{\sqrt{\mathsf v}}{2})) \ge k \text{ for any } x \in B_0(8\sqrt{\mathsf v}))> \alpha.
	\end{equation*}
\end{lemma}

In all that follows, we fix~$p \in (0,1]$ and~$\lambda > \frac{1}{2dp}$.

\bigskip
\textbf{Choice of~$h_0$.}  
We take
\begin{equation}\label{eq_choices_h0}
	\beta = 2dp\lambda,\quad \alpha =  \frac{\beta - 1}{2\beta},\quad k = \frac{2}{\alpha},
\end{equation}
and fix~$h_0$ as the value of~$h$ corresponding to~$\beta,k,\alpha$ in Lemma~\ref{lem_brw}.
The reason for these choices will become clear in the proof of the following. 

\begin{corollary}
	\label{cor_multiple_brw}
	The following holds for~$\mathsf v$ large enough. Let~$A \subseteq B_0(\sqrt{\mathsf v})$, and let~$(\eta_t)_{t \ge 0}$ be the branching random walk with parameters~$\beta = 2dp\lambda$ and~$\mathsf{v}$, with~$\eta_0 = \sum_{x \in A} \delta_{x}$. Then, letting~$h_0$ be chosen as above, we have
	\[
		\mathbb P (\eta_{h_0}(B_x(\tfrac{\sqrt{\mathsf v}}{2})) \ge |A| \text{ for any } x \in B_0(4\sqrt{\mathsf v})) > 1 -  \exp \Bigl\{- \frac{\beta - 1}{16\beta} \cdot |A|\Bigr\}.
	\]
\end{corollary}
\begin{proof}
	We construct the branching random walk~$(\eta_t)_{t \ge 0}$ as~$\eta_t := \sum_{x \in A} \eta^{(x)}_t$, where for each~$x \in A$,~$(\eta^{(x)}_t)_{t \ge 0}$ is a branching random walk with~$\eta^{(x)}_0 = \delta_{x}$, and these are independent for different choices of~$x$. 

	Let~$k$ and~$\alpha$ be as in~\eqref{eq_choices_h0}. For each~$x \in A$, let~$X_x$ denote the indicator function of the event that~$\eta^{(x)}_{h_0}(B_y(\sqrt{\mathsf v}/2)) \ge k$ for each~$y \in B_x(8\sqrt{\mathsf v})$. Then, by the choice of~$h_0$, we have~$\sum_{x \in A} X_x \sim \mathrm{Bin}(|A|,\mathsf p)$, with~$\mathsf p \ge \alpha$. A standard Chernoff bound (e.g. Corollary 27.7 in~\cite{FK}) then gives
	\[\mathbb P\Big(\sum_{x \in A} X_x < |A|\alpha/2\Big) \le \exp\Big\{-\frac{|A|\alpha}{8}\Big\} = \exp\Big\{-\frac{\beta - 1}{16\beta} \cdot |A|\Big\}.\]
	Next, since~$A \subseteq B_0(\sqrt \mathsf v)$, {each} ~$y \in B_0(4\sqrt{\mathsf v})$ belongs to~$B_x(8\sqrt{\mathsf v})$ for {all}~$x \in A$. Hence, if~$\sum_{x \in A} X_x \ge |A|\alpha/2$,  for any~$y \in B_0(4\sqrt \mathsf v)$ we have
	\[\eta_{h_0}(B_y(\tfrac{\sqrt \mathsf v}{2})) \ge \sum_{x\in A:X_x = 1} \eta_{h_0}^{(x)}(B_y(\tfrac{\sqrt \mathsf v}{2})) \ge k\sum_{x \in A} X_x \ge k \cdot |A|\cdot \frac{\alpha}{2} = |A|. \qedhere\]
\end{proof}

\subsection{Coupling between interchange-and-contact process and branching random walk}\label{ss_coupling_of_brw}
Throughout this section, we fix~$\lambda > 0$,~$\mathsf v > 0$,~$p \in (0,1]$, and a finite set~$A \subset \mathbb Z^d$. Recall the measure~$\hat\pi_p^A$ (Definition~\ref{def_pi_hat}) obtained by assigning state~$\statei$ to every vertice in~$A$, and~$\stateh$ with probability~$p$ and~$0$ with probability~$1-p$, independently, outside $A$. 
We will define a coupling between
\begin{equation*}
	\begin{array}{rl} 
    (\zeta_t)_{t \ge 0}: &\text{interchange-and-contact process} \\
    &\text{with parameters } \mathsf v,\;\lambda \\
    &\text{started from $\zeta_0 \sim \hat{\pi}^A_p$}
    \end{array}
    \quad \text{and} \quad
    \begin{array}{rl}
    (\eta_t)_{t \ge 0}: 
    &\text{branching random walk} \\
    &\text{with parameters } \mathsf v,\;\beta = 2d\lambda p\\
    &\text{started from }\eta_0 := \sum_{x\in A} \delta_x.\end{array}
\end{equation*}
	The coupling will have the property that, at least for a period of time, each infected particle in~$(\zeta_t)$  has a random walker counterpart in~$(\eta_t)$, and these two are never too far from each other in space, with high probability. 
%We will define a coupling between the interchange-and-contact process with parameters~$\mathsf v$ and~$\lambda$ started from~$\zeta_0$ and the branching random walk with parameters~$\mathsf v$ and~$\beta = 2d\lambda p$. 
To avoid confusion, we reserve the term `particle' for the interchange-and-contact process, and the term `walker' for the branching random walk.

We work on a probability space in which~$\zeta_0$ with law~$\hat{\pi}^A_p$ and the graphical representation $H$ of the interchange-and-contact process with parameters~$\lambda$ and~$\sfv$ are defined (and are independent). We will later add some additional (and independent) randomness to this space. \\[.2cm]
\textbf{Description of coupling.} Using the graphical representation~$H$, we construct
the process~$(\zeta_t)_{t \ge 0}$ started from~$\zeta_0$ and the process $(\Psi^A_t)_{t \ge 0}$, the containment flow from $A$ (see \eqref{eq_flow_from_A} in Definition~\ref{def_containment_try}).
% \begin{itemize}
%     \item the process~$(\zeta_t)_{t \ge 0}$ started from~$\zeta_0$ (in which~$A$ has infected particles and~$\mathbb Z^d \backslash A$ has healthy particles and holes);
%     \item the containment flow~$\Psi$ (see Definition~\ref{def_containment_try}); we will be particularly interested in~$(\Psi^A_t)_{t \ge 0}$, given in~\eqref{eq_flow_from_A}.
% \end{itemize}
Recall the definition of~$T^A$ in~\eqref{eq_def_of_TA} and also that
\begin{equation*} \mathds{1}\{\zeta_t(x) = \statei\} \leq \Psi^A_t(x),\quad t \ge 0,\; x \in \mathbb Z^d.
\end{equation*}
Note that for ~$t < T^A$, there is no transmission mark which both starts and ends in $\{x\colon\zeta_t(x) = \statei\}$. 
It is also important to note that~$T^A$ does not depend on~$\{\zeta_0(x): x \notin A\}$. 

Proceeding similarly to what we did in the proof of Lemma~\ref{lem_good_event_Z},
let~$0< \mathsf s_1 < \mathsf s_2 <\cdots$ denote the times at which the cardinality of the set $\{x\colon\zeta_t(x) = \statei\}$ increases one unit; for each~$j$, there is some~$z_j$ such that~$\zeta_{\mathsf s_j-}(z_j) \neq \statei$ and~$\zeta_{\mathsf s_j}(z_j) = \statei$. Enumerate~$A = \{y_1,\ldots,y_m\}$ and define
	\begin{align*}&(x_1,\mathsf t_1) = (y_1,0),\;\ldots,\; (x_m,\mathsf t_m) = (y_m,0),\\ &\hspace{3cm} (x_{m+1},\mathsf t_{m+1}) = (z_1,\mathsf s_1),\; (x_{m+2},\mathsf t_{m+2}) = (z_2,\mathsf s_2),\;\ldots.\end{align*}
		For each~$j$, an infection appears at a particle located in~$x_j$ at time~$\mathsf t_j$ (or, in case~$\mathsf t_j = 0$, the infection was already initially present). This infected particle then moves for~$t \ge \mathsf t_j$ according to the interchange flow~$t \mapsto \Phi(x_j,\mathsf t_j,t)$, and  eventually encounters a recovery mark and becomes healthy; we let~$\mathsf t_j'$ be the time when this occurs. We also let
		\[X^{(j)}_t := \Phi(x_j,\mathsf t_j,t),\quad t \ge \mathsf t_j.\]
		Although we define this process for all~$t \ge \mathsf t_j$, we will be mostly interested in it for~$t \in [\mathsf t_j,\mathsf t_j')$. In particular, we have the decomposition
\begin{equation}
	\label{eq_decomposition}
		\{(x,t): \zeta_t(x) = \statei\} = \bigcup_{j: t \in [\mathsf t_j,\mathsf t_j')}   \{(X^{(j)}_t,t)\}.
\end{equation}
Now define
\[\mathcal S^{(j)} := \{\mathsf t_j\} \cup \{t \in (\mathsf t_j,\mathsf t_j'):\;  X^{(j)}_{t-} \sim X^{(k)}_{t-} \text{ for some } k \neq j\}, \]
that is,~$\mathcal S^{(j)}$ contains~$\mathsf t_j$ (the time at which the~$j$-th infection appears in the system), together with all times~$t \in (\mathsf t_j, \mathsf t_j')$ with the property that immediately before~$t$, the particle carrying this infection had an infected neighbor.

We now want to introduce the process~$(\eta_t)_{t\ge 0}$ in this same probability space. 
This will be done in two stages. 
First, we will describe its behavior until time~$T^A$; during this period, each walker is associated to an infected particle. 
Both walker and infect particle appear at the same moment and the former (mostly) mimics the motion of the latter.
At time~$T^A$ (in case it is finite), the coupling breaks, and we let~$(\eta_t)_{t > T^A}$ evolve independently of~$(\zeta_t)_{t > T^A}$, following the law of a branching random walk started from~$\eta_{T^A}$.

To give the description of the first stage, we enlarge the probability space with a family
$((Y^{(j)}_t)_{t \ge 0}:\; j \in \mathbb N)$,
of independent continuous-time random walks on~$\mathbb{Z}^d$ which start at 0 and jump to each neighboring position with rate~$\mathsf v$ (they are also independent of~$\zeta_0$ and~$H$).
For~$j \in \{1,\ldots, |A|\}$, we define the walker trajectory~$(W^{(j)}_t)_{0 \le t \le \mathsf t_j' \wedge T^A}$ by setting
\[
	W^{(j)}_t = x_j + \sum_{s \in [0,t] \backslash \mathcal S^{(j)}}(X^{(j)}_s - X^{(j)}_{s-}) + \sum_{ s \in [0,t] \cap \mathcal S^{(j)}} (Y^{(j)}_s - Y^{(j)}_{s-}), \quad t \in [0,\mathsf t_j' \wedge T^A],
\]
that is, at times outside~$\mathcal{S}^{(j)}$, the walker mimics~$(X^{(j)}_t)$, and at times in~$\mathcal{S}^{(j)}$, it mimics the independent process~$(Y^{(j)}_t)$. 
Here and throughout, any sum over an uncountable index set is understood to have only finitely many non-zero terms.

Next, let~$n = \max\{j:\mathsf t_j < T^A\}$; we want to define the trajectory of the~$j$-th walker, for~$j \in \{|A|+1,\ldots,n\}$. 
This will be done inductively: fix~$j$ in this set, and assume that~$(W^{(i)}_t)$ has already been defined for all~$i < j$. By the definition of~$(x_j,\mathsf t_j)$, there exists some~$i < j$ such that at time~$\mathsf t_j$, there is an infection mark from~$(X^{(i)}_{\mathsf t_j},\mathsf t_j)$ towards~$(x_j,\mathsf t_j)$ (the infection with index~$i$ is the ``parent'' of the infection with index~$j$). We then let
\[
	W^{(j)}_t = W^{(i)}_{t_j} +  \sum_{s \in [t_j,t] \backslash \mathcal S^{(j)}}(X^{(j)}_s - X^{(j)}_{s-}) + \sum_{ s \in [t_j,t] \cap \mathcal S^{(j)}} (Y^{(j)}_s - Y^{(j)}_{s-}), \quad t \in [\mathsf t_j,\mathsf t_j' \wedge T^A],
\]
that is, the rule for the motion is the same as before, and the only difference is the starting position, which is taken as the same as the walker corresponding to the parent infection, at the time of transmission. We now set
\[\eta_t := \sum_{j:t \in [\mathsf t_j,\mathsf t_j')} \delta_{W^{(j)}_t},\qquad t \in [0,T^A).\]
To complete the description of the first stage, we only need to define~$\eta_{T^A}$ (in case~$T^A < \infty$). By definition, at time~$T^A$ there is a transmission mark from some vertex~$x \in \Psi^A_{T^A}$ to some neighboring vertex~$y \in \Psi^A_{T^A}$. Now, there are two cases.
\begin{itemize}
	\item If~$\zeta_{T^A-}(x)\neq \statei$, then this transmission mark has no effect in the interchange-and-contact process, and it should not impact the branching random walk either, since up to this point, infected particles and walkers are in bijection. 
    We thus set~$\eta_{T^A} = \eta_{T^A-}$ in this case.
	\item If~$\zeta_{T^A-}(x) = \statei$, then by~\eqref{eq_decomposition} there is an index~$j$ such that~$X^{(j)}_{T_A-} = x$. We then set~$\eta_{T^A} = \eta_{T^A-} + \delta_{W^{(j)}_{T^A-}}$ (that is, we add a new walker at the same position of the parent, where~$j$ is the index of this parent). Note that, in case we also had~$\zeta_{T^A-}(y)=\statei$, this introduces a discrepancy: the new walker of~$(\eta_{T^A})$, represented by~$\delta_{W^{(j)}_{T^A-}}$, has no counterpart in the interchange-and-contact process.
\end{itemize}
Now that~$(\eta_t)$ is defined up to~$T^A$, as already mentioned, the process is defined to continue after~$T^A$ (in case~$T^A < \infty$) by behaving as a branching random walk, independently of~$(\zeta_t)_{t > T^A}$. This completes the description of the coupling.\\[.2cm]
In verifying that~$(\eta_t)_{t \ge 0}$ has the correct law of a branching random walk, it is immediately clear that distinct walkers move independently with the correct distribution, and walkers die with rate~$1$. 
The only point that requires a careful consideration is that walkers produce offspring (at their own location) with rate~$\beta = 2d\lambda p$. 
Of course, this only needs to be checked before time~$T^A$. 

To justify this, we argue as follows. Let~$t < T^A$, and consider a walker at time~$t$, say at~$W^{(i)}_t$. This walker is tied to the infected particle at~$X^{(i)}_t$. The infected particle encounters a transmission mark with rate~$2d\lambda$ (counting all directions); say that this happens at time~$t'$, with~$t \le t' < T^A$, and that the target position of the transmission mark is vertex~$y$. We then have~$y \notin \Psi^A_{t'}$ (since~$t' < T^A$). Letting~$y^*$ be the unique vertex such that~$\Phi(y^*,0,t') = y$, we have that the trajectory~$(\Phi(y^*,0,s))_{0 \le s \le t'}$ does not intersect~$(\Psi^A_s)_{0 \le s \le t'}$ at any point in time. This means that the particle/hole status of~$y$ at time~$t'$ is still in equilibrium (it is a particle with probability~$p$ and a hole with probability~$1-p$). If it is a hole, no new infection is created, so no new walker is introduced to~$\eta_{t'}$. If it is a particle, then a new infection appears, and a new walker is placed at the position~$W^{(i)}_{t'}$. This shows that existing walkers indeed create offspring at their own location with rate~$\beta=2d\lambda p$.

%Letting~$h_0$ be the constant chosen after Lemma~\ref{lem_brw}, we already proved in Lemma~\ref{lem_good_event_Z} that, assuming~$|A| \le \mathsf v^{\varepsilon}$, with probability above~$1-\mathsf v^{-\varepsilon}$ we have~$T^A > h_0$ (among other things). In the event that~$T^A > h_0$, the coupling described above does not break before~$h_0$. Now we want to show that additionally, with high probability, every infected particle stays close to the walker to which it is paired until~$h_0$.

We would now like to control the distance between an infected particle and the walker to which it is paired. Note that a discrepancy may already be present at the time the infected particle appears (and the corresponding walker is born). Apart from this, if the Lebesgue measure of~$\mathcal{S}^{(j)}$  is not too large, then there is little time for any additional discrepancy to be introduced for the infected particle with index~$j$. 
For any~$j$, on the event~$\{\mathsf t_j < T^A\}$, we have 
\begin{align}\label{eq_discrepancy}
	\|X^{(j)}_t - W^{(j)}_t \| 
    &\le\| X^{(j)}_{\mathsf t_j} - W^{(j)}_{\mathsf t_j}\| +\| \mathcal{D}^{(j)}_t\| +\| \mathcal{E}^{(j)}_t\|\qquad \text{for all }t \in [\mathsf t_j, \mathsf t_j'\wedge T^A  ),\\
    \nonumber
    \text{where}\quad
    &\mathcal{D}^{(j)}_t :=  \sum_{\mathclap{s \in [\mathsf t_j,t] \cap \mathcal{S}^{(j)}}}\ ( X^{(j)}_s - X^{(j)}_{s-}),\qquad \mathcal{E}^{(j)}_t := \sum_{\mathclap{s \in [\mathsf t_j, t] \cap \mathcal{S}^{(j)} }}\ (Y^{(j)}_s - Y^{(j)}_{s-}).
\end{align}
These random variables are defined in the event~$\{\mathsf t_j < \infty\}$, and for all~$t \ge \mathsf t_j$.

\begin{figure}[t]
\begin{center}
\setlength\fboxsep{0cm}
\setlength\fboxrule{0.01cm}
\fbox{\includegraphics[width=0.9\textwidth]{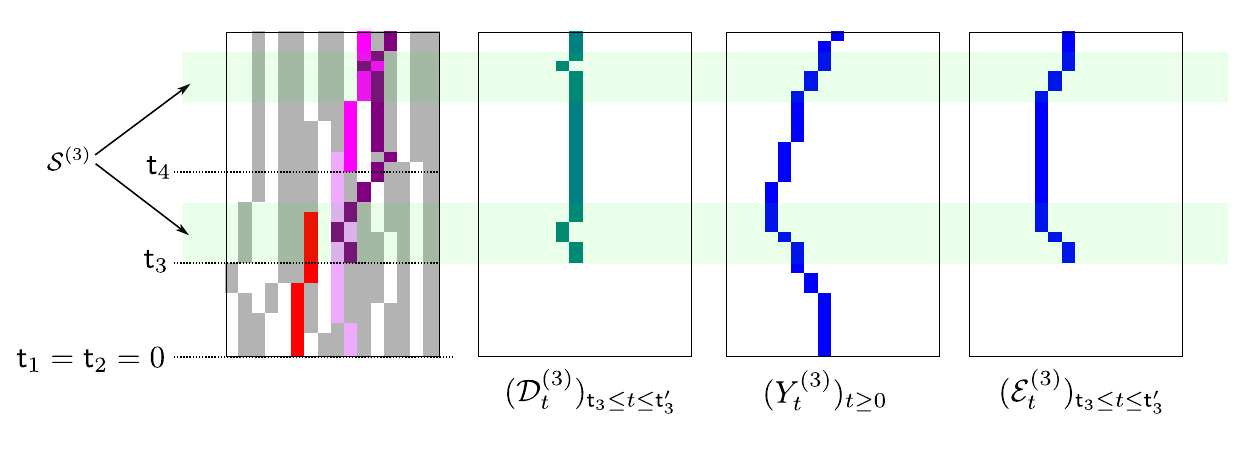}}
\end{center}
\label{fig_coupling}
\caption{Illustration of the processes~$(\mathcal D^{(j)}_t)_{t \in [\mathsf t_j,\mathsf t_j']}$ and~$(\mathcal E^{(j)}_t)_{t \in [\mathsf t_j,\mathsf t_j']}$. 
The interchange-and-contact process is depicted on the left. 
White spots are empty, and gray spots contain healthy particles.
For illustrative purposes, distinct infected particles are depicted with different colors. 
We follow the third infection, which appears at time~$\mathsf t_3$ whose path~$(X^{(3)}_t)_{t \ge \mathsf t_3}$ is colored in dark purple.
The set of times~$\mathcal S^{(3)}$ is highlighted: this is roughly the set of times when this third infection neighbors some other infection. 
The process~$(\mathcal D^{(3)}_t)$ mimics the jumps of~$(X^{(3)}_t)$ at times in~$\mathcal S^{(3)}$, and stays still otherwise. 
The process~$(\mathcal E^{(3)}_t)$ mimics the jumps of the independent random walk~$(Y^{(3)}_t)$ at times in~$\mathcal S^{(3)}$, and stays still otherwise.}
\end{figure}

Let~$\mathrm{Leb}(B)$ denote the Lebesgue measure of a set~$B \subseteq \mathbb{R}$. 

\begin{lemma} \label{lem_D_and_E}
There are constants~$c,C > 0$ such that for any~$j \in \mathbb N$ and~$t \ge 0$ we have
	\begin{equation}\label{eq_bound_short_D}
		\mathbb P\left(\begin{array}{l}\mathsf t_j < \infty,\; \mathrm{Leb}([\mathsf t_j,\mathsf t_j+t] \cap \mathcal{S}^{(j)}) < \mathsf v^{-1/4},\\[.2cm] \max_{\mathsf t_j \leq s \le \mathsf t_j + t}\; (\|\mathcal{D}^{(j)}_s\| \vee \|\mathcal E^{(j)}_s \|) > \mathsf v^{7/16}\end{array}\right) < C\exp\{-c\mathsf v^{1/8}\}.
	\end{equation}
\end{lemma}
The reason for the value~$\mathsf v^{-1/4}$ in the above is that we want to later apply the bound from Lemma~\ref{lem_good_event_Z} for the amount of time particles stay together. The reason for~$\mathsf v^{7/16}$ is as follows. Intuitively, a continuous-time random walk that jumps with rate~$\mathsf v$ to each neighboring location can reach distance of order~$(\mathsf v s)^{1/2}$ within time~$s$. Hence, if~$s = \mathsf v^{-1/4}$, the distance reached is about~$(\mathsf v \cdot \mathsf v^{-1/4})^{1/2} = \mathsf v^{3/8}$. We take~$7/16$ because it is larger than~$3/8$ and smaller than~$1/2$. We want it to be smaller than~$1/2$ because eventually, we want to say that if~$\eta_{h_0}$ has many walkers inside a ball of the form~$B_x(\tfrac12 \sqrt{\mathsf v})$, then~$\zeta_{h_0}$ has many infected particles inside the ball~$B_x(\sqrt \mathsf v)$. For this to work, the distance between each walker and the infected particle to which it is paired has to be smaller than~$\tfrac12\sqrt{\mathsf v}$.

\begin{proof}[Proof of Lemma~\ref{lem_D_and_E}]
	Fix~$j \in \mathbb N$. On the event~$\{\mathsf t_j < \infty\}$, for all~$t \ge \mathsf t_j$ define
	\[\tilde {\mathcal D}^{(j)}_t := \begin{cases}
		{ \mathcal D}^{(j)}_t &\text{if } t \in [\mathsf t_j, \mathsf t_j');\\[.2cm]
		{ \mathcal D}^{(j)}_{\mathsf t_j'} + \sum_{s \in (\mathsf t_j',t]} (X^{(j)}_s - X^{(j)}_{s-})&\text{if } t \ge \mathsf t_j'.
	\end{cases}\]
Before time~$\mathsf t_j'$, both~$(\mathcal D^{(j)}_t)$  and~$(\tilde{\mathcal D}^{(j)}_t)$ replicate the jumps of~$(X^{(j)}_t)$ in a selective way: a jump that happens at time~$t$ is only copied in case~$t \in \mathcal S^{(j)}$. Then, we complete the trajectory~$(\tilde{\mathcal D}^{(j)}_t)_{t \ge \mathsf t_j'}$ by saying that after time~$\mathsf t_j'$, this process just replicates \textit{all} jumps of~$(X^{(j)}_t)$ (regardless of whether or not the time of the jump belongs to~$\mathcal{S}^{(j)}$).

	Our next step is to do a time change in the trajectory of~$(\tilde{\mathcal D}^{(j)}_t)_{t \ge \mathsf t_j}$ so that it starts at time zero, and more importantly, we delete the time intervals corresponding to periods when it was not following the jumps of~$(X^{(j)}_t)$. 
	Formally, this is done as follows. First define
	\[\mathsf{L}^{(j)}_t :=\begin{cases} \mathrm{Leb}([\mathsf t_j, t] \cap \mathcal{S}^{(j)})&\text{if } t \in[\mathsf t_j, \mathsf t_j');\\[.1cm]
	\mathrm{Leb}(\mathcal{S}^{(j)}) + t - \mathsf t_j'&\text{if } t \ge \mathsf t_j'.\end{cases}\]
	This is a process that, up to time~$\mathsf t_j'$, increases with unit speed within~$\mathcal S^{(j)}$, and stays still otherwise. Then, define the pseudo-inverse
    \begin{align*}
    \mathscr{R}^{(j)}_r
        &:= \inf\{t \ge \mathsf t_j:\; \mathsf{L}^{(j)}_t = r\}, 
        \qquad r \ge 0,\qquad \text{(note that~$\mathscr R^{(j)}_0 = \mathsf t_j$)}\\
    \mathscr{D}^{(j)}_r 
        &:= \tilde{\mathcal D}^{(j)}_{\mathscr R^{(j)}_{\smash{r}}},\qquad r \ge 0.
    \end{align*}
	It is now not difficult to check that, conditionally on the event~$\{\mathsf t_j < \infty\}$,~$(\mathscr{D}^{\smash{(j)}}_r)_{r \ge 0}$ is a continuous-time random walk on~$\mathbb Z^d$ that starts at the origin and jumps to each neighboring position with rate~$\mathsf v$. To do this, it suffices to condition on the trajectory of this process up to say time~$r$, and to show that, in a time interval~$[r,r+\delta)$ with~$\delta$ small, it jumps with probability of order~$\mathsf v \delta$ to each neighboring location; we omit the details.
Next, note that for any~$t > 0$,
	\begin{align*} 
    \biggl\{
    &\begin{array}{l}
    \mathrm{Leb}([\mathsf t_j,\mathsf t_j + t] \cap \mathcal S^{(j)}) < \mathsf v^{-\frac14},\\
    \mathsf t_j < \infty,\; \max_{\mathsf t_j \le s \le \mathsf t_j + t} \| \mathcal D^{(j)}_s\| > \mathsf v^{\frac{7}{16}}
    \end{array}
    \biggr\}
    \subseteq 
    \biggl\{ \mathsf t_j < \infty,\; \max_{r \le \sfv^{-\frac14}} \| \mathscr D_r^{(j)} \| > \mathsf v^{\frac{7}{16}}
    \biggr\},\\
	\text{so}\qquad \mathbb P 
    \biggl(
    &\begin{array}{l}
    \mathrm{Leb}([\mathsf t_j,\mathsf t_j+t] \cap \mathcal{S}^{(j)}) < \mathsf v^{-\frac14},\\ 
    \mathsf t_j < \infty,\; \max_{\mathsf t_j \leq s \le \mathsf t_j + t}\; \|\mathcal{D}^{(j)}_s\|  > \mathsf v^{\frac{7}{16}}\end{array}
    \biggr)\le \mathbb P 
    \biggl(
    \max_{r \le \sfv^{-\frac14}} \| \mathscr D_r^{(j)} \| > \mathsf v^{\frac{7}{16}}\; \biggm|\; \mathsf t_j < \infty
    \biggr) \\
    &= \mathbb P\left( \max_{s \in [0, \mathsf v^{3/4}]} \|\mathcal X_s\| > \mathsf v^{\frac{7}{16}}\right),
	\end{align*}
	where~$(\mathcal{X}_s)_{s \ge 0}$ is the continuous-time random walk with~$\mathcal{X}_0 = 0$ that jumps to neighboring positions with rate 1. Using standard large deviations bounds for random walks (see e.g. \cite[Proposition 2.4.5]{LL}) and Poisson random variables, there exist~$c,C > 0$ such that, for any~$\alpha > 1$ and any~$S > 0$,
		\[\mathbb P\Bigl( \max_{0 \le s \le S} \|\mathcal X_s\| > \alpha \sqrt{S}\Bigr) \le C\exp\{-c\alpha^2\}.\]
		Applying this with~$S = \mathsf{v}^{3/4}$ and~$\alpha = \mathsf{v}^{1/16}$ gives the upper bound~$C \exp\{-c\mathsf v^{1/8}\}$.

		Finally, an entirely similar argument also shows that
		\[\mathbb P \left(\begin{array}{l}\mathsf t_j < \infty,\; \mathrm{Leb}([\mathsf t_j,\mathsf t_j+t] \cap \mathcal{S}^{(j)}) < \mathsf v^{-\frac14},\\[.2cm] \max_{\mathsf t_j \leq s \le \mathsf t_j + t}\; \|\mathcal{E}^{(j)}_s\|  > \mathsf v^{\frac{7}{16}}\end{array}\right) < C\exp\{-c\mathsf v^{1/8}\};\]
			we omit the details. This completes the proof.
\end{proof}

%We are now ready to prove Proposition~\ref{prop_estrela_vermelha_new}.
\begin{proof}[Proof of Proposition~\ref{prop_estrela_vermelha_new}]
	We work on a probability space where the coupling between~$(\zeta_t)$ and~$(\eta_t)$ described above is defined. We define three good events, the first being the one that appears in Lemma~\ref{lem_good_event_Z}, with~$h = h_0$:
	\[
		\mathcal G_1 := \{|\Psi^A_{h_0}| \le \mathsf v^{3\varepsilon_0},\; \mathcal K_{h_0}^A \le \mathsf v^{-1/4},\; T^A > h_0\},
	\]
    with $\mathcal K_{h_0}^A$ defined as in~\eqref{eq_def_K_t_A}.
	Next, we let~$\mathcal G_2 := \cap_{j=1}^{\lceil \mathsf v^{3\varepsilon_0} \rceil} (E^{(j)})^c$, where~$E^{(j)}$ is the event
	\[
		\Big\{\mathsf t_j < h_0, \;\mathrm{Leb}([\mathsf t_j,h_0] \cap \mathcal S^{(j)}) \le \mathsf{v}^{-1/4},\; \max_{s \in [\mathsf t_j, \mathsf t_j' \wedge h_0]} (\|\mathcal D^{(j)}_s\| \vee \|\mathcal E^{(j)}_s\| ) > \mathsf{v}^{7/16} \Big\}.
	\]
	The third good event is the one that appears in the statement of Corollary~\ref{cor_multiple_brw}:
	\[
		\mathcal G_3 := \left\{ \eta_{h_0}(B_x(\tfrac{\sqrt{v}}{2})) \ge \lceil \mathsf v^{\varepsilon_0}\rceil \text{ for any } x \in B_0(4\sqrt{\mathsf v})\right\}.
	\]
	By Lemma~\ref{lem_good_event_Z}, Lemma~\ref{lem_D_and_E} and Corollary~\ref{cor_multiple_brw}, we have
	\[
		\mathbb P(\mathcal G_1 \cap \mathcal G_2 \cap \mathcal G_3) \ge 1 - \mathsf{v}^{-\varepsilon_0} - \mathsf v^{3\varepsilon_0} \cdot  C \exp\{ -c \mathsf v^{1/8}\} - \exp \left\{- \frac{2dp\lambda - 1}{32dp\lambda} \cdot \lceil \mathsf v^{\varepsilon_0}  \rceil\right\}.
	\]
	By taking~$\varepsilon_0$ small enough and then taking~$\mathsf v$ large enough, the r.h.s. above is larger than~$1-2\mathsf v^{-\varepsilon_0}$. We now claim that 
	\begin{equation}
	\label{eq_if_good_then0}
	\text{on $\mathcal G_1 \cap \mathcal G_2$, for any $j$ with $\mathsf t_j < h_0$, we have\;\; $\max_{\mathclap{s \in [\mathsf t_j, \mathsf t_j' \wedge h_0]}}\;\; (\| \mathcal D^{(j)}_s \| \vee \|  \mathcal{E}^{(j)}_s \|) \le \mathsf v^{\frac{7}{16}}.$}
	\end{equation}
	To prove this, assume that~$\mathcal G_1 \cap \mathcal G_2$ occurs and fix~$j$ such that~$\mathsf t_j < h_0$. Since
	\begin{equation}\label{eq_t_vs_Z}
	\max\{i: \mathsf t_i < h_0\} \le |\Psi^A_{h_0}| \le \mathsf v^{3\varepsilon_0},
	\end{equation}
    we have~$j \le \mathsf v^{3\varepsilon_0}$. Moreover, since~$E^{(j)}$ does not occur, we see that
    \begin{equation*}
    \mathrm{Leb}([\mathsf t_j, \mathsf t_j' \wedge h_0] \cap \mathcal S^{(j)}) \le \mathcal K^A_{h_0} \le \mathsf v^{-1/4}
    \quad \text{and} \quad
    \max_{s \in [\mathsf t_j, \mathsf t_j' \wedge h_0]} (\| \mathcal D^{(j)}_s \| \vee \|  \mathcal{E}^{(j)}_s \|) \le \mathsf v^{7/16}.
    \end{equation*}
	Next, we will prove that
    \begin{equation}
	\label{eq_if_good_then}
	\text{on $\mathcal G_1 \cap \mathcal G_2$, for any $j$ with $\mathsf t_j < h_0$, we have\;\; $\max_{\mathclap{s \in [\mathsf t_j, \mathsf t_j' \wedge h_0]}}\;\; \| X^{(j)}_s - W^{(j)}_s \| \le j \cdot (\mathsf v^{\frac{7}{16}} + 1)$.}
	\end{equation}
	Assume that~$\mathcal G_1 \cap \mathcal G_2$ occurs.  For~$j = 1$, we have~$\mathsf t_1 = 0$ and~$X^{(1)}_0 = W^{(1)}_0$, so 
	\[ \max_{s \in [0,\mathsf t_1' \wedge h_0]} \| X^{(1)}_s - W^{(1)}_s\| \stackrel{\eqref{eq_discrepancy}}{\le} \max_{s \in [0,\mathsf t_1' \wedge h_0]} (\| \mathcal D^{(1)}_s\| \vee \| \mathcal E^{(1)}_s\|) \stackrel{\eqref{eq_if_good_then0}}{\le} \mathsf v^{7/16}. \]
	Assume that the desired inequality has already been proved for~$1,\ldots,j-1$, and that~$\mathsf t_{j} < h_0$. Note that again by~\eqref{eq_discrepancy} and~\eqref{eq_if_good_then0}, we have
	\begin{equation}\label{eq_missing_starting}
 \max_{s \in [\mathsf t_{j}, \mathsf t_{j}' \wedge h_0]} \|X^{(j)}_s - W^{(j)}_s\| \le \| X^{(j)}_{\mathsf t_{j}} - W^{(j)}_{\mathsf t_{j}}\| + \mathsf v^{7/16}.
	\end{equation}
	In case~$\mathsf t_{j} = 0$, we have~$X^{(j)}_{\mathsf t_{j}} = W^{(j)}_{\mathsf t_j}$, so the desired inequality holds. Now assume that~$\mathsf t_j > 0$. Then, there is some~$i < j$ such that the infection that appears at~$X^{(j)}_{\mathsf t_j}$ at time~$\mathsf t_j$ was transmitted from the infected particle at~$X^{(i)}_{\mathsf t_j}$, which is a location neighboring~$X^{(j)}_{\mathsf t_j}$. We then have~$\|X^{(j)}_{\mathsf t_j} - X^{(i)}_{\mathsf t_j}\| = 1$, and~$W^{(j)}_{\mathsf t_j} = W^{(i)}_{\mathsf t_j}$. Hence,
	\[ \| X^{(j)}_{\mathsf t_j} - W^{(j)}_{\mathsf t_j}\| \le 1 + \|X^{(i)}_{\mathsf t_j} - W^{(i)}_{\mathsf t_j} \| \le 1 + i \cdot (\mathsf v^{7/16} + 1) \le 1 + (j-1) (\mathsf v^{7/16}+1),
	\]
	where the second inequality follows from the induction hypothesis. Together with~\eqref{eq_missing_starting}, this gives the desired inequality in this case as well. We have now established~\eqref{eq_if_good_then}.

	Using~\eqref{eq_t_vs_Z} together with~\eqref{eq_if_good_then}, we see that on~$\mathcal G_1 \cap \mathcal G_2$ we have
	\begin{equation*}
	\max_{s \in [\mathsf t_j, \mathsf t_j' \wedge h_0]} \| X^{(j)}_s - W^{(j)}_s\| \le \mathsf v^{3 \varepsilon_0} (\mathsf v^{7/16}+1), \quad \text{for all~$j$ with~$\mathsf t_j < h_0$.}    
	\end{equation*}
        
    Now assume that~$\mathcal G_3$ also occurs. Then, we have~$\eta_{h_0}(B_{\lfloor\sqrt{\sfv}\rfloor\mathsf e_1}(\tfrac12 \sqrt \mathsf v)) > \lceil \mathsf v^{\varepsilon_0} \rceil$. Each of these~$\lceil \mathsf v^{\varepsilon_0} \rceil$ walkers is paired with an infected particle which is at distance at most~$\mathsf v^{3 \varepsilon_0} \cdot (\mathsf v^{7/16}+1)$ from it. By choosing~$\varepsilon_0$ small enough and then choosing~$\mathsf v$ large enough, we have~$\mathsf v^{3 \varepsilon_0} \cdot (\mathsf v^{7/16}+1) < \tfrac 12 \sqrt \mathsf v$, so all these infected particles are inside~$B_{\lfloor\sqrt{\sfv}\rfloor\mathsf e_1}(\sqrt \mathsf v)$. The same argument applies to~$B_{0}(\sqrt{\mathsf v})$ and $B_{-\lfloor\sqrt{\sfv}\rfloor\mathsf e_1}(\sqrt{\mathsf v})$.
    This completes the proof.
\end{proof}

\subsection{Propagation from deterministic initial configuration}
\label{ss_propagation_deterministic}
We now aim to obtain a version of Proposition~\ref{prop_estrela_vermelha_new} in which the initial configuration of the interchange-and-contact process is  deterministic. Recall the definition of the function~$g^\downarrow$ from Definition~\ref{def_gs}, and the projection~$\zeta \mapsto \xi^\zeta$ from Definition~\ref{def_projections}. 
\begin{proposition}[Propagation starting from a deterministic configuration] \label{prop_estrela_vermelha}
	Let~$\lambda > 0$ and~$p,p' \in (0,1]$ be such that~$p < p'$ and~$2d\lambda p > 1$. Let~$h_0 > 0$ and~$\varepsilon_0 \in (0,1/16)$ be taken corresponding to~$\lambda,p$ in Proposition~\ref{prop_estrela_vermelha_new}. The following holds for~$\sfv$ large enough. Let~$(\zeta_t)_{t \ge 0}$ be the interchange-and-contact process with parameters~$\sfv$ and~$\lambda$, and assume that it starts from a (deterministic) configuration~$\zeta_0$ containing at least~$\mathsf v^{\varepsilon_0}$ infected particles inside~$B_0(\sqrt{\mathsf v})$. Then, letting
\begin{equation}
	\label{eq_def_Theta00}
{\Theta} = (\ell_\Theta,L_\Theta, t_\Theta,p_\Theta),\;\text{where}\ \ell_\Theta:=\mathsf v^{1/(8d)}\!\!,\;L_\Theta:=\sqrt{\mathsf v}\log^2(\mathsf v),\; t_\Theta:=\mathsf v^{1-2\varepsilon_0}\!\!,\;p_\Theta:=\tfrac12(p+p'),
\end{equation}
we have
	\begin{equation*}
		\mathbb P\left(
		\begin{array}{l} 
		\zeta_{h_0} \text{ has more than $\mathsf v^{\varepsilon_0}$ infected vertices inside}\\ \text{each of } B_{-\lfloor \sqrt{\mathsf v} \rfloor \mathsf e_1}(\sqrt{\mathsf v}),\; B_{0}(\sqrt{\mathsf v}),\; B_{\lfloor \sqrt{\mathsf v} \rfloor \mathsf e_1}(\sqrt{\mathsf v})  \end{array} \right) > 1-  2\sfv^{-\varepsilon_0/2} - g^{\downarrow}(\Theta,\xi^{\zeta_0}).
	\end{equation*}
\end{proposition}

Some preliminary work will be needed before we prove this proposition. {For the rest of this section, we fix~$\lambda > 0$ and~$p,p'$ with~$p<p'$ and~$2d\lambda p > 1$.}

The first lemma we need is in the same spirit as Lemma~\ref{lem_integral_gs}, with the main differences that here we consider a specific choice of parameters, and allow an initial set~$A$ to contain only infected particles. Recall the function~$g^\uparrow$ from Definition~\ref{def_gs}, and the measure~$\pi_p^A$ from Definition~\ref{def_pi_p_A}.
\begin{lemma}\label{lem_integrals_A}
	The following holds for~$\mathsf v$ large enough.  Letting~${\Theta}$ be as in~\eqref{eq_def_Theta00}, for any~$A \subseteq \mathbb Z^d$ with~$|A| \le  \mathsf v^{\varepsilon_0} $ we have 
	\begin{equation}
        \label{eq_integrals_A}
		\int_{\{0,1\}^{\mathbb Z^d}} g^\uparrow({\Theta},\xi)\; \pi^A_{p}(\mathrm{d}\xi) < \exp\{-\mathsf v^{1/16}\}.
	\end{equation}
\end{lemma}
\begin{proof}
	Given~$\xi \in \{0,1\}^{\mathbb Z^d}$ and~$A \subset \mathbb Z^d$, we let~$\xi^{1 \to A} \in \{0,1\}^{\mathbb Z^d}$ be the configuration given by~$\xi^{1\to A}(x) = 1$ if~$x \in A$, and~$\xi^{1 \to A}(x) = \xi(x)$ otherwise.

	As in the definition of~${\Theta}$, let~$p_\Theta := \tfrac12 (p+p')$. Also let~$\hat p := \tfrac12 (p+p_\Theta)$, and let~$\hat{\Theta}$ be the same set of parameters as~${\Theta}$, except that the last parameter~$p_\Theta$ is replaced by the smaller value~$\hat p$. We claim that if~$\mathsf v$ is large enough, then for any~$A \subset \mathbb Z^d$ with~$|A| \le \mathsf v^{\varepsilon_0}$ and any~$\xi \in \{0,1\}^{\mathbb Z^d}$, we have
	\begin{equation}
		\label{eq_reduce_Theta_prime}
	g^\uparrow({\Theta},\xi^{1 \to A}) \le g^\uparrow(\hat{\Theta}, \xi). 
	\end{equation}
	Before we prove this, let us see how it allows us to conclude. We have
	\begin{align*}
		\int g^\uparrow({\Theta},\xi)\; \pi^A_{p}(\mathrm{d}\xi)  = \int g^\uparrow({\Theta},\xi^{1 \to A})\;\pi_{p}(\mathrm{d}\xi) \stackrel{\eqref{eq_reduce_Theta_prime}}{\le} \int g^\uparrow(\hat{\Theta},\xi)\; \pi_{p}(\mathrm{d}\xi).
	\end{align*}
	By Lemma~\ref{lem_integral_gs}, the r.h.s. is smaller than 
	\begin{align*}
		(2\sqrt{\sfv}\log^2(\sfv)+1)^d \cdot (e(2\mathsf v^{1/(8d)}+2)^d \mathsf v^{1-2\varepsilon_0} + e)\cdot \exp \Big\{-\frac18 (2 \mathsf v^{1/(8d)}+1)^d (p'-p)^2 \Big\}.
	\end{align*} 
	By taking~$\mathsf v$ large enough, this is smaller than~$\exp\{-\mathsf v^{1/16}\}$.

	It remains to prove~\eqref{eq_reduce_Theta_prime}. Fix~$A \subset \mathbb Z^d$ with~$|A| \le \mathsf v^{\varepsilon_0}$ and~$\xi \in \{0,1\}^{\mathbb Z^d}$. Let~$(\xi_t)_{t \ge 0}$ be the interchange process started from~$\xi$, and using the same graphical construction, let~$(\tilde{\xi}_t)_{t \ge 0}$ be the interchange process started from~$\xi^{1 \to A}$. Note that for any~$t$, the number of~$x \in \mathbb Z^d$ for which~$\xi_t(x) \neq \tilde \xi_t(x)$ is at most~$\mathsf v^{\varepsilon_0}$. Hence, 
	\begin{align*}
		g^\uparrow({\Theta},\xi^{1 \to A})&= \mathbb P( |\tilde \xi_t \cap B| > p_\Theta |B| \text{ for some $t \le t_\Theta$ and box $B \subset B_0(L_\Theta)$ of radius $\ell_\Theta$}  ) \\[.2cm] &\le \mathbb P\left( |\xi_t \cap B| > p_\Theta |B| -  \mathsf v^{\varepsilon_0} \text{ for some $t \le t_\Theta$ and box $B \subset B_0(L_\Theta)$ of radius $\ell_\Theta$}  \right).
	\end{align*}
	If~$B$ is a box of radius~$\ell_\Theta$, then~$|B| = (2\ell_\Theta+1)^d = (2\mathsf v^{1/(8d)} + 1)^d > \mathsf v^{1/8}$, so~$\mathsf v^{\varepsilon_0} \ll (p_\Theta-\hat p)|B|$ if~$\mathsf v$ is large enough, since~$\varepsilon_0 < 1/16$. Hence, the probability on the r.h.s. above is smaller than
    \begin{equation*}
    \ \ \mathbb P\left( |\xi_t \cap B| > \hat p |B|  \text{ for some $t \le
        t_\Theta$ and box $B \subset B_0(L_\Theta)$ of radius $\ell_\Theta$}
        \right) = g^\uparrow(\hat \Theta,\xi).\qedhere
    \end{equation*}
\end{proof}

Next, we turn to an application of Lemma~\ref{lem_coupling_rate_one} to the present context. The main differences are that here we consider the set of parameters~$\Theta$ from~\eqref{eq_def_Theta00}, and allow the processes to have rate~$\sfv$ rather than 1.
\begin{lemma} \label{lem_coupling_interchanges}
	The following holds for~$\mathsf v$ large enough. Given~$\xi, \xi' \in\{0,1\}^{\mathbb Z^d}$, there exists a probability space in which there are two graphical constructions $H$ and~$H'$ of the interchange process with rate~$\mathsf v$, with the following property. Let~$(\xi_t)_{t \ge 0}$ be the interchange process started from~$\xi$ and constructed with~$H$, and~$(\xi_t')_{t \ge 0}$ be the interchange process started from~$\xi'$ and constructed with~$H'$. Taking~${\Theta}$ as in~\eqref{eq_def_Theta00}, outside an event of probability at most
    \begin{align*}
    g^\downarrow({\Theta},\xi)  + g^\uparrow({\Theta},\xi') &+ \exp\{-\log^2(\mathsf v)\},\quad \text{we have}\\
    \xi_s(x) \ge \xi_s'(x) \quad \text{ for all } (x,s) &\in B_0(\tfrac{1}{4}\sqrt{\mathsf v} \log^2( \mathsf v)) \times [\mathsf v^{-2\varepsilon_0},h_0].
    \end{align*}
\end{lemma}
\begin{proof}
We use the coupling provided by Lemma~\ref{lem_coupling_rate_one} to obtain two graphical constructions for the rate-one interchange process, denoted~$H_1$ and~$H_1'$, corresponding to~$\xi$ and~$\xi'$ as in the statement of that lemma. 

	Let~$(\,\xi_{1,t})_{t \ge 0}$ be the interchange process started from~$\xi$ and constructed with $H_1$, and $(\xi'_{1,t})_{t \ge 0}$ the one started from~$\xi'$ and constructed from~$H_1'$.  Next, setting~$\xi_{\mathsf v,t}:= \xi_{\mathsf v t}$ and~$\xi'_{\mathsf v,t}:= \xi'_{\mathsf v t}$, we obtain two interchange processes with rate~$\mathsf v$. Note that~$(\xi_{\mathsf v,t})$ and~$(\xi'_{\mathsf v,t})$ follow the graphical constructions~$H,H'$ that are defined as the graphical constructions obtained from~$H_1$ and~$H_1'$ (respectively) after speeding up time by a factor~$\mathsf v$.

	Setting~${\Theta} = (\ell_\Theta:=\mathsf v^{1/(8d)},\;L_\Theta:=\sqrt{\mathsf v}\log^2(\mathsf v),\; t_\Theta:=\mathsf v^{1-2\varepsilon_0},\;p_\Theta:=\tfrac12(p+p'))$ as in~\eqref{eq_def_Theta00} and~$T=h_0\mathsf v$, we have
	\begin{align*}
		&\mathbb P \left(\xi_{\mathsf v,s}'(x) \ge \xi_{\mathsf v,s}(x) \text{ for all } (x,s) \in B_0(L_\Theta/4) \times [\mathsf v^{-2\varepsilon_0},h_0] \right) \\
		&\qquad =\mathbb P \left(\xi_{1,s}'(x) \ge \xi_{1,s}(x) \text{ for all } (x,s) \in B_0(L_\Theta/4) \times [t_\Theta,T] \right) \\[-1mm]
		&\qquad \overset{\mathclap{\eqref{eq_error_prob_coupling}}}{\ge}
        1- g^\downarrow({\Theta},\xi) - g^\uparrow({\Theta},\xi') - \mathrm{err}_{\mathrm{coup}}(\Theta).
	\end{align*} 
The result will now follow if we prove that
	\begin{equation}
		\label{eq_bound_16_1}
		|B_0(\tfrac{L_\Theta}{2})|\cdot ( 1 - \mathrm{meet}(\ell_\Theta) )^{\bigl\lfloor \tfrac{t_\Theta}{\ell_\Theta^2} \bigr\rfloor}  \le \frac{e^{-\log^2(\mathsf v) }}{2},
        \qquad
        \mathrm{discr}^{\mathrm{ip}}(\tfrac{L_\Theta}{4}, \tfrac{L_\Theta}{2},  T) \le \frac{e^{-\log^2(\mathsf v) }}{2}.
	\end{equation}
	Let us prove the first inequality. Plugging in the values of~$\ell_\Theta, L_\Theta,t_\Theta$ and using~$1-x \le e^{-x}$ and~$\lfloor x \rfloor \ge x/2$ for~$x \ge 1$, we bound
	\begin{align*}
		|B_0(\tfrac{L_\Theta}{2})|\cdot 
        ( 1 - \mathrm{meet}(\ell_\Theta) )^{\bigl\lfloor \tfrac{t_\Theta}{\ell_\Theta^2} \bigr\rfloor}   
        %&\le (\sqrt{\mathsf v}\log^2(\mathsf v) + 1)^d \cdot \exp \Bigl\{-\tfrac12\mathrm{meet}(\mathsf{v}^{1/(8d)})\cdot  \mathsf{v}^{1-2\varepsilon_0-\frac{1}{4d}}  \Bigr\}\\
		&\stackrel{\mathclap{\eqref{eq_meet}}}{\le} (\sqrt{\mathsf v} \log^2(\mathsf v)+1)^d\cdot \exp\Bigl\{-c\mathsf v^{1-2\varepsilon_0 - \frac{1}{4d}-\frac{(d-2)\vee 0}{8d}} \Bigr\}. 
	\end{align*}
	Since~$\varepsilon_0 < \frac{1}{16}$, we have~$1 - 2\varepsilon_0 - \frac{1}{4d} -  \frac{(d-2)\vee 0}{8d} > \frac12$. This shows that the first inequality in~\eqref{eq_bound_16_1} holds when~$\mathsf v$ is large enough.
	To prove the second inequality in~\eqref{eq_bound_16_1}, we use Lemma~\ref{lem_first_rw_discr}:
	\begin{align*}
		\mathrm{discr}^{\mathrm{ip}}(L_\Theta/4,L_\Theta/2,T) \le 16ed^3 T \left(L_\Theta+1\right)^{d-1} \cdot \exp\Bigl\{-\frac{L_\Theta}{4}\log \Bigl(1 + \frac{L_\Theta}{8T} \Bigr) \Bigr\}.
	\end{align*}
 	Note that~$L_\Theta \ll T$, so~$\tfrac{L_\Theta}{8T}$ is small and we can bound~$\log(1+\tfrac{L_\Theta}{8T}) \ge \tfrac{L_\Theta}{16T}$. We now plug in the values of~$L_\Theta$ and~$T$; it is easily seen that the second inequality in~\eqref{eq_bound_16_1} holds for~$\mathsf v$ large.
\end{proof}

\begin{proof}[Proof of Proposition~\ref{prop_estrela_vermelha}]
Let~$\lambda,p,p',h_0,\varepsilon_0$ be as in the statement of the proposition. We also let~$\sfv$ be large, to be chosen later.
For an interchange-and-contact process~$(\zeta_t)_{t \ge 0}$ (started from an arbitrary initial configuration), define the events
% \begin{align*}
%     &\mathcal A:= \left\{ \begin{array}{l} 
% 		\zeta_{h_0} \text{ has more than $\mathsf v^{\varepsilon_0}$ infected vertices inside}\\ \text{each of } B_{-\lfloor \sqrt{\mathsf v} \rfloor \mathsf e_1}(\sqrt{\mathsf v}),\; B_{0}(\sqrt{\mathsf v}),\; B_{\lfloor \sqrt{\mathsf v} \rfloor \mathsf e_1}(\sqrt{\mathsf v})  \end{array} \right\},\\[.2cm]
%         &\mathcal B:= \left\{\text{before time $h_0$, there are no infected particles outside~$B_0(\tfrac14 \sqrt{\sfv} \log^2(\sfv))$}\right\}.
% \end{align*}
\begin{equation*}
    \mathcal A\!:=\! \Bigl\{
    \begin{array}{l} 
		\zeta_{h_0} \text{ has more than $\mathsf v^{\varepsilon_0}$ infected sites in}\\
        \text{$B_{\!-\lfloor \sqrt{\sfv} \rfloor \mathsf e_1}\!(\!\sqrt{\sfv})$, $B_{0}(\!\sqrt{\sfv})$ and $B_{\lfloor \sqrt{\sfv} \rfloor \mathsf e_1}\!(\!\sqrt{\sfv})$}
    \end{array} \Bigr\},\hspace{2mm}
    \mathcal B\!:=\! \Bigl\{
    \begin{array}{l} 
		\text{before $h_0$, $(\zeta_t)$ has no infected}\\
        \text{particles outside~$B_0(\tfrac{\sqrt{\sfv} \log^2(\sfv)}{4})$}
    \end{array} \Bigr\}.
\end{equation*}
Given~$\zeta \in \{0,\stateh,\statei\}^{\mathbb Z^d}$, let~$\mathbb P_\zeta$ denote a probability measure under which we have defined an interchange-and-contact process with parameters~$\sfv$ and~$\lambda$, started from~$\zeta$. 

Let~$A \subset B_0(\sqrt{\sfv})$ be a set with~$|A|:=\lceil \sfv^{\varepsilon_0} \rceil$. We need to prove for all $\zeta$ with $\{x:\zeta(x)=\statei\}\supset A$ that
\begin{equation}
\label{eq_new_condition_is}
    \mathbb P_{\zeta}(\mathcal A) > 1- 2\sfv^{-\varepsilon_0/2} - g^\downarrow(\Theta,\xi^{\zeta})
\end{equation}
By monotonicity considerations, it suffices to prove \eqref{eq_new_condition_is} for all $\zeta$ with $\{x:\zeta(x)=\statei\}= A$.
We now state and prove two auxiliary claims.
\begin{claim}\label{cl_first_help}
    We have \quad
    $\displaystyle\int \mathbb P_\zeta(\mathcal A \cap \mathcal B)\; \hat{\pi}_p^A(\mathrm{d}\zeta) > 1 - \sfv^{-\varepsilon_0/2} - \exp\{-\log^{3/2}(\sfv)\}$.
\end{claim}
\begin{proof}
    By Proposition~\ref{prop_estrela_vermelha_new}, we have
\begin{equation*}
    \int \mathbb P_{\zeta'}(\mathcal A) \; \hat{\pi}^A_p(\mathrm{d}\zeta') > 1 - \sfv^{-\varepsilon_0/2}.
\end{equation*}
Letting~$\mathbb P$ be a probability measure under which a graphical construction of the interchange-and-contact process is defined and recalling that~$L_\Theta = \sqrt{\sfv} \log^2(\sfv)$, for any~$\zeta' \in \{0,\stateh,\statei\}^{\mathbb Z^d}$ for which~$\{x: \zeta'(x)=\statei\} = A$ we bound:
\begin{align*}
    \mathbb{P}_{\zeta'}(\mathcal{B}^c) &\le \sum_{x \in A}\; \;\sum_{y:\|y-x\| =  \lfloor L_\Theta/8 \rfloor}\mathbb{P}(y \in \cup_{s \le h_0} \Psi^{\{x\}}_s)\\[-3mm]
    &\le |A|\! \cdot\! |B_0(\lfloor L_\Theta/8 \rfloor)|\! \cdot\! 16d^2e h_0 e^{4d\lambda h_0}\!\! \cdot \exp\Bigl\{-\frac12 \lfloor L_\Theta/8 \rfloor \log\Bigl(1+ \frac{\lfloor L_\Theta/8 \rfloor}{2(\sfv + \lambda)h_0} \Bigr)\Bigr\},
\end{align*}
where the last inequality follows from Lemma~\ref{seg_primeira_dom}. It is straightforward to check that the above is smaller than~$\exp\{-\log^{3/2}(\sfv)\}$ when~$\sfv$ is large enough.
\end{proof}
\begin{claim} \label{cl_second_help}
For all $\zeta,\zeta'$ such that $\{x:\zeta(x)=\statei\}=\{x:\zeta'(x)=\statei\}=A$ we have
\begin{equation}\label{eq_for_new_condition_is}
\mathbb P_\zeta(\mathcal A)
    \ge \mathbb P_{\zeta'}(\mathcal A \cap \mathcal B) 
    - g^\downarrow(\Theta,\xi^{\zeta}) 
    - g^\uparrow(\Theta, \xi^{\zeta'}) 
    - 4d\lambda \sfv^{-\varepsilon_0} - \exp\{ -\log^2(\sfv)\}.
\end{equation}
\end{claim}
Before we prove Claim~\ref{cl_second_help}, we show how the two claims imply~\eqref{eq_new_condition_is}. Integrating both sides of~\eqref{eq_for_new_condition_is} as functions of~$\zeta'$, with respect to~$\hat \pi_{p}^A$, we have that~$\mathbb P_{\zeta}(\mathcal A)$ is larger than
\begin{align*}
    \int \mathbb P_{\zeta'}(\mathcal A \cap \mathcal B) \; \hat{\pi}^A_p(\mathrm{d}\zeta') - g^\downarrow(\Theta,\xi^{\zeta}) - \!\int\! g^\uparrow(\Theta, \xi^{\zeta'})\; \hat{\pi}^A_p(\mathrm{d}\zeta') 
    - 4d\lambda \sfv^{-\varepsilon_0}\! - e^{ -\log^2(\sfv)}.
\end{align*}
We use Claim~\ref{cl_first_help} to bound the first term from below by~$1-\sfv^{-\varepsilon_0/2} - \exp\{-\log^{3/2}(\sfv)\}$.	Moreover, we bound
\[\int g^\uparrow({\Theta},\xi^{\zeta'})\;\hat\pi^A_{p}(\mathrm{d}\zeta') =\int g^\uparrow({\Theta},\xi)\;\pi^A_{p}(\mathrm{d}\xi) \le \exp\{-\mathsf v^{1/16} \},\]
where the inequality is given by Lemma~\ref{lem_integrals_A}.
Putting things together, we have proved that~$\mathbb{P}_\zeta(\mathcal A)$ is larger than
\[
1-\sfv^{-\varepsilon_0/2}\! - e^{-\log^{3/2}(\sfv)}\! - g^\downarrow(\Theta,\xi^{\zeta})
- e^{-\sfv^{\mathrlap{1/16}}}\; - 4d\lambda \sfv^{-\varepsilon_0} - e^{-\log^2(\sfv)}.
\]
When~$\sfv$ is large enough, the r.h.s. is larger than $1- 2\sfv^{-\varepsilon_0/2} - g^\downarrow(\Theta,\xi^{\zeta})$, so the proof of~\eqref{eq_new_condition_is} is complete. We now prove the second claim.

\begin{proof}[Proof of Claim~\ref{cl_second_help}] Fix~$\zeta,\zeta'$ with~$\{x:\zeta(x)=\statei\}=\{x:\zeta'(x)=\statei\}=A$.   We will use the projections~$\xi^{\zeta}$ and~$\xi^{\zeta'}$, as in Definition~\ref{def_projections}.
We take two graphical constructions~$H,H'$ for the interchange process with rate~$\mathsf v$ corresponding to~$\xi^{\zeta}$ and~$\xi^{\zeta'}$ (respectively) as in Lemma~\ref{lem_coupling_interchanges}. On top of~$H$ and~$H'$, we take Poisson processes~$(R_x)$ and $(\mathcal{T}_{(x,y)})$ of recovery and transmission marks, respectively, as in Definition~\ref{def_graph_cpip}. We denote by~$\smash{\widehat {\mathbb P}}$ the probability measure in the probability space where these objects are defined.
We then construct~$(\zeta_{t})_{t \ge 0}$ and~$(\zeta'_t)_{t \ge 0}$ coupled together in this space as follows:~$(\zeta_t)$ starts from~$\zeta$ and uses the instructions in~$H,(\mathcal R_x),(\mathcal T_{(x,y)})$, and~$(\zeta'_t)$ starts from~$\zeta'$ and uses the instructions in~$H',(\mathcal R_x),(\mathcal T_{(x,y)})$.

	We now introduce three good events~$\mathcal{G}_1$,~$\mathcal{G}_2$ and~$\mathcal{G}_3$ which will satisfy
	\begin{equation} \label{eq_key_set_inclusion}
	\mathcal{G}_1 \cap \mathcal{G}_2 \cap \mathcal{G}_3 \subseteq \mathcal A
        = \left\{
        \begin{array}{l} 
		\zeta_{h_0} \text{ has more than $\mathsf v^{\varepsilon_0}$ infected sites in}\\
        \text{$B_{-\lfloor \sqrt{\sfv} \rfloor \mathsf e_1}\!(\sqrt{\sfv})$, $B_{0}(\sqrt{\sfv})$ and $B_{\lfloor \sqrt{\sfv} \rfloor \mathsf e_1}\!(\sqrt{\sfv})$}
        \end{array} \right\}.
  %       \left\{ \begin{array}{l} 
		% \zeta_{h_0} \text{ has more than $\mathsf v^{\varepsilon_0}$ infected vertices inside}\\ \text{each of } B_{-\lfloor \sqrt{\mathsf v} \rfloor \mathsf e_1}(\sqrt{\mathsf v}),\; B_{0}(\sqrt{\mathsf v}),\; B_{\lfloor \sqrt{\mathsf v} \rfloor \mathsf e_1}(\sqrt{\mathsf v})  \end{array} \right\}.
	\end{equation}
Let us define
% \begin{equation*}\begin{split}
% &\mathcal{G}_1:= \left\{ \begin{array}{l} 
% 		\zeta_{h_0}' \text{ has more than $\mathsf v^{\varepsilon_0}$ infected vertices inside}\\ \text{each of } B_{-\lfloor \sqrt{\mathsf v} \rfloor \mathsf e_1}(\sqrt{\mathsf v}),\; B_{0}(\sqrt{\mathsf v}),\; B_{\lfloor \sqrt{\mathsf v} \rfloor \mathsf e_1}(\sqrt{\mathsf v})  \end{array} \right\} \\
%         &\hspace{3cm}\cap \left\{ \begin{array}{l} \text{before time $h_0$, there are no infected}\\ \text{particles for~$(\zeta_t')$ outside $B_0(\tfrac14 \sqrt{\sfv}) \log^2(\sfv)$} \end{array}\right\}.
%     \end{split}
% \end{equation*}
\begin{align*}
\mathcal{G}_1
&:= \left\{
\begin{array}{l} 
	\zeta'_{h_0} \text{ has more than $\mathsf v^{\varepsilon_0}$ infected sites in}\\
    \text{$B_{\!-\lfloor \sqrt{\sfv} \rfloor \mathsf e_1}\!(\!\sqrt{\sfv})$, $B_{0}(\!\sqrt{\sfv})$ and $B_{\lfloor \sqrt{\sfv} \rfloor \mathsf e_1}\!(\!\sqrt{\sfv})$}
\end{array}
\right\}\! \cap\! \left\{
\begin{array}{l} 
		\text{before $h_0$, $(\zeta_t')$ has no infected}\\
        \text{particles outside~$B_0(\tfrac14 \sqrt{\sfv} \log^2(\sfv))$}
    \end{array}\right\}, \\
\mathcal{G}_2
    &:= \{ \xi^{\zeta_t} \ge \xi^{\zeta'_t} \text{ for all } (x,t) \in B_0(\tfrac{1}{4}\sqrt{\mathsf v} \log^2(\mathsf v)) \times [\mathsf v^{-2\varepsilon_0}, h_0] \},\\
\mathcal G_3
    &:= \{\text{no new infection appears in $(\zeta'_t)$ before time $\mathsf v^{-2 \varepsilon_0}$} \}.
\end{align*}
	In words, $\mathcal{G}_1$ is the analogue of $\mathcal A \cap \mathcal B$ for $(\zeta'_t)$ and $\mathcal{G}_2$ requires that in the space-time set~$B_0(\tfrac{1}{4}\sqrt{\mathsf v} \log^2(\mathsf v)) \times [\mathsf v^{-2\varepsilon_0}, h_0]$, wherever~$(\zeta'_t)$ has a particle,~$(\zeta_t)$ also has one (ignoring the healthy/infected status of these particles).	It is straightforward to check the inclusion~\eqref{eq_key_set_inclusion}.
	
    % We now consider the probabilities of the three good events. 
    By definition and Lemma~\ref{lem_coupling_interchanges}, we have
    \begin{equation*}
        \widehat{\mathbb P}(\mathcal G_1) = \mathbb{P}_{\zeta'}(\mathcal A \cap \mathcal B)
        \quad \text{and} \quad
        \widehat{\mathbb P}(\mathcal G_2^c) \le g^\downarrow({\Theta},\xi^\zeta) + g^\uparrow({\Theta},\xi^{\zeta'}) + \exp\{ -\log^2(\mathsf v)\}.
    \end{equation*}
	Finally, note that the number of infected particles in~$(\zeta'_t)_{t \ge 0}$ is stochastically dominated by a continuous-time Markov chain on~$\mathbb N$ that starts at~$\lceil \mathsf v^{\varepsilon_0} \rceil$ and jumps from~$k$ to~$k+1$ with rate~$2d\lambda k$. In particular,~$\widehat{\mathbb P}(\mathcal G_3^c)$ is smaller than or equal to the probability that this chain has its first jump before time~$\mathsf v^{-2 \varepsilon_0}$, that is,
	\[\mathbb P(\mathcal G_3^c) \le 1-\exp\{-2d\lambda \lceil \mathsf v^{\varepsilon_0} \rceil \cdot \mathsf v^{-2 \varepsilon_0} \} \le 4d\lambda \mathsf v^{-\varepsilon_0},\]
	where the second inequality holds for~$\mathsf v$ large enough. Hence, we have proved that
	\begin{align*}
		\mathbb P_\zeta(\mathcal A) 
        %&=\ \widehat{\mathbb P} \left( \begin{array}{l} 
		% \zeta_{h_0} \text{ has more than $\mathsf v^{\varepsilon_0}$ infected vertices inside}\\ \text{each of } B_{-\lfloor \sqrt{\mathsf v} \rfloor \mathsf e_1}(\sqrt{\mathsf v}),\; B_{0}(\sqrt{\mathsf v}),\; B_{\lfloor \sqrt{\mathsf v} \rfloor \mathsf e_1}(\sqrt{\mathsf v})  \end{array}  \right)\\
        &\smash{\overset{\smash{\eqref{eq_key_set_inclusion}}}{\ge}} \ \widehat{\mathbb P}(\mathcal G_1 \cap \mathcal G_2 \cap \mathcal G_3) \ge\  \mathbb P(\mathcal G_1) - \mathbb P(\mathcal G_2^c) - \mathbb P(\mathcal G_3^c)\\
		&\ge\  \mathbb P_{\zeta'}(\mathcal A \cap \mathcal B) - g^\downarrow({\Theta},\xi^\zeta) - g^\uparrow({\Theta},\xi^{\zeta'}) - \exp\{ -\log^2(\mathsf v)\} - 4d\lambda \mathsf v^{-\varepsilon_0}.\qedhere
	\end{align*}
\end{proof}
\phantom{\qedhere}
\end{proof}

\section{Proof of Theorem~\ref{thm_main}: survival}
\label{s_surv_renorm}
For the rest of this section, fix~$\lambda > 0$ and~$p \in [0,1]$ with~$2dp \lambda > 1$. These are the values of~$\lambda$ and~$p$ for which we will prove~\eqref{eq_main_surv}. Then,  fix~$\plow$ slightly smaller than~$p$ so that~$2d\plow\lambda > 1$ also holds, and take~$p_0:=\tfrac12(\plow+p)$. Take~$h_0$ and~$\varepsilon_0$ corresponding to~$\lambda, \plow$ in Proposition~\ref{prop_estrela_vermelha_new}, with~$\varepsilon_0<1/16$. Hence, $\lambda, p, \plow, p_0, \varepsilon_0$ and $h_0$ are all fixed. We assume throughout that~$\sfv$ is large enough, as required by the two propositions, and will keep increasing it when necessary.

We will keep denoting by~$(\zeta_t)_{t \ge 0}$ the interchange-and-contact process with parameters~$\lambda$ and~$\sfv$. The initial configuration will be specified in each context; whenever it is not specified, it is irrelevant.

\subsection{Renormalization scheme}

\subsubsection{Bottom-scale grid}

We define
\begin{equation*}
    \Tlow := (\ell_{\Tlow},L_{\Tlow}, t_{\Tlow},p_{\Tlow}),\; \text{where}\;\ell_{\Tlow}:=\mathsf v^{1/(8d)},\;L_{\Tlow}:=\sqrt{\mathsf v}\log^2(\mathsf v),\; t_{\Tlow}:=\mathsf v^{1-2\varepsilon_0},\;p_{\Tlow}:=\tfrac12(\plow + p_0),
\end{equation*}
that is,~$\Tlow$ is the same as~$\Theta$ that appears in Proposition~\ref{prop_estrela_vermelha}, except that the last parameter
is now~$\tfrac12(\plow + p_0)$.
For~$\zeta \in \{0,\stateh,\statei\}^{\mathbb Z^d}$, we abbreviate
\[
G_{\sfv}(\zeta):=g^\downarrow(\Tlow,\; \xi^{\zeta} \cdot \mathds{1}_{B_0(2L_{\Tlow})}).
\]

Let us define the bottom-scale grid of our renormalization scheme.

\begin{definition}[Scale-0 grid and boxes] 
	\label{def_bottom_scale_boxes} 
 Given~$m \in \mathbb Z$ and~$n \in \mathbb N_0$, define
\begin{align*}
    \mathcal L_0 &:= \lfloor \sqrt{\mathsf v} \rfloor,\\
    \mathbf{x}_0(m) &:= \mathcal L_0 m  \cdot \mathsf{e}_1  \in \mathbb Z^d,\\
    \vec{\mathbf{x}}_0(m,n)&:= \mathcal L_0 m  \cdot \mathsf{e}_1 + h_0 n \cdot \mathsf{e}_{d+1} \in \mathbb Z^d \times [0,\infty),
\end{align*}
	where~$\mathsf e_1 := (1,0,\ldots,0) \in \mathbb Z^d \times [0,\infty)$ and~$\mathsf e_{d+1} := (0,\ldots,0,1) \in \mathbb Z^d \times [0,\infty)$.
	The points~$\vec{\mathbf{x}}_0(m,n)$ are called the \emph{scale-0 grid points}. 
    Next, let
    \[\mathcal{L}_0^{\mathrm{side}}:=2\sqrt{\sfv}\log^2(\sfv)\]
    and define the collection of space-time boxes~$\{\mathcal Q_0(m,n): m \in \mathbb Z,\; n \in \mathbb N_0\}$ by letting
    \begin{align*}
    \mathcal Q_0(0,0)&:= [-\mathcal{L}_0^{\mathrm{side}},\mathcal{L}_0^{\mathrm{side}}]^d \times [0,h_0] \subset \mathbb R^d \times [0,\infty), \\
    \mathcal Q_0(m,n)&:= \vec{\mathbf{x}}_0(m,n) + \mathcal Q_0(0,0), \quad m \in \mathbb Z,\; n \in \mathbb N_0.
    \end{align*}
\end{definition}

As mentioned in the Introduction, given~$\zeta \in \{0,\stateh, \statei\}^{\mathbb Z^d}$ and~$x \in \mathbb Z^d$, we let~$\zeta \circ \theta(x) \in \{0,\stateh, \statei\}^{\mathbb Z^d}$ be the translation given by
\[[\zeta \circ \theta(x)](y) = \zeta(x+y),\quad y \in \mathbb Z^d.\]
% \marginpar{\tiny estou confusa com algo tolo. queremos $-x$ ou $+x$ na def? {\color{blue} Tem razão! \'E para ser $+x$, mudei. Na verdade, eu agora introduzi translações na parte de Notation, então talvez seja melhor apagar daqui.de acordo.} acho que fica melhor concentrar na Notation e apagar aqui, talvez lembrando ''following notation introduced in ...".}
\begin{definition}[Bad points in scale 0]
\label{def_bad_0_new}
	We declare that the point~$ (0,0) \in \mathbb Z \times \mathbb N_0$ is \emph{0-bad} for a realization of~$(\zeta_t)$ if either
\begin{itemize}
	\item[$(\mathrm{B}1)$] \emph{``few particles at the initial time'':} we have
		\begin{equation}G_\mathsf{v}(\zeta_0) \ge \exp\{-\tfrac12\mathsf v^{\varepsilon_0}\}.\label{eq_few_particles1}\end{equation}
        \end{itemize}
        or
        \begin{itemize}
	\item[$(\mathrm{B}2)$] \emph{``good conditions for propagation, but no propagation'':} \eqref{eq_few_particles1} does not hold and~$\zeta_{0}$ has at least~$\mathsf v^{\varepsilon_0}$ infections inside the box~$B_{0}( \sqrt{\mathsf v})$, but~$\zeta_{h_0}$ has fewer than~$\mathsf v^{\varepsilon_0}$ infections inside (at least) one of the boxes
 \begin{equation}\label{eq_important_boxes_new}
     B_{\mathbf x_0(-1)}( \sqrt{\mathsf v}),\;\;B_{0}( \sqrt{\mathsf v}),\;\; B_{\mathbf x_0(1)}( \sqrt{\mathsf v}).
 \end{equation}
\end{itemize}
	For~$m \in \mathbb Z$ and~$n \in \mathbb N_0$, we say that the point~$(m,n) \in \mathbb Z \times \mathbb N_0$ is 0-bad for~$(\zeta_t)$ in case the point~$(0,0)$ is 0-bad for the process translated so that~$\vec{\mathbf{x}}_0(m,n)$ becomes the space-time origin, that is, the process
	\[(\zeta_{nh_0+t} \circ \theta(\mathbf x_0(m)))_{t \ge 0}.\]
\end{definition}
\begin{remark}\label{rem_dependence0}
In order to check whether condition~$\mathrm{(B1)}$ is satisfied {for $(m,n)$}, it is enough to know the value of~$\zeta_t(x)$ for~$(x,t)$ in
\[
B_{\mathbf x_0(m)}(\mathcal{L}_0^{\mathrm{side}}) \times \{h_0 n\}.
\]
In order to check whether condition~$\mathrm{(B2)}$ is satisfied { for $(m,n)$}, it is enough to know the value of~$\zeta_t(x)$ for~$(x,t)$ in the same space-time set as above, together with
\[
B_{\mathbf x_0(m)}(2\sqrt{\sfv})\times \{h_0 (n+1)\}. 
\]
Both these space-time sets are contained in~$\mathcal Q_0(m,n)$.
Thus, we can decide whether~$(m,n)$ is bad with knowledge of~$(\zeta_t(x):(x,t) \in \mathcal Q_0(m,n))$.
\end{remark}

\begin{remark}
	Since Definition~\ref{def_bad_0_new} is somewhat involved, it is useful to spell out its negation, that is, to describe when a point~$(m,n) \in \mathbb Z \times \mathbb N_0$ is not 0-bad (i.e. it is \emph{0-good}) for~$(\zeta_t)_{t \ge 0}$. We do this for~$(m,n) = (0,0)$; this point is 0-good if one of the following two conditions holds:
\begin{itemize}
	\item[$(\mathrm{G}1)$] \emph{``many particles, but few infected ones at the initial time'':} we have $G_{\mathsf v}(\zeta_0) < \exp\{-\tfrac12\mathsf v^{\varepsilon_0}\}$, and~$\zeta_{0}$ has fewer than~$\mathsf v^{\varepsilon_0}$ infected particles inside~$B_{0}( \sqrt{\mathsf v})$;
	\item[$(\mathrm{G}2)$] \emph{``successful propagation'':} we have~$G_{\mathsf v}(\zeta_0) < \exp\{-\tfrac12\mathsf v^{\varepsilon_0}\}$,~$\zeta_{0}$ has at least~$\mathsf v^{\varepsilon_0}$ infections inside~$B_{0}( \sqrt{\mathsf v})$, and $\zeta_{h_0}$ has at least~$\mathsf v^{\varepsilon_0}$ infections inside each of the boxes in~\eqref{eq_important_boxes_new}.
\end{itemize}
\end{remark}

Since eventually our notion of good points will serve to show survival of the infection, it may seem odd to label condition $(\mathrm{G}1)$ above as `good'. The reason for this labelling is technical. We want to be able to prove that bad points are very rare when~$\zeta_0$ has many particles (for instance, when it dominates a sufficiently dense Bernoulli product measure), \emph{regardless of whether these particles are healthy or infected}. To achieve this goal, it is helpful to label situations where there are many particles but few infections as good. At the same time, this will not cause trouble when we show survival of the infection, due to the following simple observation, which we record as a lemma.
\begin{lemma}
	Let~$0=m_0,\ldots,m_k \in \mathbb Z$ with~$|m_{i+1}-m_i| \le 1$ for each~$i$. Assume that~$\zeta_{0}$ has at least~$\mathsf v^{\varepsilon_0}$ infections inside~$B_{0}(\sqrt{\mathsf v})$, and that the points~$(m_i,i)$, with~$0 \le i \le k$, are all 0-good for~$(\zeta_t)$. Then, the boxes
 \[
B_{{\mathbf{x}_0}(m_k-1)}( \sqrt{\mathsf v}),\;\;B_{\mathbf{x}_0(m_k)}( \sqrt{\mathsf v}),\;\;B_{\mathbf{x}_0(m_k+1)}( \sqrt{\mathsf v})
 \]
all have at least~$\mathsf v^{\varepsilon_0}$ infections in~$\zeta_{(k+1)h_0}$.
\end{lemma}

Using Proposition~\ref{prop_estrela_vermelha_new}, we will now show that, for a process with density of particles above~$\plow$, the probability that a point is~$0$-bad is small.
\begin{corollary}
	\label{cor_scale_0_G}
	The following holds if~$\mathsf v$ is large enough.
	Assume that~$(\zeta_t)$ starts from a random configuration~$\zeta_0$ such that the law of the projection~$\xi^{\zeta_0} \in \{0,1\}^{\mathbb Z^d}$ stochastically dominates~$\pi_{p_0}$. Then, for any~$(m,n)$, the probability that~$(m,n)$ is~$0$-bad for~$(\zeta_t)$ is smaller than~$3\mathsf v^{-\varepsilon_0/2}$.
\end{corollary}
\begin{proof}
	The assumption that~$\xi^{\zeta_0}$ stochastically dominates~$\pi_{p_0}$ implies that~$\xi^{\zeta_{nh_0}\circ \theta(\mathbf{x}(m))}$ {also does it;} this can be easily seen using the graphical representation and the fact that Bernoulli product measures are stationary for the interchange process. Due to this observation, it suffices to prove the bound for~$(m,n) = (0,0)$.

We start by finding an upper bound for~$\mathbb E[G_{\mathsf v} (\zeta_0)]$.
If~$\xi,\xi' \in \{0,1\}^{\mathbb Z^d}$ are such that~$\xi(x)\le \xi'(x)$ for all~$x$, then~$g^\downarrow(\Tlow,\xi) \ge g^\downarrow(\Tlow,\xi')$. Using this, we have\vspace{-1mm}
\begin{equation*}
\mathbb E[G_{\mathsf v} (\zeta_0)]
    = \mathbb E[g^\downarrow(\Tlow,\xi^{\zeta_0} \cdot \mathds{1}_{B_0(2L_{\Tlow})})] 
    \le \smash{\int} g^\downarrow(\Tlow,\xi \cdot \mathds{1}_{B_0(2L_{\Tlow})})\; \pi_{p_0}(\mathrm{d}\xi).
\end{equation*}
Now note that, for any~$\xi$,
\begin{align}
&|g^\downarrow(\Tlow,\xi \cdot \mathds{1}_{B_0(2L_{\Tlow})})-g^\downarrow(\Tlow,\xi )| \le \mathrm{discr}^{\mathrm{ip}}(L_{\Tlow},2L_{\Tlow},\mathsf v^{1-2\varepsilon_0})\nonumber\\
\label{eq_gdown_to_Gv}
    &\stackrel{\mathclap{\eqref{eq_with_t_1}}}{\leq} 16ed^3 \mathsf v^{1-2\varepsilon_0} \cdot (4\sqrt{\mathsf v}\log^2(\mathsf v)+1)^{d-1} \cdot \exp \Bigl\{-(\sqrt{\mathsf v}\log^2(\mathsf v)) \cdot \log \Bigl(1+ \frac{\sqrt{\mathsf v}\log^2(\mathsf v) }{2\mathsf v^{1-2\varepsilon_0}} \Bigr) \Bigr\}.
\end{align}
When~$\sfv$ is large enough, the r.h.s. is smaller than~$\exp\{-\sfv^{\varepsilon_0}\}$. This shows that
\[
\mathbb E[G_{\mathsf v} (\zeta_0)] \le  \int g^\downarrow(\Tlow,\xi)\; \pi_{p_0}(\mathrm{d}\xi) +\exp\{-\sfv^{\varepsilon_0}\}.
\]
Using the definition of~$\Tlow$ and Lemma~\ref{lem_integral_gs}, the integral on the r.h.s. is smaller than
\[
		(2\sqrt{\sfv}\log^2(\sfv)+1)^d \cdot (e(2\mathsf v^{1/(8d)}+2)^d \mathsf v^{1-2\varepsilon_0} + e)\cdot \exp \Bigl\{-\frac12 (2 \mathsf v^{1/(8d)}+1)^d (p_0-\plow)^2 \Bigr\},
\]
which is smaller than~$\exp\{-\sfv^{1/16}\}$ when~$\sfv$ is large enough. We have thus proved that
\[\mathbb E[G_{\mathsf v} (\zeta_0)] \le \exp\{-\sfv^{1/16}\} + \exp\{-\sfv^{\varepsilon_0}\} \le 2\exp\{-\sfv^{\varepsilon_0}\},\]
since we have taken~$\varepsilon_0<1/16$. Markov's inequality now gives
\begin{align*}
	&\mathbb{P}\left(G_\mathsf{v}(\zeta_0 ) \ge \exp\{-\tfrac12 \mathsf v^{\varepsilon_0} \}\right) \le  \exp\{\tfrac12 \mathsf v^{\varepsilon_0}\} \cdot \mathbb E[G_{\mathsf v} (\zeta_0)] < 2\exp\{-\tfrac12 \mathsf v^{\varepsilon_0}\},
 \end{align*}
 controlling the probability of condition~$(\mathrm B1)$ in Definition~\ref{def_bad_0_new}.
	Now, let~$\mathcal A_0$ denote the event that~$G_{\mathsf v}(\zeta_0) < \exp\{-\tfrac12\mathsf v^{\varepsilon_0}\}$ and~$\zeta_0$   has at least~$\mathsf v^{\varepsilon_0}$ infections in~$B_{0}( \sqrt{\mathsf v})$. Let~$\mathcal A_0'$ be the event that~$\mathcal A_0$ occurs, but~$\zeta_{h_0}$ fails to have at least~$\mathsf v^{\varepsilon_0}$ infections in either of the boxes in~\eqref{eq_important_boxes_new}. Note that~$\mathcal A_0'$ corresponds to the event described in condition~$(\mathrm B2)$. We then have
	\begin{align*}
	\text{on } \mathcal A_0, \quad \mathbb P(\mathcal A_0' \mid \zeta_0) 
        &\le g^\downarrow(\Tlow,\xi^{\zeta_0})+2\sfv^{-\varepsilon_0/2} 
        \overset{\mathclap{\eqref{eq_gdown_to_Gv}}}{\le} G_{\mathsf v}(\zeta_0) + \exp\{-\sfv^{\varepsilon_0}\} + 2\sfv^{-\varepsilon_0/2}\\
        &\le 2\exp\{-\tfrac12 \sfv^{\varepsilon_0}\}+  2\sfv^{-\varepsilon_0/2}
	\end{align*}
	where the first inequality follows from Proposition~\ref{prop_estrela_vermelha} and the last inequality follows from the fact that~$G_{\sfv}(\zeta_0) < \exp\{-\tfrac12 \mathsf v^{\varepsilon_0}\}$ on~$\mathcal A_0$. Integrating the above inequality now gives
	\begin{equation*}
    \mathbb P(\mathcal A_0') 
    \le \mathbb E[\mathds{1}_{\mathcal A_0} \cdot \mathbb P(\mathcal A_0' \mid \zeta_0)] 
    \le 2\exp\{-\tfrac12 \sfv^{\varepsilon_0}\} +2\sfv^{-\varepsilon_0/2}.
    \end{equation*}

Putting things together, we have proved that
	\[\mathbb P((0,0) \text{ is 0-bad for $(\zeta_t)$}) \le 4\exp\{-\tfrac12 \sfv^{\varepsilon_0}\} + 2\sfv^{-\varepsilon_0/2}; \]
	when~$\mathsf v$ is large enough, the r.h.s. is smaller than~$3\mathsf v^{-\varepsilon_0/2}$, as desired.
\end{proof}

\subsubsection{Higher-scale grids and boxes}

Our next goal is to define a sequence of grid scales and a collection of boxes associated to each scale. The boxes will be taken so that their union covers the ``slab''~$\mathbb R \times [-\mathcal L_0^{\mathrm{side}},\mathcal L_0^{\mathrm{side}}]^{d-1} \times [0,\infty)$.
In our construction, it will be useful to allow for some spatial overlap between adjacent boxes. The overlap on scale $N$ is controlled by a factor{~$\rho_N \in [1,2)$}. We define it by setting
{\begin{equation}
\label{eq:defi_rho_N}
    \rho_N := \sum_{i=0}^{N} 2^{-i},\quad N \in \mathbb N_0.
\end{equation}  }
The growth of scales will be controlled by the value
\begin{equation*}
\alphav:=\lfloor \mathsf v^{\varepsilon_0/64}\rfloor.
\end{equation*}
Recall that~$h_0$ has been fixed, and~$\mathcal L_0 := \lfloor \sqrt{\mathsf v} \rfloor$.
\begin{definition}[Scale-$N$ grid and boxes] 
Let
\begin{align}\label{eq_def_of_primes}
	&\mathcal L_{N} :=  \alphav^{N^2} \cdot \mathcal L_0  \quad \text{and}\quad h_N' :=  \rho_N\alphav^{N^2} \cdot  h_0  ,\quad N \in \mathbb N.
\end{align}
	In order to obtain an integer multiple of~$h_{N-1}$ from the latter, we set
	\[h_{N}:= \lfloor h_{N}'/h_{N-1} \rfloor \cdot h_{N-1}, \quad N \in \mathbb N.\]
Given~$m \in \mathbb Z$ and~$n \in \mathbb N_0$, define
\begin{equation*}\label{eq_def_bfx}
\mathbf{x}_N(m) := \mathcal L_N m \cdot \mathsf{e}_1 \in \mathbb Z^d,
\quad \text{and} \quad
\vec{\mathbf{x}}_N(m,n):=  \mathcal L_N m \cdot \mathsf{e}_1 + h_N n \cdot \mathsf{e}_{d+1} \in \mathbb Z^{d} \times [0,\infty).
\end{equation*}
The points~$\vec{\mathbf{x}}_N(m,n)$ are called the \emph{scale-$N$ grid points}.
Next, let
	\begin{equation}\label{eq_def_LN_side}
    \mathcal L_N^{\mathrm{side}}:=\rho_N \mathcal L_N,\quad N \in \mathbb N.
	\end{equation}
    Define the collection of space-time boxes~$\{\mathcal Q_N(m,n): m \in \mathbb Z,\; n \in \mathbb N_0\}$  by
    \begin{align*}
    \cQ_N(0,0) 
        &:= [-\mathcal L_N^{\mathrm{side}}, \mathcal L_N^{\mathrm{side}}] \times [-\mathcal L_0^\mathrm{side},\mathcal L_0^\mathrm{side}]^{d-1} \times [0,h_N]
        % \subset \mathbb R^{d} \times [0,\infty)
        %\subset \mathbb R \times [-\mathcal L_0^\mathrm{side},\mathcal L_0^\mathrm{side}]^{d-1} \times [0,\infty)
        \\
    \cQ_N(m,n)
        &:= \vec{\mathbf{x}}_N(m,n) + \cQ_N(0,0), \quad m \in \mathbb Z,\; n \in \mathbb N_0.
    \end{align*}
\end{definition}

Next, we give an inductive definition of a bad point for scale $N$. 
\begin{definition}[Bad points in scale~$N$]\label{def_bad_N} Let~$N \in \mathbb N$.
	We declare that the point~$(m,n) \in \mathbb Z \times \mathbb N_0$ is \emph{$N$-bad} for~$(\zeta_t)$ if there are indices $(i,j)$ and
$(i',j')$ such that
\begin{itemize}
	\item $(i,j)$ and~$(i',j')$ are~$(N-1)$-bad;
	\item $\mathcal{Q}_{N-1}(i,j)$ and~$\mathcal{Q}_{N-1}(i',j')$ are contained in~$\mathcal Q_N(m,n)$;
	\item either $|j-j'| > 1$ or $\bigl[|j-j'| \le 1$ and $|i-i'| > \sqrt{\alphav}\bigr]$.
	\end{itemize}
\end{definition}
Note that as before,~$(m,n)$ is~$N$-bad for~$(\zeta_t)$ if and only if~$(0,0)$ is~$N$-bad for the translated process~$(\zeta_{nh_N+t} \circ \theta(\mathbf x_N(m)))_{t \ge 0}$.

\begin{remark}\label{rem_dependence}
Using Remark~\ref{rem_dependence0} and arguing by induction, we see that it is possible to decide whether~$(m,n)$ is~$N$-bad with knowledge of~$\{\zeta_t(x): (x,t) \in \mathcal Q_N(m,n)\}$.
\end{remark}

\begin{remark}\label{rem_slow}
Let us give an heuristic explanation for our choices of scales and sizes of the renormalization scheme. Note that apart from scale 0, which is somewhat special, the quotient~$\mathcal{L}_N^{\mathrm{side}}/h_N$ is roughly the same in all scales, a natural choice. Somewhat trickier is the fact that the factor~$\rho_N$ appears in the definition of~$h_N$ and~$\mathcal{L}_N^{\mathrm{side}}$, but not on the spatial grid length~$\mathcal{L}_N$, thus causing spatial overlap between adjacent boxes. This is what we now address on an intuitive level; this intuition is mathematically implemented in the statement and proof of Lemma~\ref{lem_accessible_spread} below.

For the sake of this explanation, define the \emph{cone of scale-$N$ boxes} 
\[\mathscr{C}_N:=\bigcup_{\substack{(m,n):n \ge 0,\\-n \le m \le n}} \mathcal{Q}_N(m,n).\]
Let us think of~$\mathscr{C}_N$ as the region inside which the infection could ideally propagate using level-$N$ boxes -- by `ideally' we mean we are thinking of an idealized scenario where, very roughly speaking, given that $\mathcal{Q}_N(m,n)$ has many infections,
\begin{equation*} \begin{split}
    &\mathcal{Q}_N(m-1,n+1), \mathcal{Q}_N(m,n+1), \mathcal{Q}_N(m+1,n+1) \text{ are good } \\ &\Longrightarrow \quad \text{$\mathcal{Q}_N(m-1,n+1), \mathcal{Q}_N(m,n+1), \mathcal{Q}_N(m+1,n+1)$ have many infections.}
\end{split}\end{equation*}
Note that~$\mathscr{C}_N$ has slope equal to~$\mathrm{slope}(N):=\mathcal{L}_N/h_N$. For the renormalization to work from one scale to the next, it is very important that~$\mathrm{slope}(N) \ge \mathrm{slope}({N+1})$; this way, we could hope that a propagating front of level-$N$ boxes could produce a propagating front of level-$(N+1)$ boxes, thus allowing us to prove propagation in all scales by an inductive argument.

In fact, having~$\mathrm{slope}(N) = \mathrm{slope}({N+1})$ (which would hold without the introduction of the overlap) would not be good enough: we need a strict inequality. Indeed, when we consider good~$N$-boxes propagating inside an environment of good~$(N+1)$-boxes, the \emph{effective speed} is slightly less than ~$\mathrm{slope}(N)$, because occasionally (albeit sporadically) a space-time region of bad level-$N$ boxes has to be circumvented. The role of the overlap factors~$(\rho_N)_N$ is to cause~$\mathrm{slope}(N)$ to be (slowly) decreasing in~$N$, in order to guarantee that the inequality~$\mathrm{slope}(N) > \mathrm{slope}({N+1})$ holds (even if we reduce the l.h.s. to its effective value which accounts for loss of speed).

Incidentally, this loss of speed effect is the main reason we have taken our renormalization scales growing faster than exponentially. If we had taken the scale growth as~$\alphav^N$ rather than~$\alphav^{N^2}$, then the ratio between the box side lengths~$\mathcal{L}_{N}^{\mathrm{side}}$ and~$\mathcal{L}_{N+1}^{\mathrm{side}}$ would not tend to zero with~$N$, but would stay constant instead, causing the speed to decrease by a constant factor with each scale, eventually vanishing. 
\end{remark}

The following is a summary of the renormalization scheme described so far, for ease of reference:

\medskip
\noindent
\begin{boxedminipage}{\textwidth}
{\footnotesize
\begin{tabular}{p{0.24\textwidth} p{0.76\textwidth}}
\textbf{Initialization constants}:&$\plow,\lambda$ \text{with }$2d\lambda \plow > 1$;\;$h_0$,~$\varepsilon_0$ corresponding to~$\lambda$ and~$\plow$ in Proposition~\ref{prop_estrela_vermelha_new}
\end{tabular}\\[.2cm]
\begin{tabular}{p{0.43\textwidth} p{0.56\textwidth}}
	\multicolumn{2}{p{\textwidth}}{\textbf{Renormalization growth constants}:$\quad$ $\alpha_{\mathsf{v}}:=\lfloor \mathsf{v}^{\varepsilon_0/64}\rfloor$,\; $\rho_N:=\sum_{i=0}^N 2^{-i},\; N \in \mathbb{N}_0$}\\[.2cm]
	\textbf{Grids}&\textbf{Boxes}\\
	$\mathcal{L}_N := \lfloor \sqrt{\mathsf{v}} \rfloor \cdot \alpha_{\mathsf v}^{N^2},\; N \in \mathbb{N}_0$
	 & $\mathcal{L}_0^{\mathrm{side}}:=2\sqrt{\mathsf v}\log^2(\mathsf v);\;\mathcal{L}_N^{\mathrm{side}} := \rho_N \mathcal{L}_N,\; N \ge 1$   \\
	 $h_N':=\rho_N \alpha_{\mathsf{v}}^{\mathrlap{N^2}} \cdot h_0;\; h_N:=\left\lfloor \tfrac{h_N'}{h_{N-1}} \right\rfloor h_{N-1},\; N \ge 1$
	& $\mathcal Q_N(0,0):=[-\mathcal{L}_N^{\mathrm{side}}, \mathcal{L}_N^{\mathrm{side}}] \times [-\mathcal{L}_0^{\mathrm{side}}, \mathcal{L}_0^{\mathrm{side}}]^{d-1} \times [0,h_N]$\\
	For $m \in \mathbb{Z},\; n \in \mathbb{N}_0$:
    &For $m \in \mathbb{Z},\; n \in \mathbb{N}_0$:\\
	$\mathbf{x}_N(m):=\mathcal{L}_N m \cdot \mathsf{e}_1, $& $\mathcal{Q}_N(m,n):= \vec{\mathbf{x}}_N(m,n)+\mathcal{Q}_N(0,0)$\\
	$\vec{\mathbf{x}}_N(m,n):=\mathbf{x}_N(m)+ h_N n \cdot \mathsf{e}_{d+1}$&\\[.2cm]
\multicolumn{2}{p{\textwidth}}{\textbf{Bad points at scale 0}  }\\
\multicolumn{2}{p{\textwidth}}{$\Tlow :=  (\ell_{\Tlow}:=\mathsf v^{1/(8d)},\;L_{\Tlow}:=\sqrt{\mathsf v}\log^2(\mathsf v),\; t_{\Tlow}:=\mathsf v^{1-2\varepsilon_0},\;p_{\Tlow}:=\tfrac12(\plow + p_0))$, $G_{\mathsf{v}}(\zeta):=g^\downarrow(\Tlow,\; \xi^{\zeta} \cdot \mathds{1}_{B_0(\mathcal{L}_0^{\mathrm{side}})})$ }\\
	\multicolumn{2}{p{0.97\textwidth}}{$(0,0)$ is~$0$-bad for~$(\zeta_t)_{t\ge 0}$ if: \textit{either} $G_\mathsf{v}(\zeta_0) \ge \exp\{-\tfrac12\mathsf v^{\varepsilon_0}\}$ \textit{or} [$G_\mathsf{v}(\zeta_0) < \exp\{-\tfrac12\mathsf v^{\varepsilon_0}\}$ \textit{and} $\zeta_0$ has more than~$\mathsf{v}^{\varepsilon_0}$ infections in~$B_0(\sqrt{\mathsf v})$ \textit{and} $\zeta_{h_0}$ has fewer than~$\mathsf{v}^{\varepsilon_0}$ infections in one of the boxes $B_{\mathbf x_0(-1)}( \sqrt{\mathsf v}),\;B_{0}( \sqrt{\mathsf v}),\; B_{\mathbf x_0(1)}( \sqrt{\mathsf v})$]}\\
	\multicolumn{2}{p{0.97\textwidth}}{$(m,n)$ is~$0$-bad for~$(\zeta_t)_{t\ge 0}$ if $(0,0)$ is~$0$-bad for~$(\zeta_{nh_0+t} \circ \theta(\mathbf{x}_0(m)))_{t \ge 0}$}
	\\[.2cm]
	\multicolumn{2}{p{0.97\textwidth}}{\textbf{Bad points at scale $N \ge 1$}  }\\
	\multicolumn{2}{p{0.97\textwidth}}{$(m,n)$ is~$N$-bad for~$(\zeta_t)_{t \ge 0}$ if there exist~$(i,j),(i',j')$, both~$(N-1)$-bad, such that~$\mathcal{Q}_{N-1}(\vec{\mathbf{x}}_{N-1}(i,j))$ and $\mathcal{Q}_{N-1}(\vec{\mathbf{x}}_{N-1}(i',j'))$ are contained in $\mathcal{Q}_N(m,n)$ \textit{and} [\textit{either} $|j-j'|>1$ \textit{or} [$|j-j'|\le 1$ \textit{and} $|i-i'| > \sqrt{\alphav}$]]}\\
\end{tabular}
}
\end{boxedminipage}\\[.2cm]

We now define~$p_1,p_2,\ldots$ recursively with
\begin{equation}\label{eq_nice_forumulas_p}
	p_{N+1} := \tfrac12 (p_N+p), \quad N \in \mathbb N_0.
\end{equation}

\begin{proposition}\label{prop_higher_scales_0}
	The following holds if~$\sfv$ is large enough. 
	Let~$N \in \mathbb N$ and assume that~$(\zeta_t)$ starts from a random configuration~$\zeta_0$ such that the law of the projection~$\xi^{\zeta_0} \in \{0,1\}^{\mathbb Z^d}$ stochastically dominates~$\pi_{p_N}$. Then, for any~$(m,n)$, the probability that~$(m,n)$ is~$N$-bad for~$(\zeta_t)$ is smaller than~$\alphav^{-8(N+2)}$.
\end{proposition}
We prove this proposition in Section~\ref{s_survival_induction}.

\subsubsection{Completion of proof of survival}
We now show how Corollary~\ref{cor_scale_0_G} and Proposition~\ref{prop_higher_scales_0} can be combined to prove~\eqref{eq_main_surv}, the survival side of Theorem~\ref{thm_main}.

It will be useful to have some estimates on the number of scale~$(N-1)$ boxes that are contained in a scale~$N$ box.
Denote by $\llbracket a,b \rrbracket$ the integer interval $[a,b] \cap \ZZ$. For any~$N\in \mathbb N_0$, we can write
\begin{equation}\label{eq_handy_integers}
	\{(i,j): \mathcal{Q}_{N-1}(i,j) \subset \cQ_N(m,n) \}
    = \llbracket l_N(m), r_N(m) \rrbracket \times \llbracket b_N(n), t_N(n)
    \rrbracket,
\end{equation}
for integers $l_N(m), r_N(m), b_N(n), t_N(n)$ representing the left-, right-, bottom- and top-most extreme indices, respectively.
It is clear that $b_N(n+1) = t_N(n)+1$, but~$l_N(m+1)$ and~$r_N(m)$ do not satisfy this relation, due to the spatial overlap between boxes. We obtain explicit formulas for these indices, starting with~$l_N(m)$, which is the smallest integer~$i$ such that~$\mathbf{x}_{N-1}(i) - \mathcal{L}_{N-1}^{\mathrm{side}} \ge \mathbf{x}_N(m)-\mathcal L_N^{\mathrm{side}}$, so
\begin{align}
	\label{eq_formula_for_lN}
l_N(m)&=\left\lceil \frac{\mathcal{L}_N }{\mathcal{L}_{N-1}}m - \frac{\rho_N \mathcal{L}_N}{\mathcal{L}_{N-1}} +\rho_{N-1} \right\rceil = 
\left\lceil \alpha_{\sfv}^{2N-1}(m - \rho_N) +\rho_{N-1}\right\rceil, \quad \text{and}\\
	\label{eq_formula_for_rN}
r_N(m)&= \left\lfloor \alpha_{\sfv}^{2N-1}(m + \rho_N) -\rho_{N-1}\right\rfloor,
\end{align}
similarly. Since~$h_N$ was taken as an integer multiple of~$h_{N-1}$, we have
\begin{equation}\label{eq_formula_for_bt}
b_N(n)=\frac{h_N}{h_{N-1}}n,\qquad t_N(n)=\frac{h_N}{h_{N-1}}(n+1)-1.
\end{equation}

Next, we note that~$h_N'/h_{N-1}'$ and~$h_N/h_{N-1}$ have the same order of magnitude. 
Indeed, setting~$h_0':=h_0$, note that for all~$N \ge 1$, 
\begin{equation}\label{eq_h_n_prime}
h_N' \ge h_N \ge h_N' - h_{N-1} \ge h_N' - h_{N-1}'\ge (1- \alpha_{\sfv}^{-2N+1})h_N',
\end{equation}
where in the last inequality we used that $\rho_{N-1}/\rho_N < 1$. Using~\eqref{eq_h_n_prime}, we obtain
\begin{gather}
\label{eq_for_case_N}
\begin{aligned}
\frac{h_1'}{h_0'}
    \ge \ \ \frac{h_1}{h_0}\ \
    &\ge (1-\alpha_{\sfv}^{-1})\cdot \frac{h_1'}{h_0'},
    &&\text{for $N = 1$;} \\
\frac{1}{1-\alpha_{\sfv}^{-2N+3}}\cdot \frac{h_N'}{h_{N-1}'}
    \ge \frac{h_N}{h_{N-1}}
    &\ge (1- \alpha_{\sfv}^{-2N+1})\cdot \frac{h_N'}{h_{N-1}'}, 
    &&\text{for $N \ge 2$}.
%     \\
% \implies \quad
%     \frac{h_N}{h_{N-1}} 
%     &= \Bigl(\frac{\rho_N}{\rho_{N-1}}+o(1)\Bigr)\alphav^{2N-1},
%     &&\text{for $N\ge 1$},
\end{aligned}    
\end{gather}
Taking $\sfv$ large, the terms multiplying $\frac{h'_N}{h'_{N-1}}$ get arbitrarily close to 1, uniformly in $N\ge 1$.

\begin{definition}[Accessible points]
\hspace{0cm}
\begin{itemize}
\item 
	A point~$(m,n) \in \mathbb Z \times \mathbb N_0$ is \emph{$0$-accessible} if there are
indices ${0 = m_0, m_1, \ldots, m_n = m}$ such that $|m_{k+1}-m_k| \le 1$ for~$k=0,\ldots,n-1$
		and $(m_0,0),\ldots, (m_n,n)$ are all $0$-good.

	\item Let $N \in \mathbb N$. A point~$(m,n)$ is \emph{$N$-accessible} if,
		among all points~$\{(i, t_N(n)):i \in \llbracket l_N(m), r_N(m)\rrbracket\}$, all are~$(N-1)$-accessible except for at most~$\sqrt{\alpha_{\sfv}}$.
\end{itemize}
\end{definition}

The next lemma shows that accessibility spreads well in a region of good boxes.
\begin{lemma}
\label{lem_accessible_spread}
Let~$m,m' \in \mathbb Z$ with~$|m-m'| \le 1$ and~$n \in \mathbb N_0$.
If $(m,n)$ is $N$-accessible and $(m', n+1)$ is $N$-good, then
$(m',n+1)$ is~$N$-accessible.
\end{lemma}
\begin{proof}
If $N=0$, the statement of the lemma is immediate. Now, we assume that
the statement of the lemma holds for scale $N-1$, and prove that
it also holds for scale $N$. We will only do the proof for the case~$m'=m+1$. The case~$m'=m-1$ is handled by symmetry, and the case~$m'=m$ is easier. Hence, from now on we assume that~$(m,n)$ is~$N$-accessible and~$(m+1,n+1)$ is~$N$-good.
Let $f: \llbracket b_N(n+1), t_N(n+1)\rrbracket \to \NN$ be defined as
\begin{equation*}
	f(j) := |\{i \in \llbracket l_N(m+1), r_N(m+1)\rrbracket: (i,j)\ \text{is $(N-1)$-accessible}\}|.
\end{equation*}
	In words, the function $f$ counts the number of $(N-1)$-accessible points at a fixed height between~$b_N(n+1)$ and~$t_N(n+1)$. The statement that~$(m+1,n+1)$ is~$N$-accessible, can now be expressed as 
	\begin{equation*}
		f(t_N(n+1)) \ge r_N(m+1) - l_N(m+1) + 1 - \sqrt{\alpha_{\sfv}}.
	\end{equation*}

	A first observation in this direction is that~$f(b_N(n+1))$ cannot be too small. To see this, note that, since~$(m,n)$ is~$N$-accessible, we have
    \[|\{i \in \llbracket l_N(m),r_N(m)\rrbracket: (i,t_N(n))\text{ is not $(N-1)$-accessible}\}| \le \sqrt{\alphav},\]
    and since~$(m+1,n+1)$ is~$N$-good, we have
    \[|\{i \in \llbracket l_N(m+1),r_N(m+1)\rrbracket: (i,b_N(n+1))\text{ is $(N-1)$-bad}\}| \le \sqrt{\alphav}.\]
    Recall that~$b_N(n+1)=t_N(n)+1$. The induction hypothesis implies that, if $$i \in \llbracket l_N(m),r_N(m)\rrbracket \cap \llbracket l_N(m+1),r_N(m+1)\rrbracket$$ is such that $(\!i,t_N(n)\!)$ is~$(\!N\!-\!1\!)$-accessible and~$(\!i,b_N(n\!+\!1)\!)$ is~$(\!N\!-\!1\!)$-good, then~$(\!i,b_N\!(n\!+\!1)\!)$ is~$(N-1)$-accessible. This shows that
\begin{align}
f(b_N(n&+1)) 
	\ge r_N(m)- l_N(m+1) + 1 - 2\sqrt{\alpha_\sfv}\nonumber\\
	&\overset{\mathclap{\eqref{eq_formula_for_lN},\eqref{eq_formula_for_rN}}}{=}
	\quad \lfloor \alpha_{\sfv}^{2N-1}(m+\rho_N) -\rho_{N-1}\rfloor - \lceil \alpha_{\sfv}^{2N-1}(m+1-\rho_N)+\rho_{N-1}\rceil  +1 - 2\sqrt{\alpha_\sfv}\nonumber\\
    \label{eq:bound_fb}
	&= \quad \alpha_{\sfv}^{2N-1}(2\rho_N-1) - 2\sqrt{\alpha_{\sfv}} + O(1)
\end{align}
as $\sfv \to \infty$. Let us abbreviate notation by defining
\begin{equation*}
    b := b_N(n\!+\!1),\quad t:=t_N(n+1),\quad r:= r_N(m\!+\!1),\quad l:= l_N(m\!+\!1),
    \quad \text{and} \;
    \mathsf{a} := r - l + 1.
\end{equation*}

Now that we have a lower bound on~$f(b)$, our strategy is to show that the increments of~$f$ are easy to control, using the 	induction hypothesis. 

	By the definition of~$(m+1,n+1)$ being~$N$-good, we can find a box of indices of the form
	\[ \llbracket i^*, i^*+\sqrt{\alpha_{\sfv}} \rrbracket \times \llbracket j^*,j^*+1 \rrbracket \subset \llbracket l,r \rrbracket \times \llbracket b,t\rrbracket\]
	such that, except possibly for~$(i,j)$ inside this box, every~$(i,j)$ in~$\llbracket l,r \rrbracket \times \llbracket b,t\rrbracket$ is~$(N-1)$-good.
%\[
%	(i,j) \in (\llbracket l,r \rrbracket \times \llbracket b+1,t\rrbracket) \backslash (\llbracket i^*, i^*+\alpha_{\sfv}^{N/2} \rrbracket \times \llbracket j^*,j^*+1 \rrbracket) \quad \Longrightarrow \quad (i,j) \text{ is $(N-1)$-good}.
%\]
Using the induction hypothesis, we now note that, for all~$j \in \{b+1,\ldots, t\} \backslash \{j^*,j^*+1\}$, we have:
	\begin{itemize}
    \item if~$f(j-1) < \mathsf{a}$, then~$f(j) \ge f(j-1)+ 1$ (indeed: since~$f(j-1)<\mathsf{a}$, we can find some~$i,i' \in \llbracket l,r \rrbracket$ such that~$|i-i'|=1$,~$(i,j-1)$ is~$(N-1)$-accessible, but~$(i',j-1)$ is not; then, since~$j \notin \{j^*,j^*+1\}$, we have that~$(i',j)$ is~$(N-1)$-good, so it gains the property of being~$(N-1)$-accessible from~$(i,j-1)$);
\item if~$f(j-1) = \mathsf{a}$, then~$f(j)=\mathsf{a}$ (indeed: for every~$i \in \llbracket l,r \rrbracket$, we have that~$(i,j-1)$ is~$(N-1)$-accessible, so~$(i,j)$ gains the property of being~$(N-1)$-accessible from~$(i,j-1)$).
	\end{itemize}
    Moreover, we have~$f(j^*+1) \ge f(j^*-1) - \sqrt{\alphav}$, since at most~$\sqrt{\alphav}$ points lose the property of being~$(N-1)$-accessible due to being in the bad region of indices.
    
	From this, it is readily seen that
	\begin{equation}
		\text{if there exists $j \in \llbracket b, t\rrbracket$ such that~$f(j) = \mathsf{a}$, then~$f(t) \ge \mathsf{a} - \sqrt{\alpha_\sfv}$.}
	\end{equation}
	Let us prove that there indeed exists~$j$ such that~$f(j)=\mathsf{a}$.

	For all~$j \in \{b+1,\ldots, t\}$, the above observations about the increments of~$f$ imply that
\begin{align*}
\text{if } f<\mathsf{a} \text{ on } \{b,\ldots,j\}, \text{ then }
f(j)
	&\ge f(b) - \sqrt{\alpha_\sfv} + j-b-2,
\end{align*}
and, as a consequence,
\begin{equation*}
	\text{if } f(t) < f(b) - \sqrt{\alpha_\sfv}+t-b-2, \text{ then there is $j' \in \{b,\ldots,t\}$ such that $f(j')=\mathsf{a}$}.
\end{equation*}
So, it suffices to prove that~$f(t) < f(b) - \sqrt{\alpha_\sfv}+t-b-2$. Keeping in mind that~$f \le \mathsf{a}$, it suffices to prove that
\[
	f(b) - \sqrt{\alpha_\sfv}+t-b-2 > \mathsf{a}.
\]
It follows from~\eqref{eq_formula_for_lN} and~\eqref{eq_formula_for_rN} that
\begin{equation} \label{eq_new_bound_A}\mathsf{a}=2\rho_N\alpha_\sfv^{2N-1} + O(1).\end{equation}
	Additionally, recalling that $h'_N/h'_{N-1} = (\rho_N/\rho_{N-1}) \alphav^{2N-1}$, it follows from~\eqref{eq_for_case_N} that
\begin{equation}\label{eq_new_bound_A2}
	t-b \ge \frac{h_N}{h_{N-1}}\ge \frac{\rho_N}{\rho_{N-1}}(\alpha_\sfv^{2N-1} - 1).
\end{equation}
Using~\eqref{eq:bound_fb},~\eqref{eq_new_bound_A} and~\eqref{eq_new_bound_A2}, we obtain:
\begin{align}
	\nonumber&f(b) - \sqrt{\alpha_\sfv}+t-b-2 - \mathsf{a} \\
	\nonumber&\ge \alpha_{\sfv}^{2N-1}(2\rho_N-1) - 2\sqrt{\alpha_{\sfv}} - O(1)-\sqrt{\alpha_\sfv} + \frac{\rho_N}{\rho_{N-1}}\alpha_\sfv^{2N-1}  - 2\rho_N\alpha_\sfv^{2N-1} - O(1)\\[-2mm]
	&= \alpha_{\sfv}^{2N-1} \Bigl( \frac{\rho_N}{\rho_{N-1}}-1 \Bigr) - 3\sqrt{\alpha_\sfv}  -O(1).\label{eq_wanted_exxpr}
\end{align}
Recall from~\eqref{eq:defi_rho_N} that~$\rho_N=\rho_{N-1}+ 2^{-N}$. Hence,
\begin{equation*}
\alpha_{\sfv}^{2N-1} \Bigl( \frac{\rho_N}{\rho_{N-1}}-1 \Bigr) = \alpha_{\sfv}^{2N-1} \cdot \frac{2^{-N}}{\rho_{N-1}} \ge \alpha_{\sfv}^{2N-1} \cdot 2^{-N-1}= \alpha_{\sfv}^{2N-1-\frac{\log 2}{\log \alpha_\sfv}(N-1)}.
\end{equation*}
When~$\sfv$ is large enough (so that~$\alpha_\sfv$ is large), uniformly in~$N$, the r.h.s. above is much larger than~$\sqrt{\alpha_\sfv}$. This shows that the expression in~\eqref{eq_wanted_exxpr} is positive, concluding the proof.
\end{proof}

\begin{proposition}\label{prop_thm_from_easy}
	If the interchange-and-contact process is started from a random configuration~$\zeta_0$ with law~$\hat{\pi}^{B_0(\mathcal{L}_0^{\mathrm{side}})}_p$, then the infection stays present at all times with positive probability.
\end{proposition}
\begin{proof}
	Recall that~$p > p_N$ for any~$N$. Consequently, when~$\zeta_0 \sim \hat{\pi}^{B_0(\mathcal{L}_0^{\mathrm{side}})}_p$, the projection~$\xi^{\zeta_t}$ stochastically dominates~$\pi_{p_N}$, for any~$t$ and~$N$. Hence, Corollary~\ref{cor_scale_0_G} implies that
	\begin{equation}
		\label{eq_unnbd1}
	\text{for {all} }(m,n),\quad \mathbb P((m,n) \text{ is $0$-bad}) \le 2 \mathsf v^{-\varepsilon_0/2},
	\end{equation}
	and Proposition~\ref{prop_higher_scales_0} implies that, for {each}~$N \in \mathbb N$, 
	\begin{equation}
		\label{eq_unnbd2}
	\text{for all }(m,n), \quad \mathbb P((m,n) \text{ is $N$-bad}) \le \alphav^{-8(N+2)}.
	\end{equation}

	For~$N \in \mathbb N$, define the event
	\begin{equation*}
		\mathcal A_N := \left\{
        \begin{array}{l}
             \text{$(i,j)$ is $(N-1)$-good for $(\zeta_t)$, for all $(i,j)$}  \\
             \text{such that $\mathcal{Q}_{N-1}(i,j) \subset \mathcal Q_N(0,0) \cup \mathcal Q_N(0,1)$}
        \end{array}
         \right\}.
	\end{equation*}
	The number of~$(i,j)$ such that~$\mathcal{Q}_{N-1}(i,j) \subset \mathcal Q_N(0,0) \cup \mathcal Q_N(0,1)$ is
\[
	2\left\lceil\frac{\mathcal L_N^{\mathrm{side}}}{\mathcal{L}_{N-1}^{\mathrm{side}}} \right\rceil \cdot \frac{h_N}{h_{N-1}} \le 2 \cdot 2\alpha_{\sfv}^{2N-1} \cdot 2\alpha_{\sfv}^{2N-1} = 8\alpha_\sfv^{4N-2}.
\]
	Then, letting~$\mathcal A:= \cap_{N=1}^\infty \mathcal A_N$, by a union bound using~\eqref{eq_unnbd1} and~\eqref{eq_unnbd2}, we have
	\[
		\mathbb P(\mathcal A) \ge 1 - 8\alpha_\sfv^{4 \cdot 1 - 2} \cdot 2\mathsf v^{-\varepsilon_0/2} -   \sum_{N=2}^\infty  8\alpha_\sfv^{4N-2}\cdot \alphav^{-8(N+2)}.
	\]
	By taking~$\sfv$ large enough, using the fact that~$\alpha_\sfv = \lfloor \sfv^{\varepsilon_0/64}\rfloor$, the r.h.s. above can be made positive. We now claim that
	\begin{equation*}
		\mathcal A \subseteq \bigcap_{N=0}^\infty \{(0,1) \text{ is $N$-accessible}\} \subseteq \{\forall t\; \exists x:\; \zeta_t(x) = \statei \}.
	\end{equation*}
	The second inclusion being obvious, we now justify the first. We assume from here on that~$\mathcal A$ occurs, and will prove by induction on~$N$ that~$(0,1)$ is~$N$-accessible for { every}~$N \in \mathbb N_0$.

	For~$N = 0$, this is clear: since~$\mathcal{Q}_0(0,0),\mathcal{Q}_0(0,1) \subset \mathcal Q_1(0,0)$ and~$\mathcal A_1$ occurs, we see that~$(0,0)$ and~$(0,1)$ are both~$0$-good, hence~$(0,1)$ is~$0$-accessible.

	Now let~$N \in \mathbb N$ and assume that we have already proved that~$(0,1)$ is~$(N-1)$-accessible. Using the notation introduced in~\eqref{eq_handy_integers}, we now check that
	\begin{equation}
		\label{eq_order_r_and_t}
	r_N(0) < t_N(1).
	\end{equation}
	Indeed, by~\eqref{eq_formula_for_rN}, we have
	\begin{equation*}
	r_N(0) = \lfloor \alpha_\sfv^{2N-1} \rho_N - \rho_{N-1}\rfloor \le 2 \alpha_\sfv^{2N-1},    
	\end{equation*}
    and by~\eqref{eq_for_case_N}, recalling that $\rho_N-\rho_{N-1} = 2^{-N}$ and~$\rho_{N-1} \in [1,2]$, we have
	\begin{align*}\nonumber
	t_N(1)=2\frac{h_N}{h_{N-1}}
        &\ge 2\frac{\rho_N}{\rho_{N-1}}(\alpha_\sfv^{2N-1} - 1) 
        = 2\Bigl(1+\frac{2^{-N}}{\rho_{N-1}}\Bigr)(\alpha_\sfv^{2N-1} - 1) \\
		&\ge 2\alpha_\sfv^{2N-1} + 2^{-N}\alphav^{2N-1} - 2 - 2^{-N+1}.\label{eq_order2rt}
	\end{align*}
	For $N\ge 1$, we have $2^{-N} \alpha_\sfv^{2N-1} \gg 1$ and the proof of~\eqref{eq_order_r_and_t} is complete.

	Now, by the induction hypothesis we have that~$(0,1)$ is~$(N-1)$-accessible, and then, using Lemma~\ref{lem_accessible_spread} and the assumption that~$\mathcal A_N$ occurs, for all~$j \in \{2,\ldots, t_N(1)\}$,
	\[(i,j)\text{ is }(N-1)\text{-accessible, for } i \in \{(-j+1) \vee l_N(0),\ldots, (j-1) \wedge r_N(0)\}.\]
	By~\eqref{eq_order_r_and_t} and the fact that~$l_N(0)=-r_N(0)$, we conclude that
	\[(i,t_N(1))\text{ is }(N-1)\text{-accessible, for } i \in \{l_N(0),\ldots, r_N(0)\},\]
	and consequently,~$(0,1)$ is~$N$-accessible, as required.
\end{proof}

\begin{proof}[Proof of Theorem~\ref{thm_main},~\eqref{eq_main_surv}]
	Assume the interchange-and-contact process is started from~$\hat{\pi}_p^{\{0\}}$, as in the statement of the theorem.
Let~$\mathcal A$ be the event that:
	\begin{itemize}
		\item $\zeta_0(x)=\stateh$ for all~$x \in B_0(\mathcal{L}_0^{\mathrm{side}} )\backslash \{0\}$;
		\item in the time interval~$[0,1]$, there are no jump marks involving any site in~$B_0(\mathcal{L}_0^{\mathrm{side}})$, and no recoveries at any site in~$B_0(\mathcal{L}_0^{\mathrm{side}})$;
		\item the infection initially present at 0 manages to spread, before time 1, to all particles in $B_0(\mathcal{L}_0^{\mathrm{side}})\backslash \{0\}$ (but it does not leave this box).
	\end{itemize}
	Clearly,~$\mathbb P(\mathcal A)>0$. 

	We now claim that conditionally on~$\mathcal A$, the law of~$\zeta_1$ is~$\hat{\pi}_p^{B_0(\mathcal{L}_0^{\mathrm{side}})}$. Indeed, let~$x_1,\ldots,x_m \in \mathbb Z^d \backslash B_0(\mathcal{L}_0^{\mathrm{side}})$. For~$i=1,\ldots,m$, let~$X_i$ be the random element of~$\mathbb Z^d$ such that~$\Phi(X_i,0,1)=x_i$. Since~$x \mapsto \Phi(x,0,1)$ is a bijection and on~$\mathcal A$ we have~$\Phi(x,0,1)=x$ for all~$x \in B_0(\mathcal{L}_0^{\mathrm{side}})$, it must hold that~$X_i \notin B_0(\mathcal{L}_0^{\mathrm{side}})$ for all~$i$. Then,
	\[\mathbb P(\zeta_1(x_1)=\cdots=\zeta_1(x_m)=0 \mid \mathcal A)=\mathbb P(\zeta_0(X_1)=\cdots=\zeta_0(X_m)=0 \mid \mathcal A)=(1-p)^m,\]
	where the second equality holds because~$\mathcal A$ only involves the initial configuration inside~$B_0(\mathcal{L}_0^{\mathrm{side}})$ and the graphical representation, and these are independent of the initial configuration outside~$B_0(\mathcal{L}_0^{\mathrm{side}})$.

	Having established that the law of~$\zeta_1$ conditionally on~$\mathcal A$ is~$\hat{\pi}_p^{B_0(\mathcal{L}_0^{\mathrm{side}})}$, the conclusion of the theorem now follows from the Markov property and Proposition~\ref{prop_thm_from_easy}.
\end{proof}

\subsection{Induction step}
\label{s_survival_induction}
In this section, we prove Proposition~\ref{prop_higher_scales_0}.
In what follows, we write $\delta_N = \alpha_{\mathsf v}^{-8(N+2)}$, which will serve as an upper bound for the probability that a point is $N$-bad.
Since $\alpha_{\mathsf{v}} = \lfloor \mathsf{v}^{\varepsilon_0/64} \rfloor$ the quantity $\delta_N$ depends on the parameter $\mathsf{v}$ and $\lambda$ of the interchange-and-contact process (recall that $\varepsilon_0$ depends on $\lambda$).

\textbf{Badness estimate at scale $N$ $(\mathrm{BE}_N)$:} 
\begin{equation}\label{eq_induction_hypothesis}\tag{$\mathrm{BE}_N$}
	\xi^{\zeta_0} \text{ stochastically dominates }\pi_{p_N} \quad \Longrightarrow \quad \mathbb P(\mathcal Q_N(0,0) \text{ is bad for }(\zeta_t)) < \delta_N.
\end{equation}
Using the fact that~$\pi_{p_N}$ is stationary for the interchange process, if~$\xi^{\zeta_0}$ stochastically dominates~$\pi_{p_N}$, then~$\xi^{\zeta_t \circ \theta(x)}$ stochastically dominates~$\pi_{p_N}$ as well, for any~$x \in \mathbb Z^d$ and~$t \ge 0$. 
In particular, if hypothesis~\eqref{eq_induction_hypothesis} holds, then we also have
\begin{equation}\label{eq_induction_hypothesis_translated}
	\xi^{\zeta_0} \text{ stochastically dominates }\pi_{p_N} \quad \Longrightarrow \quad \mathbb P(\mathcal Q_N(m,n) \text{ is bad for }(\zeta_t)) < \delta_N \; \text{ for all }(m,n).
\end{equation}

\begin{lemma}[Horizontal decoupling]
	\label{lem_horizontal_decoupling}
	Let~$N \in \mathbb N_0$ and assume that~\eqref{eq_induction_hypothesis} is satisfied.
	Let~$(\zeta_t)_{t \ge 0}$ be the interchange-and-contact process with parameters~$\mathsf v$ and~$\lambda$, started from a random configuration~$\zeta_0$ such that~$\xi^{\zeta_0}$ stochastically dominates~$\pi_{p_N}$. Let~$(m,n),(m',n') \in \mathbb Z \times \mathbb N_0$ be such that~$|n-n'| \le 1$ and~$|m-m'| \ge \sqrt{\alphav}$.
	Then,
	\begin{equation*}
		\mathbb P(\mathcal Q_N(m,n) \text{ and } \mathcal Q_N(m',n') \text{ are both bad for }(\zeta_t)) \le \delta_N^2 + \exp\{-\alphav^{N^2+1/8}\}.
	\end{equation*}
\end{lemma}

\begin{lemma}[Vertical decoupling]
	\label{lem_vertical_decoupling}
	Let~$N \in \mathbb N_0$ and assume that~\eqref{eq_induction_hypothesis} is satisfied.
	Let~$(\zeta_t)_{t \ge 0}$ be the interchange-and-contact process with parameters~$\mathsf v$ and~$\lambda$, started from a random configuration~$\zeta_0$ such that~$\xi^{\zeta_0}$ stochastically dominates~$\pi_{p_{N+1}}$. Let~$(m,n),(m',n') \in \mathbb Z \times \mathbb N_0$ be such that~$n' \ge n+1$.
	Then,
	\begin{equation*}
		\mathbb P(\mathcal Q_N(m,n) \text{ and } \mathcal Q_N(m',n') \text{ are both bad for }(\zeta_t)) \le \delta_N^2 + 3\exp\{- \alphav^{(N^2+1)/8}\}.
	\end{equation*}
\end{lemma}

\begin{proposition}[Induction step] 
\label{prop:induction_step}
    Let~$N \in \mathbb N_0$ and assume that~\eqref{eq_induction_hypothesis} is satisfied.
	Let~$(\zeta_t)_{t \ge 0}$ be the interchange-and-contact process with parameters~$\mathsf v$ and~$\lambda$, started from a random configuration~$\zeta_0$ such that~$\xi^{\zeta_0}$ stochastically dominates~$\pi_{p_{N+1}}$. 
	Then,
 \begin{equation}
 \mathbb P(\mathcal Q_{N+1}(0,0) \text{ is bad for }(\zeta_t)) \le { 16\alphav^{8N+4}}\cdot (\delta_N^2 + 3 \exp\{-\alphav^{(N^2+1)/8}\})
 \label{e:induction_step}
 \end{equation}
\end{proposition}
\begin{proof}
   The number of pairs of boxes of scale~$N$ intersecting $\mathcal{Q}_{N+1}(0,0)$ is at most
   \[
        \Big(\frac{\mathcal{L}_{N+1}^{\mathrm{side}}}{\mathcal{L_N}} \cdot \frac{h_{N+1}}{h_N}\Big)^2 \le 16\alphav^{8N+4}.
   \]
   The result then follows from the previous two lemmas together with a union bound.
\end{proof}

\begin{proposition}\label{prop_higher_scales_00}
	If~$\sfv$ is large enough, then \eqref{eq_induction_hypothesis} holds for every $N \in \mathbb{N}$.
\end{proposition}

\begin{proof}
Let $\mathsf v_o= \mathsf v_o (\lambda, p)>0$ be such that for every $\mathsf{v} \geq \mathsf{v}_o$ one has
\begin{align}
\label{e:delta_N_larger_exponential}
\delta_{N}^2=\alpha_{\mathsf v}^{-16(N+2)}
    &\geq 3 \exp\{-\alpha_{\mathsf{v}}^{(N^2+1)/8}\} \qquad \text{uniformly over $N \geq 0$, and}\\
\label{e:alpha_v_larger_4}
 \alpha_{\mathsf{v}}
    &> 4,
\end{align}
which is possible because $\alpha_{\mathsf{v}} = \lfloor \mathsf{v}^{\varepsilon_0/64}\rfloor \to \infty$  as $\mathsf{v} \to \infty$. Since we are assuming that $\xi^{\zeta_0} \in \{0,1\}^{\mathbb Z^d}$ stochastically dominates~$\pi_{p_N}$ (hence it dominates $\pi_{p_0}$), Corollary \ref{cor_scale_0_G} ensures that 
\begin{equation}
\label{e:induction_hypothesis_delta}
\mathbb P(\mathcal Q_0(m,n) \text{ is bad for }(\zeta_t)) < \mathsf v^{-\varepsilon_0/2} \le \lfloor \sfv \rfloor^{-\varepsilon_0/2} < \alpha_{\mathsf{v}}^{-8(0+2)} = \delta_0.
\end{equation}

Now assume that for a given $N-1$, \eqref{eq_induction_hypothesis_translated} holds.
Our goal is to show that it also holds for $N$.
Assume that $\xi^{\zeta_0}$ stochastically dominates $\pi_{p_{N}}$, so it also dominates $\pi_{p_{N-1}}$, hence we can divide both sides in \eqref{e:induction_step} by $\delta_{N}$ and use \eqref{e:delta_N_larger_exponential} in order to obtain 
\begin{align*}
\delta_{N}^{-1}\cdot \mathbb{P}(\mathcal Q_{N}(0,0) \text{ is bad for }(\zeta_t))
    &\le 16\alphav^{8N-4} \cdot (\delta_{N-1}^2 + 3 \exp\{-\alphav^{{((N-1)}^2+1)/8}\})\delta_{N}^{-1} \\
    &\le  32\alphav^{8N-4}\cdot \delta_{N-1}^2 \delta_{{N}}^{-1} .
 \end{align*}

Now recalling that $\delta_{N-1} = \alpha_{\mathsf v}^{-8(N+1)}$ we get for $\mathsf{v}\geq \mathsf{v}_o$,
\begin{equation}
\label{e:induction_step_delta}
\delta_{N}^{-1}\cdot \mathbb{P}(\mathcal Q_{N}(0,0) \text{ is bad for }(\zeta_t))
    \leq 32\alphav^{8N-4}  \alphav^{-16(N+1)}\alphav^{8(N+2)} = 32 \alpha_{\mathsf{v}}^{-4} < 1,
\end{equation}
where the last inequality follows from \eqref{e:alpha_v_larger_4}.
Using \eqref{e:induction_step_delta} and \eqref{e:induction_hypothesis_delta}, it follows that
\[
\mathbb P(\mathcal Q_N(m,n) \text{ is bad for }(\zeta_t)) \leq \delta_{N}.\qedhere
\]
\end{proof}

The following two subsections
prove Lemma~\ref{lem_horizontal_decoupling} and Lemma~\ref{lem_vertical_decoupling}.

\subsubsection{Horizontal decoupling: proof of Lemma~\ref{lem_horizontal_decoupling}}
\begin{proof}[Proof of Lemma~\ref{lem_horizontal_decoupling}]
	Fix~$(m,n),(m',n')$ as in the statement of the lemma; assume without loss of generality that~$n \le n'$. Let~$\mathcal A$ be the event that~$\mathcal Q_N(m,n)$ is bad, and~$\mathcal A'$ the event that~$\mathcal Q_N(m',n')$ is bad. By Remark~\ref{rem_dependence},~$\mathcal A$ can be determined from the values of~$\zeta_t(x)$ for~$(x,t)$ in~$\mathcal{Q}_N(m,n)$, and~$\mathcal A'$ can be determined from the values of~$\zeta_t(x)$ for~$(x,t)$ in~$\mathcal{Q}_N(m',n')$. We bound
	\[\mathbb P(\mathcal A \cap \mathcal A') \le \mathbb P(\mathcal A) \cdot \mathbb P(\mathcal A') + |\mathrm{Cov}(\mathds{1}_\mathcal{A},\mathds{1}_{\mathcal{A}'})| \le \delta_N^2 + |\mathrm{Cov}(\mathds{1}_\mathcal{A},\mathds{1}_{\mathcal{A}'})|,\]
	where the second inequality follows from~\eqref{eq_induction_hypothesis}. By Lemma~\ref{lem_covariances}, 
	\begin{equation}
		\label{eq_final_cov}
		|\mathrm{Cov}(\mathds{1}_\mathcal{A},\mathds{1}_{\mathcal{A}'})| \le 4\mathrm{discr}^{\mathrm{icp}}_{\mathsf v, \lambda}(\mathcal L_N^{\mathrm{side}},\;\lfloor \tfrac12 \|\mathbf{x}_N(m) - \mathbf{x}_N(m')\| \rfloor,\; h_N (n'+1) - h_N n)).
	\end{equation}
	The value of~$\mathrm{discr}^{\mathrm{icp}}_{\mathsf v,\lambda}(\ell,L,t)$ is non-increasing in~$L$ and non-decreasing in~$t$. Using the assumptions on~$(m,n),(m',n')$, we bound
    \begin{align*}
    \lfloor \tfrac12 \|\mathbf{x}_N(m) - \mathbf{x}_N(m')\| \rfloor 
        &= \lfloor \tfrac12   \mathcal L_N |m - m' |\rfloor \ge \tfrac14 \mathcal L_N \sqrt{\alphav}, \\
    h_N(n'+1) - h_N n 
        &\le 2h_N.
    \end{align*}
	Then, the r.h.s. of~\eqref{eq_final_cov} is at most
    \begin{equation}\label{eq_the_same_term}
4 \mathrm{discr}^{\mathrm{icp}}_{\mathsf v, \lambda}(\mathcal L_N^{\mathrm{side}}, \tfrac14 \mathcal L_N \sqrt{\alphav}, 2h_N).
    \end{equation}
	By Proposition~\ref{prop_discrepancy}, this is bounded from above by
	\begin{equation}\begin{split}
		&4 \cdot 64d^3e^2 \max(4d^2\mathsf v^2,1) \cdot \Bigl(9 \cdot \mathcal L_N^{\mathrm{side}} \cdot \frac 14 \mathcal L_N \sqrt{\alphav}\Bigr)^{d-1} \cdot 2h_N \exp\{8d\lambda \cdot 2h_N\} \\
        &\hspace{1cm}\cdot \exp\Bigl\{-\frac12 \Bigl(\frac14 \mathcal L_N  \sqrt{\alphav} - \mathcal L_N^{\mathrm{side}} \Bigr) \log \Bigl(1+ \frac{\tfrac14 \mathcal{L}_N \sqrt{\alphav} - \mathcal L_N^{\mathrm{side}}}{4(\mathsf v + \lambda) \cdot (2h_N)} \Bigr) \Bigr\}.
	\end{split} \label{eq_final_for_discr_icp}
	\end{equation}
In order to deal with this expression, let us recall that
\vspace{-2mm}%
\begin{align*}
\mathcal L_N = \lfloor \sqrt{\sfv} \rfloor \alphav^{N^2},\; N \ge 0
\qquad \text{and}\qquad
\mathcal L_N^{\mathrm{side}} = \begin{cases}
    2\lfloor \sqrt{\sfv} \rfloor \log^2(\sfv) & \text{if }N=0;\\
    \rho_N  \lfloor \sqrt{\sfv} \rfloor \alphav^{N^2} &\text{if } N \ge 1,
\end{cases}
\end{align*}
and also that $h_N \le h'_N \le 2 h_0 \alphav^{N^2}$ for all $N \ge 0$.
We can check that for large $\sfv$
\begin{align}\label{eq_aux_aux10}
    \tfrac14 \mathcal L_N  \sqrt{\alphav} - \mathcal L_N^{\mathrm{side}} 
    &\ge \sqrt{\sfv} \cdot \alphav^{N^2+1/4},  &&\text{for $N \in \mathbb N_0$};\\
\label{eq_aux_aux20}
    4(\mathsf v + \lambda) \cdot (2h_N) 
    &\le 16 \mathsf v h_N \le 32 \sfv h_0 \alphav^{N^2}, &&\text{for $N \in \mathbb N_0$.}
\end{align}
Using~\eqref{eq_aux_aux10} and~\eqref{eq_aux_aux20}, we obtain
\[
\Bigl(\frac14 \mathcal L_N  \sqrt{\alphav} - \mathcal L_N^{\mathrm{side}} \Bigr) \log \biggl(1+ \frac{\tfrac14 \mathcal L_N\sqrt{\alphav} - \mathcal L_N^{\mathrm{side}}}{4(\mathsf v + \lambda) \cdot (2h_N)} \biggr)
\ge \sqrt{\sfv} \cdot \alphav^{N^2+1/4} \log\biggl(1+\frac{\sqrt{\sfv} \cdot \alphav^{N^2+1/4}}{  32 \sfv h_0 \alphav^{N^2} } \biggr).\]
Using~$\log(1+x) \ge x/2$ for small~$x$ and bounding~$\sqrt{\sfv} \cdot \alphav^{N^2+1/4}/( 32 \sfv h_0 \alphav^{N^2} ) \ge 1/\sqrt{\sfv}$, the above is larger than
\[
 \sqrt{\sfv} \cdot \alphav^{N^2+1/4} \cdot \frac{1}{2\sqrt{\sfv}}  =  \frac12  \alphav^{N^2+1/4} .\]
Having this in mind, we have
\begin{equation*}
    \exp\biggl\{8d\lambda \cdot 2h_N -\frac12 \Bigl(\frac14 \mathcal L_N  \sqrt{\alphav} - \mathcal L_N^{\mathrm{side}} \Bigr) \log \biggl(1+ \frac{\tfrac14 \mathcal L_N \sqrt{\alphav} - \mathcal L_N^{\mathrm{side}}}{4(\mathsf v + \lambda) \cdot (2h_N)} \biggr) \biggr\}
    \le \exp\Bigl\{ 32d\lambda h_0 \alphav^{N^2}\!\! - \frac14  \alphav^{N^2+1/4} \Bigr\}.
\end{equation*}
Now, when~$\sfv$ is large the r.h.s. above is much smaller than~$\exp\{-\alphav^{N^2+1/8}\}$, the quotient between the two values being small uniformly over~$N$. It is easy to check that the contribution of the remaining terms in~\eqref{eq_final_for_discr_icp} is negligible in comparison, so the proof is complete.
\end{proof}

\subsubsection{Vertical decoupling: proof of Lemma~\ref{lem_vertical_decoupling}}
\begin{lemma}[Bad box at height 0, starting from random configuration with occupancy dominating $\pi_{p_N}$ inside a large box] \label{lem_feio}
	Let~$N \in \mathbb N_0$ and assume that~\eqref{eq_induction_hypothesis} is satisfied for some choice of~$\delta_N$. Let~$(\zeta_t)_{t \ge 0}$ be the interchange-and-contact process with parameters~$\mathsf v$ and $\lambda$ started from a random configuration~$\zeta_0$. Assume that the distribution of~$\zeta_0$ is such that
	\[(\xi^{\zeta_0}(y): y \in B_{\mathbf{x}_N(0)}(\sqrt{\alphav}\mathcal{L}_N)) \]
	stochastically dominates a product Bernoulli measure with parameter~$p_N$ in~$B_{\mathbf{x}_N(0)}(\!\sqrt{\alphav}\mathcal{L}_N\!)$. Then,  
	\begin{equation*}
		\mathbb P(\cQ_N(0,0) \text{ is bad for }(\zeta_t)_{t \ge 0}) \le \delta_N + \mathrm{discr}^\mathrm{icp}_{\mathsf v, \lambda}(\mathcal{L}_N^{\mathrm{side}},\sqrt{\alphav}\mathcal{L}_N,h_N).
	\end{equation*}
\end{lemma}

\begin{proof}
	We assume that the process is obtained from a graphical construction. Using extra randomness (independently of the graphical construction), we define a random configuration~$\zeta_0' \in \{0,\stateh, \statei\}^{\mathbb Z^d}$ by setting~$\zeta_0'(x) = \zeta_0(x)$ for every~$x \in B_{\mathbf{x}_N(0)}(\sqrt{\alphav}\mathcal{L}_N)$, and
\begin{equation} \label{eq_outside} x \in B_{\mathbf{x}_N(0)}(\sqrt{\alphav}\mathcal{L}_N)^c \quad \Longrightarrow\quad   \zeta'_0(x) = \begin{cases} \stateh & \text{with probability }p_N;\\0&\text{with probability } 1-p_N\end{cases}\end{equation}
(independently over~$x$).
We define the interchange-and-contact process~$(\zeta'_t)_{t \ge 0}$ started from~$\zeta_0'$, using the same graphical construction as the one for~$(\zeta_t)_{t \ge 0}$. We define the event 
\[\mathcal A:= \{\zeta_t(x) = \zeta'_t(x) \text{ for all }(x,t) \in B_{\mathbf{x}_N(0)}(\mathcal{L}^{\mathrm{side}}_N) \times [0,h_N]\}\]
and bound
\begin{align*}
\mathbb P(\cQ_N(0,0) \text{ is bad for }(\zeta_t)_{t \ge 0}) 
    &\le \mathbb P(\mathcal A \cap \{\cQ_N(0,0) \text{ is bad for }(\zeta_t)_{t \ge 0} \}) + \mathbb P(\mathcal A^c) \\
    &\le \delta_N + \mathrm{discr}^\mathrm{icp}_{\mathsf v, \lambda}(\mathcal{L}_N^{\mathrm{side}},\sqrt{\alphav}\mathcal{L}_N,h_N),
\end{align*}
where in the second inequality we have used the definition of discrepancy (Definition~\ref{def_discr_icp}) together with Lemma~\ref{lem_contain}, as well as~\eqref{eq_induction_hypothesis}.
\end{proof}

Recall the definition of~$\mathrm{err}_{\mathrm{coup}}$ from~\eqref{eq_err_coup}.

\begin{lemma}[Bad box at height $h_N$, starting from deterministic configuration] \label{lem_muito_feio}
Let~$N \in \mathbb N_0$ and assume that~\eqref{eq_induction_hypothesis} is satisfied. Let~$(\zeta_t)_{t \ge 0}$ be the interchange-and-contact process with parameters~$\mathsf v$ and $\lambda$ started from a deterministic configuration~$\zeta_0$. For~$N \ge 0$, let
	\begin{align}
		\label{eq_choice_of_thetaN}&\Theta_N := (\ell_{\Theta_N} := \mathcal L_N^{1/(4d)},\; L_{\Theta_N} := 4\sqrt{\alphav} \mathcal L_N,\; t_{\Theta_N} := \sfv h_N,\; p_{\Theta_N} := \tfrac12(p_N+p_{N+1})),\\
		&\Theta_N' := (\ell_{\Theta_N},\; L_{\Theta_N},\; t_{\Theta_N},\; T = \mathsf v h_N).\label{eq_def_ThetaNp}
	\end{align}
We then have
	\begin{align*}
	\mathbb P(&\cQ_N(0,1) \text{ is bad for }(\zeta_t)_{t \ge 0}) \\[-1mm]
        &\le \delta_N + g^{\downarrow}(\Theta_N, \zeta_0 ) + \mathrm{discr}^\mathrm{icp}_{\mathsf v, \lambda}(\mathcal{L}_N^{\mathrm{side}},\sqrt{\alphav}\mathcal{L}_N,h_N) + \int g^\uparrow(\Theta_N,\xi)\;\pi_{p_N}(\mathrm{d}\xi) + \mathrm{err}_\mathrm{coup}(\Theta_N').
	\end{align*}
\end{lemma}
\begin{proof}
Fix~$\zeta_0$ and let~$(\zeta_t)_{t \ge 0}$ be as in the statement of the lemma. 
 By the Markov property, we have
\begin{align}\label{eq_def_of_f_bad}
    \mathbb P(\mathcal Q_N(0,1) \text{ is bad for }(\zeta_t)_{t \ge 0}) = \mathbb P(\mathcal Q_N(0,0) \text{ is bad for }(\zeta_{h_N+t})_{t \ge 0}) = \mathbb E[f(\zeta_{h_N})],
\end{align}
where for~$\zeta' \in \{0,\stateh,\statei\}^{\mathbb Z^d}$, we define~$f(\zeta')$ as the probability that~$\mathcal Q_N(0,0)$ is bad for an interchange-and-contact process with parameters~$\mathsf v$ and~$\lambda$ started from~$\zeta'$. Define
\[\mu := \text{law of }\xi^{\zeta_{h_N}},\qquad \bar{\mu} := \text{law of the pair }(\xi^{\zeta_{h_N}},\zeta_{h_N}).\]

Next, using Lemma~\ref{lem_coupling_rate_one}, we can obtain a probability measure~$\bar{\nu}$ on~$\{0,1\}^{\mathbb Z^d} \times \{0,1\}^{\mathbb Z^d}$ such that
\[\text{if } (\xi,\xi') \sim \bar{\nu},\quad \text{then } \xi \sim \mu \text{ and } \xi' \sim \pi_{p_N}\]
and moreover,
\begin{align*}
&\bar \nu(\{\xi(x) \ge \xi'(x) \ \forall x \in B_0(\!\sqrt{\alphav}\mathcal{L}_N)\}) \!\ge\! 1 \!- \!\!\!\int\!\! g^{\uparrow}(\Theta_N,\cdot )\; \pi_{p_N}(\mathrm{d}\cdot) \!-\! g^{\downarrow}(\Theta_N,\zeta_0) \!-\! \mathrm{err}_{\mathrm{coup}}(\Theta_N').
\end{align*}

Let
\[\Omega_3:= \{(\xi,\zeta,\xi'):\; \xi, \xi' \in \{0,1\}^{\mathbb Z^d},\; \zeta \in \{0,\stateh,\statei\}^{\mathbb Z^d},\; \xi = \xi^{\zeta}  \}.\]
We now construct a probability measure~$\kappa$ on~$\Omega_3$ such that
\[
\text{if } (\xi,\zeta,\xi') \sim \kappa,\quad \text{ then } (\xi,\zeta) \sim \bar{\mu} \text{ and }(\xi,\xi') \sim \bar{\nu}.
\]
This can be achieved as follows. Using regular conditional probabilities, we let~$K$ and~$K'$ be the probability kernels such that
\[\bar{\mu}(A \times B) = \int_A K(\xi,B) \;\mu(\mathrm{d}\xi),\qquad \bar{\nu}(A \times C) = \int_A K'(\xi,C)\;\mu(\mathrm{d}\xi).\]
Then, we construct~$\kappa$ using an extension theorem with the prescription that
\[\kappa(A \times B \times C) = \int_A \mu(\mathrm{d}\xi) K(\xi,B)\cdot K'(\xi,C),\]
that is, the second and third coordinates are independent, given the first.
Let
\[\mathcal A:= \{(\xi,\zeta,\xi') \in \Omega_3: \xi(x) \ge \xi'(x) \text{ for all } x \in B_0(\sqrt{\alphav}\mathcal{L}_N)\}.\]
We define a function~$Z: \Omega_3 \to \{0,\stateh, \statei\}^{\mathbb Z^d}$ as follows:
\[
\text{if }(\xi,\zeta,\xi') \in \mathcal A, \text{ set } Z(\xi,\zeta,\xi') = \zeta;\qquad \text{otherwise, set }[Z(\xi,\zeta,\xi')](x) = \stateh \text{ for all }x.
\]
Note that by construction,
\[\{x \in B_0(\sqrt{\alphav}\mathcal{L}_N): [Z(\xi,\zeta,\xi')](x)\neq 0\} \supseteq \{x  \in B_0(\sqrt{\alphav}\mathcal{L}_N): \xi'(x) = 1\}.\] Hence, when~$(\xi,\zeta,\xi') \sim \kappa$, we have that~$Z(\xi,\zeta,\xi')$ is a random element of~$\{0,\stateh,\statei\}^{\mathbb Z^d}$ whose projection to~$\{0,1\}^{\mathbb Z^d}$ stochastically dominates~$\pi_{p_N}$ inside~$B_0(\sqrt{\alphav}\mathcal{L}_N)$. Recalling the function~$f$ from~\eqref{eq_def_of_f_bad} and using Lemma~\ref{lem_feio}, we then have
\begin{align*}
\int_{\mathcal{A}} f(Z)\;\mathrm{d}\kappa \le \delta_N + \mathrm{discr}^\mathrm{icp}_{\mathsf v, \lambda}(\mathcal{L}_N^{\mathrm{side}},\sqrt{\alphav}\mathcal{L}_N,h_N).
\end{align*}
Finally, since~$f$ is bounded by~$1$,
\begin{align*}
&\int_{\Omega_3} f(\zeta)\;\kappa(\mathrm{d}(\xi,\zeta,\xi')) \le \int_{\mathcal A} f(Z)\;\mathrm{d}\kappa + \kappa(\mathcal A^c)\\
&\le \delta_N + \mathrm{discr}^\mathrm{icp}_{\mathsf v, \lambda}(\mathcal{L}_N^{\mathrm{side}},\sqrt{\alphav}\mathcal{L}_N,h_N) + \int g^{\uparrow}(\Theta_N,\cdot )\; \pi_{p_N}(\mathrm{d}\cdot ) + g^{\downarrow}(\Theta_N,\zeta_0) + \mathrm{err}_{\mathrm{coup}}(\Theta_N').\qedhere
\end{align*}
\end{proof}

\begin{proof}[Proof of Lemma~\ref{lem_vertical_decoupling}]
Several steps of this proof are identical to the corresponding steps in the proof of Lemma~\ref{lem_vertical_decoupling_ext}. However, since there are important differences in the renormalization schemes and constants between Section~\ref{s_proof_extinction} and our current setting, we carry out the proof in full.

Fix~$(m,n)$ and~$(m',n')$ with~$n' \ge n+1$. 
Let~$(\zeta_t)_{t \ge 0}$ be the interchange-and-contact process with parameters~$\mathsf v$ and~$\lambda$, and assume that~$\zeta_0$ is random and such that~$\xi^{\zeta_0}$ stochastically dominates~$\pi_{p_{N+1}}$.
We abbreviate $\tilde \zeta := \zeta_{h_N(n'-1)} \circ \theta(\mathbf{x}_N(m'))$ and let
\[a:=\int g^\downarrow(\Theta_N, \xi)\;\pi_{p_{N+1}}(\mathrm{d}\xi)
\quad \text{and define the event} \quad \mathcal{A} := \left\{g^\downarrow(\Theta_N,\tilde{\zeta}) > \sqrt{a} \right\}.\]
Since~$\xi^{\zeta_0}$ stochastically dominates~$\pi_{p_{N+1}}$ and Bernoulli product measures are stationary for the interchange dynamics, we obtain that for any~$t$ and~$x$,~$\xi^{\zeta_{t}\circ \theta(x)} $ stochastically dominates~$\pi_{p_{N+1}}$ as well. Hence, by Markov's inequality and monotonicity of~$g^\downarrow$,
\begin{align*}
    \mathbb P(\mathcal A) \le a^{-1/2} \cdot \mathbb E[g^\downarrow(\Theta_N, \tilde{\zeta})] \le a^{-1/2} \cdot \int g^\downarrow(\Theta_N,\xi) \;\pi_{p_{N+1}}(\mathrm{d}\xi) = \sqrt{a}.
\end{align*}
Next, letting~$(\mathcal F_t)_{t \ge 0}$ be the natural filtration associated to~$(\zeta_t)$, Lemma~\ref{lem_muito_feio} implies that
\begin{align*}
    &\mathbb P(\mathcal Q_N(m',n') \text{ is bad for }(\zeta_t) \mid \mathcal F_{h_N(n'-1)}) \\ &\le \delta_N + g^\downarrow(\Theta_N, \tilde{\zeta})+ \mathrm{discr}^\mathrm{icp}_{\mathsf v, \lambda}(\mathcal{L}_N^{\mathrm{side}},\sqrt{\alphav}\mathcal{L}_N,h_N) + \int g^\uparrow(\Theta_N,\xi)\;\pi_{p_N}(\mathrm{d}\xi) + \mathrm{err}_\mathrm{coup}(\Theta_N').
\end{align*}
Hence, 
\begin{align*}
    \text{on } \mathcal A^c, \quad \mathbb P(\mathcal Q_N(m',n') \text{ is bad for }(\zeta_t) \mid \mathcal F_{h_N(n'-1)})\le \delta_N + \mathcal{E},
\end{align*}
where
\[\mathcal{E}:=  \sqrt{a} + \mathrm{discr}^\mathrm{icp}_{\mathsf v, \lambda}(\mathcal{L}_N^{\mathrm{side}},\sqrt{\alphav}\mathcal{L}_N,h_N)  + \int g^\uparrow(\Theta_N,\xi)\;\pi_{p_N}(\mathrm{d}\xi) + \mathrm{err}_\mathrm{coup}(\Theta_N').\]
We are now ready to bound
\begin{align}
\nonumber&\mathbb P(\mathcal Q_N(m,n) \text{ and }\mathcal Q_N(m',n') \text{ are both bad for }(\zeta_t))\\
    \nonumber&= \mathbb E[ \mathds{1}\{\mathcal Q_N(m,n) \text{ is bad for }(\zeta_t)  \} \cdot \mathbb P(\mathcal Q_N(m',n') \text{ is bad for }(\zeta_t) \mid \mathcal F_{h_N(n'-1)})]\\
    &\le \mathbb P(\mathcal{A}) + (\mathcal E + \delta_N) \cdot \mathbb P(\mathcal Q_N(m,n) \text{ is bad for }(\zeta_t)) \le \sqrt{a} + \mathcal E \delta_N + \delta_N^2 \le \sqrt{a} + \mathcal E + \delta_N^2.\label{eq_final_with_a_and_E}
\end{align}

We now turn to bounding all the error terms that we have gathered along the way. For convenience, we recall that
%\[
%\mathcal{L}_N = \lfloor \sqrt{\sfv} \rfloor \alphav^{N^2},\quad h_N'=h_0 \alphav^{N^2}, \quad \tfrac12 h_N' \le h_N \le h_N'
%\]
%and
\begin{align*}
\ell_{\Theta_N} = \mathcal L_N^{1/(4d)},\quad L_{\Theta_N} = 4\sqrt{\alphav} \mathcal L_N,\quad t_{\Theta_N} = \sfv h_N,\quad p_{\Theta_N} = \tfrac12(p_N+p_{N+1}),\quad N \ge 0.
\end{align*}
\textbf{Bound on~$\sqrt{a}$.}
Using Lemma~\ref{lem_integral_gs}, we bound
\begin{align}
\nonumber a&\le e(8\sqrt{\alphav}\mathcal{L}_N + 1)^d \cdot ((2\mathcal{L}_N^{1/(4d)}+2)^d\sfv h_N+1) \cdot \exp\Bigl\{-\frac12 (2\mathcal{L}_N^{1/(4d)}+1)^d(p_{N+1}-p_N)^2 \Bigr\}\\
\label{eq_start_bound_integralg}
&\le C \alphav^{d/2}\cdot \mathcal{L}_N^{d+1/4} \cdot \sfv  h_N \cdot \exp\Bigl\{-c \mathcal{L}_N^{1/4}(p_{N+1}-p_N)^2 \Bigr\},
\end{align}
where~$c,C$ are positive constants that do not depend on~$\sfv$ or~$N$.
Recall from~\eqref{eq_nice_forumulas_p} that~$p_{N+1}-p_N = 2^{-(N+2)}(p - \underline{p})$. Also using~$\mathcal{L}_N=\lfloor \sqrt{\sfv}\rfloor \alphav^{N^2}$ and~$h_N \le h_N' \le 2 h_0\alphav^{N^2}$, the above is smaller than
\[
C\alphav^{d/2} \cdot (\lfloor \sqrt{\sfv}\rfloor \alphav^{N^2})^{d+1/4} \cdot \sfv h_0 \alphav^{N^2} \cdot \exp\left\{-c(\lfloor\sqrt{\sfv}\rfloor \alphav^{N^2})^{1/4} \cdot 2^{-2N-4}(p-\plow)^2 \right\}.
\]
Since~$p$ and~$\underline{p}$ are fixed and do not depend on~$\mathsf v$, we can take~$\mathsf v$ large enough (uniformly over~$N$) so that the above expression is smaller than~$\exp\{-\sfv^{1/8}\alphav^{\smash{N^2}/8}\}$. Since~$\alphav \ll {\sfv}$, this is in turn much smaller than~$\exp\{-\alphav^{(N^2+1)/8}\}$. We have thus proved that
\[\sqrt{a} = \Bigl( \int g^\downarrow(\Theta_N,\xi)\;\pi_{p_{N+1}}(\mathrm{d}\xi)\Bigr)^{\smash{1/2}} \le \exp\{- \alphav^{(N^2+1)/8}\}.\]

\textbf{Bound on $\int g^\uparrow(\Theta_N,\xi)\pi_{p_N}(\mathrm{d}\xi)$.} Lemma~\ref{lem_integral_gs} gives the exact same bound obtained for $a$.
% , thus 
% \begin{equation*}\int g^\uparrow(\Theta_N,\xi) \; \pi_{p_N}(\mathrm{d}\xi) \le \exp\{- \alphav^{(N^2+1)/8}\}. \end{equation*}

\bigskip
\textbf{Bound on~$\mathrm{discr}^{\mathrm{icp}}_{\mathsf v,\lambda}(\mathcal{L}_N^{\mathrm{side}}, \sqrt{\alphav}\mathcal{L}_N, h_N)$.} 
In the proof of Lemma~\ref{lem_horizontal_decoupling} we have bounded the expression~\eqref{eq_the_same_term}, which is essentially the same as the one we have here, apart from constant factors ($4, 1/4$ and $2$) which make no difference. Hence, the same argument as in that proof shows that
\begin{equation*}
    \mathrm{discr}^{\mathrm{icp}}_{\mathsf v,\lambda}(\mathcal{L}_N^{\mathrm{side}}, \sqrt{\alphav}\mathcal{L}_N, h_N) \le \exp\{-\alphav^{N^2+1/8}\}.
\end{equation*}

\textbf{Bound on~$\mathrm{err}_{\mathrm{coup}}(\Theta_N')$.}
Recall from~\eqref{eq_err_coup} that
\[
\mathrm{err}_{\mathrm{coup}}(\ell,L,t,T) := |B_0(L/2)|\cdot \left( 1 - \mathrm{meet}(\ell) \right)^{\lfloor t/\ell^2 \rfloor} + \mathrm{discr}^{\mathrm{ip}}(L/4, L/2, T),
\]
and recall from~\eqref{eq_def_ThetaNp} that $\Theta_N' :=  (\ell_{\Theta_N},\; L_{\Theta_N},\; t_{\Theta_N},\; T = \mathsf v h_N)$. Hence,
\begin{align*}
    \mathrm{err}_{\mathrm{coup}}(\Theta_N') \!=\! |B_0(2\sqrt{\alphav}\mathcal L_N)| \!\cdot\! (1\!-\!\mathrm{meet}(\mathcal{L}_N^{\frac{1}{4d}} ))^{\lfloor \mathsf v h_N/\mathcal{L}_N^{\frac{1}{2d}}\rfloor} \!\!+\! \mathrm{discr}^{\mathrm{ip}}(\sqrt{\alphav}\mathcal{L}_N,2\sqrt{\alphav}\mathcal{L}_N, \mathsf v h_N).
\end{align*}

We combine~\eqref{eq_better_bound_meet} with the bound
\[
\frac{t}{\ell^2} \cdot \mathds{1}_{\{d=1\}}+ \frac{t}{\ell^2 \log(\ell)} \cdot \mathds{1}_{\{d =2\}} +\frac{t}{\ell^2} \cdot \mathds{1}_{\{d \ge 3\}} \ge \frac{t}{\ell^{2d}}
\]
(which holds for~$\ell$ large enough) to obtain~$( 1 - \mathrm{meet}(\ell) )^{\lfloor t/\ell^2 \rfloor} \le e^{-c t/\ell^{2d}}$ for all~$d$, so
\begin{align*}
    |B_0(2\sqrt{\alphav}\mathcal L_N)| \cdot (1-\mathrm{meet}(\mathcal{L}_N^{\frac{1}{4d}} ))^{\lfloor \mathsf v h_N/\mathcal{L}_N^{\frac{1}{2d}}\rfloor}
    \le (4\sqrt{\alphav}\mathcal{L}_N+1)^d \cdot \exp\biggl\{- c \frac{\mathsf v h_N}{\big(\mathcal{L}_{N}^{1/(4d)}\big)^{2d}}\biggr\}.
\end{align*}
Using~$\mathcal{L}_N = \lfloor \sqrt{\sfv} \rfloor \alphav^{\smash{N^2}}$ and~$h_N \ge h_N'/2 \ge h_0 \alphav^{\smash{N^2}}/2$, and bounding~$d \vee 2 \le 2d$, the r.h.s. is smaller than 
\[
(4 \lfloor\sqrt{\sfv}\rfloor \alphav^{N^2+1/2}+1)^d\cdot \exp \biggl\{ -c  \frac{\sfv h_0  \alphav^{N^2}/2}{ (\lfloor \sqrt{\sfv} \rfloor \alphav^{N^2})^{1/2} } \biggr\}.
\]
When~$\sfv$ is large (uniformly over~$N$), the r.h.s. is smaller than~$\exp\{-\sqrt{\sfv} \cdot \alphav^{\smash{N^2/2}}\}$, which in turn is smaller than~$\exp\{-\alphav^{1+N^2/2}\}$, since~$\alphav \ll \sqrt{\sfv}$.
Finally, using Lemma~\ref{lem_first_rw_discr}, we bound 
\begin{align}
\mathrm{discr}^{\mathrm{ip}}(\ell_{\Theta_N}, L_{\Theta_N}, \mathsf v h_N)
&\le 16ed^3 \sfv h_N {(4\alphav^{\frac{1}{2}}\mathcal{L}_N + 1)}^{d-1} \!\!\!\!\! \exp \Bigl\{-\alphav^{\frac{1}{2}}\mathcal{L}_N \cdot \log \Bigl( 1\!+\! \frac{\alphav^{\smash{\frac{1}{2}}\phantom{|}}\mathcal{L}_N}{2\sfv h_N}\Bigr) \Bigr\}.
\label{eq_last_discr}
\end{align}
Recalling that $h_N \le 2h_0 \alphav^{\smash{N^2}}$, we bound
\[
\frac{\sqrt{\alphav}\mathcal{L}_N}{2\sfv h_N} 
\ge \frac{\lfloor \sqrt{\sfv} \rfloor \alphav^{N^2}}{4\sfv h_0\alphav^{N^2}} > \frac{1}{\sfv^{1/4}},
\]
and then,
\[
\sqrt{\alphav}\mathcal{L}_N \cdot \log \left( 1+ \frac{\sqrt{\alphav}\mathcal{L}_N}{2\sfv h_N}\right) \ge \lfloor \sqrt{\sfv} \rfloor \alphav^{N^2+1/2} \cdot \frac{1}{2\sfv^{1/4}} > \alphav^{N^2+1/2}.
\]
Using this, it is now easy to see that the r.h.s. of~\eqref{eq_last_discr} is smaller than~$\exp\{-\alphav^{(N^2+1)/2}\}$.

\bigskip
This concludes the treatment of all error terms. Going back to~\eqref{eq_final_with_a_and_E}, we have proved that
\begin{align*}
&\mathbb P(\mathcal Q_N(m,n) \text{ and }\mathcal Q_N(m',n') \text{ are both bad for }(\zeta_t)) \\
&\le 2\exp\{- \alphav^{(N^2+1)/8}\} + \exp\{-\alphav^{N^2+1/8}\} + \exp\{-\alphav^{N^2/2+1}\} + \exp\{-\alphav^{(N^2+1)/2}\}\\
&\le 3\exp\{- \alphav^{(N^2+1)/8}\}.\qedhere
\end{align*}
\end{proof}

\appendix
\section{Stochastic domination for interchange process}
\label{sec:decoupling}

{In this section, we provide the details on the proof of Lemma~\ref{lem_coupling_rate_one}.
Before we delve into the proof, we first summarise a closely related result, Theorem~1.5 in~\cite{BT}, which is stated for the exclusion process on $\ZZ$. We briefly describe this result as follows:

Consider two well-separated space-time boxes $B_1$ and $B_2$, meaning that their distance
$\dist(B_1,B_2)$ is comparable to their perimeters $\per(B_1)$ and $\per(B_2)$:
\begin{equation*}
    \dist(B_1, B_2) \ge 6 (\per(B_1) + \per(B_2)) + C_1,
\end{equation*}
where $C_1>0$ is a universal constant. 
Then, for any pair of non-decreasing functions $f_1, f_2: \{0,1\}^{\ZZ \times \RR} \to [0,1]$ supported on $B_1$ and $B_2$, respectively, and every $p < p' \in [0,1]$, we have: 
\begin{equation}
\label{e:decoulplin_bt}
\EE_{\pi_p} [f_1 f_2]
    \le \EE_{\pi_{p'}} [f_1]\cdot \EE_{\pi_{p'}} [f_2] + c_1 \dist(B_1, B_2)^2
        \exp \bigl\{ - c_1^{-1} (p'-p)^2 \dist(B_1, B_2)^{1/4} \bigr\},
\end{equation}
where $c_1>0$ is a universal constant.
As explained in the introduction, \eqref{e:decoulplin_bt} features a technique known as \textit{sprinkling}, which helps to improve the decoupling bound at the cost of slightly modifying the density in the measures on both sides of the inequality.

Our Lemma~\ref{lem_coupling_rate_one} provides an improvement on \cite[Theorem~1.5]{BT}. 
The strategies used for proving both rely on a coupling between two processes started with slightly different densities within a given box. 
The coupling is carefully designed to ensure that outside events of very small probability, after a sufficiently long time, each particle in the process with lower density is coupled with a corresponding particle in the higher density process.
% At this stage, the difference between considering the exclusion process or the interchange process is negligible. Another relatively minor difference is that Lemma~\ref{lem_coupling_rate_one} holds for any dimension $d \ge 1$, whereas Theorem 1.5 in \cite{BT} is restricted to $d = 1$. \marginpar{\tiny E: Aqui acima estamos falando mais do Lema 2.8, quando se fala da comparação com \cite{BT} e de duas densidades. No segundo parágrafo é que focamos no Lema 2.7}

%However, there are important distinctions that we emphasize now. 
In comparison with \cite[Theorem~1.5]{BT}, besides dealing with any dimension $d \ge 1$,  the main innovation of Lemma~\ref{lem_coupling_rate_one} is that it provides a disintegrated version of the coupling: it estimates the coupling probability for any two starting configurations. 
This feature is essential for our arguments, as we later need to perform couplings that do not start from a Bernoulli product measure on $\ZZ^d$ (e.g., in Proposition~\ref{prop_estrela_vermelha_new}).
This is similar to the coupling present in \cite{kious2024sharp} for the exclusion process in $\ZZ$.
} % end color blue

\begin{proof}[Proof of Lemma~\ref{lem_coupling_rate_one}]
Given the starting configurations for $\xi$ and $\xi'$, we wish to build a coupling
of the processes $(\xi_s)$ and $(\xi'_s)$ so that, outside an event whose
probability we are able to bound, we have $\xi'_s(x) \ge \xi_s(x)$ for
every $(x,s) \in B = B_0(L/4) \times [t,T]$, see Figure~\ref{fig:decoupling}
for an illustration.

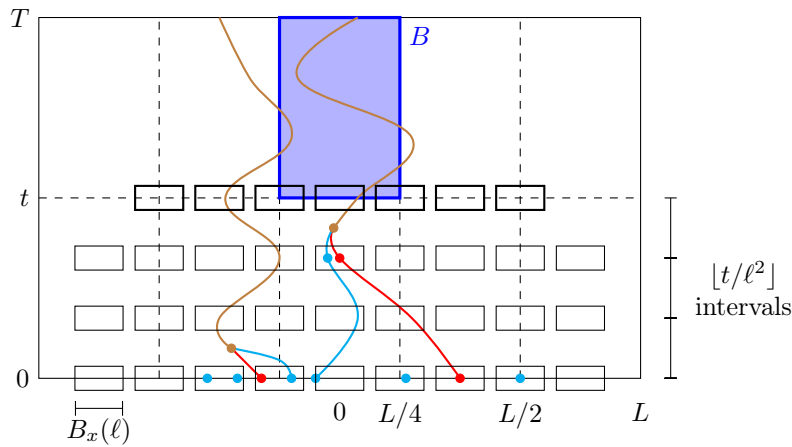
\begin{figure}
\centering
\begin{tikzpicture}[scale=.8]
    % surrounding region
    \draw (-5,3) rectangle (5,-3);
    % helping lines
    \foreach \x in {-3,-1,1,3} \draw[dashed] (\x,-3) -- (\x,3);
    \draw[dashed] (-5,0) -- (5,0);
    % labels
    \draw ( 0,-3) node[below=2mm] {$0$};
    \draw ( 1,-3) node[below=2mm] {$L/4$};
    \draw ( 3,-3) node[below=2mm] {$L/2$};
    \draw ( 5,-3) node[below=2mm] {$L$};

    \draw (-5,-3) node[left] {$0$};
    \draw (-5, 0) node[left] {$t$};
    \draw (-5, 3) node[left] {$T$};

    % coupling region B
    \filldraw[very thick, draw=blue, fill=blue, fill opacity = 0.3]
    (-1,3) rectangle (1,0);
    \draw[blue] (1,3) node[below right] {$B$};

    % initial particles
    \foreach \x in {-2.2,-1.7, -.8, -.4, 1.1, 3} \filldraw[cyan] (\x,-3) circle (.7mm);
    \foreach \x in {-1.3, 2} \filldraw[red] (\x,-3) circle (.7mm);

    % example meeting 1
    \draw[red, thick, tension=.7] plot [smooth] coordinates
    {(-1.3,-3) (-1.8,-2.5)};
    \draw[brown, thick, tension=.7] plot [smooth] coordinates
    {(-1.8,-2.5) (-2,-2) (-1,-1) (-1.9,0) (-.8,1) (-1.5,2) (-2,3)};
    \draw[cyan, thick, tension=.7] plot [smooth] coordinates
    {(-.8,-3)  (-1,-2.7) (-1.8,-2.5)};
    % meeting time
    \filldraw[brown] (-1.8,-2.5) circle (.7mm);

    % example meeting 2
    \draw[red, thick, tension=.7] plot [smooth] coordinates
    {(2,-3) (1.2,-2) (0,-1) (-.1,-.5)};
    \draw[cyan, thick, tension=.7] plot [smooth] coordinates
    {(-.4,-3) (0.3,-2) (-.2,-1) (-.1,-.5)};
    \draw[brown, thick, tension=.7] plot [smooth] coordinates
    {(-.1,-.5) (0.3,0) (1.2,1) (-.7,2) (.3,3)};
    % close particles
    \filldraw[red] (0,-1) circle (.7mm);
    \filldraw[cyan] (-.2,-1) circle (.7mm);
    % meeting time
    \filldraw[brown] (-.1,-.5) circle (.7mm);
    {(2,-3) (1.2,-2) (0,-1) (-.1,-.5) (0.3,0) (1.2,1) (-.7,2) (.3,3)};

    % repairing times
    % boxes
    \foreach \x in {-4,...,4} {
        \foreach \y in {-3,...,-1} {
            \draw (\x,\y) ++(-.4,-.2) rectangle ++(.8,.4);
            \draw (5.4,\y) -- ++(.2,0);
        }
    }
    \foreach \x in {-3,...,3}
    \draw[thick] (\x,0) ++(-.4,-.2) rectangle ++(.8,.4);

    % intervals
    \draw (5.5,-3) -- ++(0,3) node[midway,right=2mm, align=center]
    {$\lfloor t/\ell^2 \rfloor$\\ intervals};
    \foreach \y in {-3,...,0} \draw (5.4,\y) -- ++(.2,0);

    \draw[|-|] (-4.4,-3.5) -- ++(.8,0) node[midway, below] {$B_x(\ell)$};

\end{tikzpicture}
\caption{Space-time regions for the coupling in
    Lemma~\ref{lem_coupling_rate_one}, which ensures $\xi'_s(x) \ge \xi_s(x)$
    for all $(x,s) \in B$. Intuitively, the coupling works when
    all the particles passing through $B$ remain nearby on interval $[0,T]$ (controlled by
    $\mathrm{discr}^{\mathrm{ip}}$), and $\xi'$ particles (cyan) are more frequent (in
    a precise way) than $\xi$ particles (red) in $B_0(L)$ for a sufficiently long time $t$
    (controlled by $g^{\uparrow} + g^{\downarrow}$), which gives enough time
    for every $\xi$ particle to couple with a $\xi'$ particle (controlled by
    $|B_0(L/2)| (1 - \mathrm{meet}(\ell))^{\lfloor t/\ell^2 \rfloor}$).
    }
\label{fig:decoupling}
\end{figure}

\medskip
\noindent
\textbf{Pairing configurations.} We can regard $\xi$ and $\xi'$ as subsets of $\ZZ^d$. 
    Fix a collection of particles $Z \subset \xi$ and assume that $m : Z \to \xi'$ is an injective function, i.e., $m$
    associates to each particle $z \in Z$ a corresponding particle $m(z) \in \xi'$. 
    An important idea introduced in~\cite{BT} is a coupling that aims at matching
    $z$ to its pair $m(z)$. 
    If $z = m(z)$, particle $z$ is considered
    \textit{matched} from the very beginning. As the process evolves, paired particles
    that started apart become matched once they meet at a later time, and, from that time on, they will move together. 

\medskip
\noindent \textbf{Coupled evolution.}
Let
\begin{equation*}
\cJ^{i} = (\cJ^{i}_{\{x,y\}}:\, \text{$\{x,y\}$ is an edge of $\ZZ^d$})
\end{equation*}
with $i=1,2$ be two independent collections of independent Poisson point processes $\cJ^{i}_{\{x,y\}}$ on
$[0,\infty)$ with intensity $1$.
Starting from $\xi_0 = \xi$ and $\xi'_0 = \xi'$, we use $\cJ^1$ and
$\cJ^2$ to define a coupled time evolution for the pair $(\xi_s, \xi'_s)$:
\begin{itemize}
\item[(i)]
    $(\xi'_s)$ simply uses the graphical representation provided by $\cJ^2$ as in Definition \ref{def_interchange_flow}.
\item[(ii)]
    The evolution of 
    $(\xi_s)$ is slightly more subtle since it is determined by both $\cJ^1$ and $\cJ^2$ together with $m$ as follows.
    For every edge $\{x,y\}$ we use the marks in $\cJ^1_{\{x,y\}}$ when  neither $x$ nor $y$ contains 
    matched particles, and use the marks in $\cJ^2_{\{x,y\}}$ if either $x$ or $y$ contains matched particles. 
    As in \cite[Claim 3.5]{BT}, one can verify that the resulting process $(\xi_s)$ is distributed as an interchange process started from $\xi$. We denote its associated interchange flow by $\Phi$.
\end{itemize}

\medskip
\noindent
\textbf{Refreshing the pairing functions.}
Under the coupled dynamics, the distance between two paired particles follows the law of a continuous-time symmetric simple random walk on $\ZZ^d$ with jump rate $2$.
Therefore, in dimensions $d =1,2$ every pair will eventually match with probability one, but for $d \ge 3$ such a pair might never match.
Moreover, in any dimension, the matching times are heavy-tailed random variables.
Another idea from~\cite{BT}
that helps to improve the matching procedure is to only pair particles
located within a maximal distance $\ell$ and to reset the pairing function after time intervals of length approximately $\ell^2$.

We discuss the procedure further.
Fix $\cB = \{B_x(\ell)\}$ be a finite collection of disjoint
boxes of radius $\ell$ that covers $B_0(L-2\ell)$. 
By
Definition~\ref{defi:meet} and Lemma~\ref{lem_meet}, two
paired particles inside some $B_x(\ell) \in \cB$ meet before a time of order $\ell^2$ with reasonable probability.

We shall say that $(\xi, \xi')$
is a \textit{good pair of configurations} if there exists a deterministic
pairing function $m_1 : \xi \cap B_0(L-2\ell) \to \xi'$ such that for every $z
\in \xi \cap B_x(\ell)$ we have $m_1(z) \in \xi' \cap B_x(\ell)$.
Whenever we start with a good pair of configurations $(\xi_0, \xi'_0)$ at time $t=0$, we will perform the coupling with
such a pairing function $m_0$ for a time interval of length $\ell^2$.
Assuming that we get a pair $(\xi_{\ell^2}, \xi^{'}_{\ell^2})$ that is once again good, we can repeat the construction using a (possibly different) pairing function $m_1$ during the time interval $[\ell^2, 2\ell^2]$.
We iterate the procedure at times $j\ell^2$, partitioning the interval $[0,t]$ into intervals of length at least $\ell^2$:
\begin{equation}
\label{eq:intervals_matching}
[0,t]
    = \Bigl(\bigcup_{i=1}^{\lfloor t/\ell^2 \rfloor - 1}
        \bigl[(i-1)\ell^2, i \ell^2\bigr)\Bigr)
        \cup \bigl[(\lfloor t/\ell^2 \rfloor-1)\ell^2, t\bigr],
\end{equation}
the construction above produces a coupling of $(\xi_s, \xi'_s)$ started from
$(\xi, \xi')$ that holds in the interval $[0,t]$, provided the event
\begin{equation*}
A_1
    := \{\text{$(\xi_s, \xi'_s)$ are good pairs for $s = 0, \ell^2, \ldots,
        \lfloor t/\ell^2 \rfloor \cdot \ell^2$ }\}
\end{equation*}
occurs.
It is clear from Definition~\ref{def_gs} that 
\begin{equation}
    \PP(A_1^c) \le
g^{\uparrow}(\ell,L,t,p,\xi) + g^{\downarrow}(\ell,L,t,p,\xi').
\end{equation}

\medskip
\noindent
\textbf{Stochastic domination on $B$.}
On the event $A_1$, the coupling of $(\xi_s, \xi_s')$ during interval $[0,t]$
is well-defined and we would like to ensure that $\xi'_s(x) \ge \xi_s(x)$ for
every $(x,s) \in B = B_0(L/4) \times [t,T]$. Consider the event
\begin{equation*}
A_2
    := \{\text{for every } x \in \partial B_0(L/2), \text{ and every }
    0 \le s < s' \le T,\; \Phi(x,s,s') \notin \partial B_0(L/4)\}.
\end{equation*}

Recalling Definition~\ref{defi:discr_ip}, one can show that
$\PP(A_2^c) \le \mathrm{discr}^{\mathrm{ip}}(L/4, L/2, T)$. 
Moreover, on $A_1 \cap A_2$, every $\xi$ particle that touches $B$ must have stayed inside $B_0(L/2) \times [0,T]$.
Therefore, such a particle had many attempts to match with a corresponding $\xi'$ particle  until $t$.
On $\xi_t \cap B_0(L/2)$ there are at most $|B_0(L/2)|$ particles and if any of them is not matched, then it has failed to match in every interval of the partition~\eqref{eq:intervals_matching}.
Hence, denoting
\begin{equation*}
    A_3 := \{\text{every $\xi$ particle that touches $B$ was matched by time $t$ and passed on $B_0(\tfrac{L}{2}) \times \{t\}$}\},
\end{equation*}
then $\PP(A_3^c \cap A_1 \cap A_2) \le |B_0(L/2)| \cdot
(1 - \mathrm{meet}(\ell))^{\lfloor t/\ell^2 \rfloor}$.

\medskip
Summing it up, on event $A_1 \cap A_2 \cap A_3$ the desired coupling
holds, and by construction
\begin{equation*}
\PP(\cup_{i=1}^3 A_i^c)
    \le g^{\uparrow}(\ell,L,t,p,\xi) + g^{\downarrow}(\ell,L,t,p,\xi') + 
    \mathrm{err}_{\mathrm{coup}},
\end{equation*}
where $\mathrm{err}_{\mathrm{coup}} = 
    |B_0(L/2)|\cdot (1 - \mathrm{meet}(\ell))^{\lfloor t/\ell^2 \rfloor}
    + \mathrm{discr}^{\mathrm{ip}}(L/4, L/2, T).$
\end{proof}

\section{Proofs of estimates for the interchange process}\label{appendix_rw_ip}

{The following is proved in the beginning of Section 6.7 in~\cite{DP}. Although the proof therein is written for $d=1$, the extension to $d \geq 1$ is easy}.

%{\color{blue} U: Comment that the result there is done for $d=1$ but is easily generalized?}

\begin{lemma}\label{lem_distance_two_interchange}
Letting~$\Phi$ be an interchange flow  with rate~$\mathsf v = 1$, for any~$\delta > 0$, we have
\begin{equation*}
    \sup_{\substack{x,y \in \mathbb Z^d \\ x \neq y}} \;\sum_{w,z \in\mathbb Z^d } |\mathbb P(\Phi(x,0,t)=w,\; \Phi(y,0,t) = z) - \mathbb P(\Phi(x,0,t)=w) \cdot \mathbb P(\Phi(y,0,t) = z)| \xrightarrow{t \to \infty} 0.
\end{equation*}
\end{lemma}

\begin{lemma}
\label{lem_av_rw}
	Let~$p \in [0,1]$. For each~$t > 0$, let~$A_t$ be a subset of~$\mathbb Z^d$; assume that these sets satisfy the following property: for any~$K > 0$ and any~$\delta > 0$, there exists~$t_0 > 0$ such that for all~$t \ge t_0$,
	\begin{equation}\label{eq_counting_At}
		\frac{|A_t \cap B_x\big(\delta \sqrt{t}\big)|}{|B_x\big(\delta \sqrt{t}\big)|} \le  p  \quad \text{for all } x \in B_0\big(K\sqrt{t}\big).
	\end{equation}
Then, letting~$(X_t)_{t \ge 0}$ denote a random walk on~$\mathbb Z^d$ with transition function as in~\eqref{e:ctrw}, we have
\begin{equation*}
	\limsup_{t \to \infty}\mathbb P(X_t \in A_t) \le p.
\end{equation*}
\end{lemma}
\begin{proof}
	Fix~$\varepsilon > 0$. Choose~$K$ large enough that, letting~$Z \sim \mathcal N(0,\mathrm{Id})$ be a standard Gaussian in~$\mathbb R^d$, we have~$\mathbb P(Z \in [-K,K)^d) >1- \varepsilon/2$. By the Central Limit Theorem, if~$t$ is large enough we have~$\mathbb P\big(X_t \in \big[-K\sqrt{t},K\sqrt{t}\big)^d\big) > 1- \varepsilon$. We can then bound
	\begin{align*}
		\mathbb P(X_t \in A_t) \le  \varepsilon+\mathbb P\big(X_t \in A_t \cap \big[-K\sqrt{t},K\sqrt{t}\big)^d\big) 
	\end{align*}
	when~$t$ is large enough. Letting~$f:\mathbb R^d \to [0,\infty)$ be the probability density function of~$Z$, the Local Central Limit Theorem gives
	\begin{equation*}\sup_{x \in \mathbb Z^d}\left|\mathbb P(X_t = x) - \frac{1}{t^{d/2}}\cdot f\left(\frac{1}{\sqrt{t}}x\right)\right| = o\left(\frac{1}{t^{d/2}}\right);\end{equation*}
		combining this with the above bound, for~$t$ large enough (depending on~$K$) we have
	\begin{align}\label{eq_nicer_inside}
		\mathbb P(X_t \in A_t) \le 2\varepsilon+\frac{1}{t^{d/2}}\sum_{x \in A_t \cap [-K\sqrt{t},K\sqrt{t})^d} f\left(\frac{1}{\sqrt{t}}x\right).
	\end{align}

	Let~$\delta > 0$ be small, to be chosen later (not depending on~$t$), with~$K/\delta \in \mathbb N$. We write
	\begin{equation*}
    \Lambda(K,\delta):=\{-K,-K+\delta,-K+2\delta,\ldots, K-\delta\}^d,
    \ \text{so that}\
    [-K,K)^d\!\! =\, \bigcup_{\mathclap{q\in \Lambda(K,\delta)}}\ (q+[0,\delta)^d).
    \end{equation*}
    For each~$q \in \Lambda(K,\delta)$, we bound
	\begin{align*}
		\sum_{x \in A_t \cap (q\sqrt{t} + [0,\delta\sqrt{t})^d)}f\left(\frac{1}{\sqrt{t}}x\right) &\le \max_{u \in (q+[0,\delta)^d)} f(u) \cdot |A_t \cap(q\sqrt{t} + [0,\delta\sqrt{t})^d)|\\[-2mm]
		&\overset{\mathclap{\eqref{eq_counting_At}}}{\le} \max_{u \in (q+[0,\delta)^d)} f(u) \cdot p \cdot |\mathbb Z^d \cap (q\sqrt{t} + [0,\delta\sqrt{t})^d)|\\
		&\le \max_{u \in (q+[0,\delta)^d)} f(u) \cdot p \cdot (\delta\sqrt{t}+1)^d.
	\end{align*}
	By bounding~$(\delta \sqrt{t}+1)^d \le (1+\varepsilon)\delta^d t^{d/2}$ for~$t$ large and combining this with~\eqref{eq_nicer_inside}, we have
	\begin{align*}
		\limsup_{t\to \infty}\mathbb P(X_t \in A_t) &\le 2\varepsilon+p(1+\varepsilon) \delta^d \sum_{q \in \Lambda(K,\delta)}\max_{u \in (q+[0,\delta)^d)} f(u).
		 	\end{align*}
	By taking~$\delta$ small (independently of~$t$), the r.h.s. above approaches
\[
2\varepsilon+p(1+\varepsilon) \int_{[-K,K)^d} f(u)\;\mathrm{d}u \le 2\varepsilon + p (1+\varepsilon).
\]
	Since~$\varepsilon$ is arbitrary, the desired bound follows.
\end{proof}

\begin{proof}[Proof of Lemma~\ref{lem_av_rw_bal}]
Fix~$p$,~$p'$ and~$\xi_0$ as in the statement, and let~$\mathsf v$ be large, to be chosen later. 
Also fix~$u$,~$T$ and~$\mathsf e$ as in the statement. We have
\begin{align*}
    \mathbb P(\mathcal Y \in A) = \sum_{v \in A} \mathbb P(\mathcal Y = v) &=  \sum_{v \in A}\;\sum_{x \in \mathbb Z^d}  \mathbb P(\Phi(u,0,T)=x,\;\Phi(v,0,T) = x+\mathsf e)\\[.2cm]
    &=\sum_{v \in A}\;\sum_{x \in \mathbb Z^d}\mathbb P(\Phi(x,0,T) = u,\; \Phi(x+\mathsf e,0,T) = v),
\end{align*}
where the second equality follows from invariance of the law of the interchange flow under time reversal. Further using invariance of this law under spatial shifts, as well as a change of variable ($y:=u-x$), the above equals
\begin{align*}
    \sum_{y \in \mathbb Z^d}\;\sum_{v \in A} \mathbb P(\Phi(0,0,T) = y,\; \Phi(\mathsf e,0,T) = y+ v-u).
\end{align*}

We introduce the intermediate time~$t := T-\mathsf v^{-3/4}$ and note that the above equals
\begin{equation}\label{eq_more_important_expr}
    \sum_{w,z,y \in \mathbb Z^d}\sum_{v \in A} \mathbb P(\Phi(0,0,t) = w,\;\Phi(\mathsf e,0,t) = z) \cdot \mathbb P(\Phi(w,t,T) = y,\; \Phi(z,t,T) = y+v-u).
\end{equation}

Let us abbreviate $B:= B_0(\tfrac18L_0)$ and fix~$\varepsilon > 0$. When~$\mathsf v$ is large enough, we have
\begin{align*}
	\mathbb P(B \supset \{\Phi(0,0,t),\; \Phi(0,0,T),\; \Phi(\mathsf e,0,t),\; \Phi(\mathsf e,0,T)\}) > 1- \varepsilon.
\end{align*}

Hence,~\eqref{eq_more_important_expr} is smaller than
\begin{equation*}
\varepsilon + \sum_{\mathclap{\substack{w,z \in B;\\ w\neq z}}} \; \mathbb P(\Phi(0,0,t) = w,\;\Phi(\mathsf e,0,t) = z) \times \sum_{y \in B}\; \sum_{v \in A} \mathbb P(\Phi(w,t,T) = y,\; \Phi(z,t,T) = y+v-u),
\end{equation*}
which,  by Lemma~\ref{lem_distance_two_interchange} and  for~$\mathsf v$ large  is smaller than
\begin{equation}\label{eq_another_imp}
    2\varepsilon + \sum_{\mathclap{\substack{w,z \in B;\\ z\neq w}}} \; \mathbb P(\Phi(0,0,t) = w,\;\Phi(\mathsf e,0,t) = z) \times \sum_{y \in B} \mathbb P(\Phi(w,t,T) = y) \; \sum_{v \in A} \mathbb P(\Phi(z,t,T) = y+v-u).
\end{equation}
We write
\[
\sum_{v \in A} \mathbb P(\Phi(z,t,T) = y+v-u) = \mathbb P(\Phi(z-y+u,0,T-t) \in A).
\]
Recall that we have fixed~$u \in B_0(\tfrac12 L_0)$. Fix a choice of~$z \in B$ and $y \in B$. By the triangle inequality, we have~$z-y+u \in B_0(\frac34L_0)$. In particular,~$B_{z-y+u}(\tfrac14 L_0) \subseteq B_0(L_0)$.
Then, by the assumption~\eqref{eq_my_density_assumption},
\begin{equation*}
    \frac{|A \cap B_x(\mathsf v^{1/10})|}{|B_x(\mathsf v^{1/10})|} \le p \quad \text{ for all }x \in B_{z-y+u}(\tfrac14 L_0).
\end{equation*}
We also have~$T-t = \mathsf v^{-3/4}$, so~$ \sqrt{\mathsf v \cdot(T-t)} =\mathsf v^{1/8}$, which is much larger than~$\mathsf v^{1/10}$ and much smaller than~$\tfrac14 L_0 = \tfrac14 \sqrt{\sfv}\log^4(\sfv)$. It is then easy to see that the above implies that, fixing~$K$ and~$\delta$, and taking~$\mathsf v$ large enough (depending on~$K$ and~$\delta$), we have
\begin{equation*}
    \frac{|A \cap B_x\big(\delta \cdot \sqrt{\mathsf v \cdot (T-t)} \big)|}{|B_x\big(\delta \cdot \sqrt{\mathsf v \cdot (T-t)} \big)|} \le p+\varepsilon \quad \text{ for all }x \in B_{z-y+u}\big(K \cdot \sqrt{\mathsf v \cdot (T-t)}\big).
\end{equation*}
Then, Lemma~\ref{lem_av_rw} (with time multiplied by~$\mathsf v$) implies that 
\[\mathbb P(\Phi(z-y+u,0,T-t) \in A) < p+2\varepsilon\]
if~$\mathsf v$ is large enough. 
From this bound, the expression in~\eqref{eq_another_imp} is smaller than~$p+ 4\varepsilon$.
\end{proof}

\section*{Acknowledgments}
The authors thank Andreas Kyprianou for the suggestion of using reference~\cite{Biggins} to prove Lemma~\ref{lem_bbm}.
M.E.V. is partially supported by CNPq grant 310734/2021-5 and by FAPERJ grant
E-26/200.442/\break 2023.
The research of M.H.\ is partially supported by FAPEMIG grant APQ-01214-21,
CNPq grant 312566/2023-9  and CNPq grant 406001/2021-9.
Moreover, D.U., D.V. and M.E.V. were all supported by CNPq \textit{Universal} grant 403423/2023-6.

\bibliographystyle{plain}
\bibliography{bibliography}       % Bibliography file (usually '*.bib')
\end{document}